\documentclass[letterpaper,11pt]{article}
\usepackage[margin=1in]{geometry}  %margin of page
\usepackage{setspace}  %double-space
\usepackage{natbib}
 \bibpunct[, ]{(}{)}{,}{a}{}{,}%

\usepackage{amsfonts}
\usepackage{amsthm}
\usepackage{mathrsfs}
\usepackage{amsmath}
\usepackage{amssymb}
\usepackage{latexsym}
\usepackage{indentfirst}
\usepackage{footmisc}
\usepackage{algorithm}
\usepackage{algorithmic}
\usepackage{float}
\usepackage{graphicx}
\usepackage{epstopdf}
\usepackage{longtable}
\usepackage{enumerate}
\usepackage{enumitem}
\usepackage{authblk}
\usepackage{bbm}
\usepackage{booktabs}
\usepackage{url}

\usepackage{multirow}
\usepackage{bm}
\usepackage{array}
\usepackage{subfigure}
\usepackage{caption}
\usepackage{color}

\usepackage{pifont}  %to draw pentagon and asterisk

\def\Vcal{\mathcal{V}}

\def\Acal{\mathcal{A}}

\def\Dcal{\mathcal{D}}

\def\Ecal{\mathcal{E}}

\def\Scal{\mathcal{S}}

\def\Ical{\mathcal{I}}

\def\Ucal{\mathcal{U}}
\def\Vcal{\mathcal{V}}

\def\Hcal{\mathcal{H}}
\def\Wcal{\mathcal{W}}

\def\Pb{ \mathbb{P} }
\def\Eb{ \mathbb{E} }
\def\Rb{ \mathbb{R} }
\def\balpha{\boldsymbol{\alpha}}

\def\1{\mathbb{I}}

\def\ocbau{\mathcal{OCBA}^\mathcal{U}}

\newtheorem{lemma}{LEMMA}

\newtheorem{theorem}{THEOREM}
\newtheorem{proposition}{PROPOSITION}
\newtheorem{example}{EXAMPLE}

\allowdisplaybreaks

\title{Optimal Simulation Budget Allocation Under Unknown Sampling Variance}

\author[a]{Jianzhong Du}
\author[b]{Ilya O. Ryzhov}
\author[c]{Siyang Gao}
\affil[a]{School of Management, University of Science and Technology of China, Anhui, China, dujianzhong@ustc.edu.cn}
\affil[b]{Robert H. Smith School of Business, College Park, MD, USA, iryzhov@umd.edu}
\affil[c]{Department of Systems Engineering and Department of Data Science, City University of Hong Kong, Hong Kong, China, siyangao@cityu.edu.hk \footnote{Corresponding author: Siyang Gao}}

\date{\vspace{-8ex}}

\begin{document}
%%%%%%%%%%%%%%%%

\maketitle

\begin{abstract}
The ranking and selection problem is a popular framework in the simulation literature for studying optimal information collection. We study a version of this problem in which the simulation output for each design is normally distributed with both its mean and variance being unknown. Using a Bayesian representation of the probability of correct selection, which allows us to explicitly model uncertainty about the variance, we provide a theoretical characterization of the optimal allocation of the simulation budget. Prior work on optimal budget allocation was unable to distinguish between known and unknown sampling variance. We show the impact of this type of uncertainty on the allocation, and design new sequential procedures that can be guaranteed to learn the optimal allocation asymptotically without the need for tuning or forced exploration.
% Enter your abstract

\end{abstract}



%%%%%%%%%%%%%%%%%%%%%%%%%%%%%%%%%%%%%%%%%%%%%%%%%%%%%%%%%%%%%%%%%%%%%%

\section{Introduction}

In the ranking and selection (R\&S) problem \citep{HoNe09,hong2021review}, the goal is to select the best among a finite set of ``designs'' or ``alternatives''. The value of each design is unknown, but can be estimated from expensive simulation experiments. The goal is to divide the budget between designs in a way that maximizes our chances of identifying the best after all experiments have concluded. R\&S has applications in, e.g., plant biology \citep{HuMc16}, medical decision-making \citep{du2024contextual}, and materials science \citep{kerfonta2024sequential}, but it is also popular in the simulation community as a canonical framework for the fundamental study of information collection. Any problem in which information is acquired exhibits the ``exploration/exploitation'' tradeoff, in which a decision-maker must choose between a decision that appears to be good, and one whose outcome is highly uncertain but may be better than expected. R\&S has a clean mathematical formulation in which this core tradeoff is fully present.

There are many ways to categorize the numerous approaches that have been developed for R\&S. For our purposes, two distinctions are particularly important:
\begin{itemize}
	\item \textit{Fixed budget vs. fixed precision.} In fixed-precision R\&S, one continues to run experiments until some termination criterion is satisfied. For example, indifference-zone procedures \citep{kim2001fully} sequentially screen out designs whose estimated values are sufficiently poor, and continue until only one design remains. On the other hand, in fixed-budget R\&S, the total simulation budget is known ahead of time and the goal is to use the available samples efficiently. For example, expected improvement, or EI \citep{jones1998efficient}, aims to maximize the efficiency of each individual experiment by optimizing a myopic criterion.
	\item \textit{Frequentist vs. Bayesian statistics.} A critical question in R\&S is how to model our uncertainty about the values of the designs. Frequentist methods \citep{glynn2004} use point estimates, simply averaging the results of all experiments conducted on a particular design. Bayesian methods \citep{Ch06} model the unknown value of each design as a random variable whose distribution allows us to make probabilistic forecasts of how far the true value might be from the current estimate.
\end{itemize}
Any combination of these categories is possible. Indifference-zone methods \citep{kim2001fully,kim2006asymptotic} are frequentist and fixed-precision, as are more recent screening techniques by \cite{fan2016indifference}, \cite{ma2017efficient} and \cite{wang2024bonferroni}; all of these papers use dynamic stopping rules. The literature on best-arm identification is almost exclusively fixed-precision, but includes both frequentist \citep{BuMuSt11,garivier2016optimal} and Bayesian \citep{russo2020simple,QiYo25} approaches. In the simulation literature, the early work by \cite{chick2001a,chick2001b} also used Bayesian models together with dynamic termination. On the other hand, the optimal computing budget allocation (OCBA) literature \citep{chen2000,chen2008,gao2017a} is predominantly frequentist and fixed-budget, as is the closely related work by, e.g., \cite{pasupathy2015}, \cite{gao2017}, and \cite{kim2024rate} on optimal sampling laws based on large deviations theory. Lastly, the family of methods based on some variant of EI \citep{chick2010,SaNeSt14}, including the related knowledge gradient method \citep{frazier2008}, is fixed-budget and Bayesian.

This paper takes a \textit{fixed-budget} and \textit{Bayesian} view of an R\&S problem in which the simulation output is independent across designs and normally distributed, with both the means and variances of the normal distributions being unknown. The assumption of normality is classical for R\&S, and appears in the vast majority of the simulation literature. Entire methodologies, including indifference-zone, OCBA, and EI, are built around normality; while some authors have examined more general distributions (including, e.g., \citealp{russo2020simple}, \citealp{chen2023balancing}, and \citealp{bandyopadhyay2024optimal}), others have argued \citep{ShBrZe18} that normality offers a good approximation in general settings. In any case, the computational tractability afforded by normality continues to attract attention even in very recent work, such as \cite{jourdan2023dealing} and \cite{QiYo25}.

We adopt the Bayesian view because it explicitly incorporates uncertainty about the sampling variance into the statistical model. Frequentist methods handle the known-variance and unknown-variance cases in a similar way, simply plugging in a point estimate in the latter case. In the Bayesian setting, uncertain variance changes the distribution of the unknown mean. For example, EI methods that measure the potential of a design by integrating a certain value function over its posterior distribution will now integrate over the joint density of the mean and the variance, producing a completely different sampling criterion \citep{chick2010,HaRyDe16}. Thus, Bayesian approaches are particularly attractive when unknown variance is a significant concern.

Fixed-budget procedures are appealing because they tend to be less conservative: empirically, they are often able to find the best design in a smaller number of samples.\footnote[1]{Conservativeness is a well-known issue in fixed-precision procedures; see, e.g., \cite{fan2025first}.} Moreover, while they cannot guarantee a certain level of precision at the time of termination, one can derive strong asymptotic guarantees using a technique pioneered by \cite{glynn2004}. This method uses large deviations theory to characterize an allocation of the budget that is optimal, in the sense that the probability of correctly identifying the best design converges to one at the fastest possible rate. This optimal allocation cannot be implemented directly, because it depends on the unknown problem parameters, but one can estimate it and adjust sampling decisions over time to recover it asymptotically. It has been observed that both OCBA \citep{chen2011} and EI \citep{ryzhov2016convergence} are essentially approximating this optimal allocation, giving rise to a stream of literature that sought to learn it exactly, whether through another variant of EI \citep{chen2019complete}, a hypersphere approximation of the probability of correct selection (PCS) \citep{peng2018ranking,zhang2023asymptotically}, or based on the optimality conditions directly \citep{gao2017,chen2023balancing,kim2024selection}.

The large deviations approach is frequentist in nature, and cannot distinguish between known and unknown variance. \cite{russo2020simple} derived similar convergence rates in a Bayesian-inspired setting, but the technique works with only one unknown parameter. A follow-up work by \cite{qin2017improving} likewise assumed known variance. To our knowledge, \cite{jourdan2023dealing} is the only prior work in this stream to directly engage with unknown variance (in a fixed-precision context), but it implicitly assumes certain differentiability properties that (as we show in this paper) may not actually hold. Thus, the problem of optimal sampling allocations under unknown variance has remained open, particularly under fixed budget.

The present paper solves this problem.\footnote[2]{A brief summary version of this work will appear in the {\it Proceedings of the 2025 Winter Simulation Conference} \citep{du2025wsc}. The proceedings paper omits all proofs and does not discuss how to learn the optimal allocation sequentially. The complete treatment is given only in the present paper.} First, we derive the convergence rate of the PCS (or, more precisely, of its complement) for a fixed sampling allocation. As in the large deviations approach, the rate turns out to be exponential, but the large deviations technique itself cannot be used to show this, because we work with a Bayesian representation of the probability: an integral over the joint posterior density of the unknown means and variances. This density does not meet the standard conditions that are typically used to derive large deviations laws, and the rate exponent that we derive behaves very differently from the much simpler setting of known variance. The exponent can be stated as the optimal value of a certain optimization problem. In the known-variance setting, this problem is convex and solvable in closed form. In our setting, it is \textit{not} convex, and its optimal solution may be discontinuous in the allocation parameters. It is this discontinuity that was not taken into account by \cite{jourdan2023dealing}.

Second, we formulate and solve the optimal allocation problem. The formulation is similar to what is done in the large deviations approach: having derived the rate exponent, we choose an allocation that maximizes it, thus speeding up the convergence rate. In the known-variance setting, the optimality conditions can be derived using standard convex programming techniques (they are simply the KKT conditions of this maximization problem). In our setting, the discontinuity present in the rate exponent renders this approach unusable. Nonetheless, we develop a different analytical technique to show that the optimal allocation still exists and is unique, and we derive a set of conditions that characterize it. In instances where the discontinuity is not present, these conditions simplify to an analog of those seen in large deviations-based methods, but we provide a general form that has a solution in all instances.

Third, we propose a computationally efficient sequential procedure that is guaranteed to learn the optimal allocation (whether or not the discontinuity is present) asymptotically. Structurally, the algorithm uses the balancing principle of \cite{chen2023balancing}, but the analysis is substantially different because the optimality conditions are more complex and one cannot take the differentiability of certain objects for granted. The algorithm does not require any tuning (unlike the top-two method of \citealp{russo2020simple}) or forced exploration (which remains present in many recent papers, such as \citealp{bandyopadhyay2024optimal}). We also conduct numerical comparisons to demonstrate the benefits of explicitly incorporating unknown variance into the sampling allocation, and to provide intuition for why our optimality conditions improve efficiency in that setting.

Fourth, as a byproduct of our analysis, we obtain new insights into some prior results on R\&S with unknown variance. \cite{ryzhov2016convergence} observed that, under known variance, the original, unmodified EI method of \cite{jones1998efficient} behaves like OCBA, producing an approximation of the optimal allocation derived by \cite{glynn2004}. That paper conjectured that similar approximation behavior may be taking place under unknown variance. Since our paper is the first to derive the optimal allocation for the unknown-variance setting, we are able to verify this conjecture: indeed, EI makes an OCBA-like approximation of the optimal solution. Our results also answer an open question raised in \cite{jourdan2023dealing}. As mentioned, this paper adopted a fixed-precision view, but the authors speculated whether their results would be applicable to the fixed-budget setting. We show that the rate exponent that we use as our objective function is equivalent to the objective derived in their work, and therefore (modulo our correction of the discontinuity issue) the same allocation is optimal for both frameworks. This observation provides a bridge between fixed-budget and fixed-precision models.

Finally, while our paper focuses on one particular R\&S model, we believe that our analysis has broader theoretical interest, in that it starkly illustrates the difficulties that can arise when we are learning multiple unknown parameters per design rather than just one. When we state the rate exponent as the optimal value of an optimization problem, the decision variable of that problem influences the objective through both unknown parameters simultaneously. This simultaneous dependence is what causes the problem to become non-convex. It is likely that any future research involving multi-parameter uncertainty will also encounter this issue.

\section{Preliminaries}\label{sec:pre}

Section \ref{sec:formulation} defines relevant notation and formulates the Bayesian optimal computing budget allocation problem. Section \ref{sec:comparison} relates this problem to the more widely studied frequentist variant. For ease of reference, we provide a summary of the main notation used in this paper in Table \ref{tabadd:1}.

\begin{table}[htbp]
	\centering
	\small
	\renewcommand{\arraystretch}{1.05}
	\setlength{\tabcolsep}{4pt}
	\caption{Main notation used in the paper.}
	\begin{tabular}{>{\raggedright\arraybackslash}p{2.9cm}
			>{\raggedright\arraybackslash}p{5.7cm}
			>{\raggedright\arraybackslash}p{6.7cm}}
		\toprule
		Symbol & Meaning & Notes/Where used \\
		\midrule
		$k$ & Number of designs (alternatives) & Fixed positive integer \\
		$i, j$ & Design indices & $i,j \in \{1,\dots,k\}$ \\
		$i^*$ & True best design & $i^* = \arg\max_i \mu_i$, unique \\
		$n$ & Total sampling budget & $n=\sum_i N_i$ \\
		$N_i$ & Samples allocated to design $i$ & Decision variables (Section \ref{sec:formulation}) \\
		$\alpha_i$ & Sampling proportion & $\alpha_i=N_i/n$, $\sum_i\alpha_i=1$ \\
		$\boldsymbol{\alpha}$ & Vector of sampling proportions & $(\alpha_1,\dots,\alpha_k)$ \\
		$X_{i,l}$ & $l$th observation of design $i$ & $X_{i,l}\sim \mathcal{N}(\mu_i,\sigma_i^2)$, independent across $i$\\
		$\bar X_i^n$ & Sample average of design $i$ & $\bar X_i^n=\frac{1}{N_i}\sum_{l=1}^{N_i}X_{i,l}$ \\
		$(\mu_i,\sigma_i^2)$ & True mean/variance & Unknown constants \\
		$\boldsymbol{\mu}, \boldsymbol{\sigma}^2$ & Vectors of means/variances & $\boldsymbol{\mu}\in\mathbb{R}^k$, $\boldsymbol{\sigma}^2\in\mathbb{R}_+^k$ \\
		$\phi_i,\psi_i$ & Generic mean/variance variables & In posteriors/integrals \\
		$\boldsymbol{\phi},\boldsymbol{\psi}$ & Vectors of $\phi_i,\psi_i$ & Integration variables \\
		$\pi^0(\boldsymbol{\phi},\boldsymbol{\psi})$ & Prior density & General prior (Section \ref{sec:prior}) \\
		$\pi^n(\boldsymbol{\phi},\boldsymbol{\psi})$ & Posterior after $n$ samples & Bayesian update and likelihood \\
		$L^n(\phi_i,\psi_i)$ & Likelihood for design $i$ & $L^n(\phi_i,\psi_i)=\prod_{l=1}^{N_i} f(X_{i,l}\mid \phi_i,\psi_i)$ \\
		$f(\cdot\mid \phi_i,\psi_i)$ & $\mathcal{N}(\phi_i,\psi_i)$ density & \\
		$\hat{\mu}_i^n,(\hat{\sigma}_i^n)^2$ & Posterior means of $\mu_i,\sigma_i^2$ & Section \ref{sec:formulation} \\
		$i^{*,n}$ & Posterior best design & $i^{*,n}=\arg\max_i \hat{\mu}_i^n$ \\
		$PCS_B^n$ & Bayesian prob. of correct selection & Eq. \eqref{eq:pcs_bayes} \\
		$PCS_F^n$ & Frequentist prob. of correct selection & Eq. \eqref{ineq:pcs_freq} \\
		$PFS_B^n,\ PFS_F^n$ & Bayes/freq. prob. of false selection & $PFS_B^n=1-PCS_B^n$, $PFS_F^n=1-PCS_F^n$ \\
		$\mathcal{G}(\boldsymbol{\alpha})$ & LD-rate exponent for $PCS_F^n$ & Eqs. \eqref{glynn}, \eqref{eq:freqrate} \\
		$\Vcal_i(\alpha_i,\alpha_{i^*})$ & Rate (unknown var., Bayesian PCS) & Eq. \eqref{eq:vdef} \\
		$g_i(\phi_i,r)$ & Auxiliary objective & Eq. \eqref{eq:taudef} \\
		$\Wcal_i(r)$ & Rescaled objective & Eq. \eqref{eq:taudef} \\
		$r$ & Proportion ratio & $r=\alpha_i/\alpha_{i^*}$ \\
		$\phi_i^{\min}(r),\phi_i^{\max}(r)$ & Smallest/largest minimizers of $\Wcal_i$ & Lemma \ref{lem:phimono0} \\
		$\Ucal_i^{\min}(r)$ & Log-term at $\phi_i^{\min}(r)$ & $\log(1+\frac{(\mu_i-\phi_i^{\min}(r))^2}{\sigma_i^2})$; Eq. \eqref{eq:ucaldef} \\
		$\Ucal_i^{*,\min}(r)$ & Best-design log-term at $\phi_i^{\min}(r)$ & $\log(1+\frac{(\mu_{i^*}-\phi_i^{\min}(r))^2}{\sigma_{i^*}^2})$; Eq. \eqref{eq:ucalmaxdef} \\
		$\Ucal_i^{\max}(r),\Ucal_i^{*,\max}(r)$ & Analogous logs at $\phi_i^{\max}(r)$ & Defined analogously to $\Ucal_i^{\min},\Ucal_i^{*,\min}$\\
		$\Xi_i$ & Pairwise ``error" set & $\Xi_i=\{(\boldsymbol{\phi},\boldsymbol{\psi}): \phi_{i^*}\le \phi_i\}$ \\
		$H_w$ & Prior mass lower-bounded region & Eq. \eqref{ineq:prior_reg} \\
		%$\Delta$ & Small positive constant & In proofs (Appendix) \\
		$\tilde S_i^n$ & MLE standard deviation & Technical lemmas \\
		$\hat{\Vcal}_i^m$ & Plug-in estimate of $\Vcal_i$ at time $m$ & Eq. \eqref{eq:vcalm} \\
		$\hat{\Ucal}_i^{m},\,\hat{\Ucal}_i^{*,m}$ & Plug-in estimates of $\Ucal_i^{\min}(r),\Ucal_i^{*,\min}(r)$ & Section \ref{sec:algorithms} \\
		$\hat{\alpha}_i^m$ & Estimated proportion at time $m$ & $\hat{\alpha}_i^m=N_i^m/m$ \\
		$i^{*,m}$ & Estimated best design at time $m$ & $i^{*,m}=\arg\max_i \hat{\mu}_i^m$ \\
		$\hat{\Wcal}_i^m$ & Estimate of $\Wcal_i$ & $\hat{\Wcal}_i^m=\frac{2}{\hat{\alpha}_{i^{*,m}}^m}\hat{\Vcal}_i^m$ \\
		$\ocbau$ & Proposed algorithm & Algorithm \ref{alg:ocba} \\
		\bottomrule
	\end{tabular}
	\label{tabadd:1}
\end{table}

\subsection{Bayesian Formulation}\label{sec:formulation}

Let there be $k$ designs, with $\mu_i$ being the unknown true value of design $i \in \left\{1,...,k\right\}$. Let $\left\{X_{i,l}\right\}^{\infty}_{l=1}$ be a sequence of noisy observations, each of which is drawn from the distribution $\mathcal{N}\left(\mu_i,\sigma^2_i\right)$, where the variance $\sigma^2_i$ is also unknown. Two samples $X_{i,l}$ and $X_{j,l'}$ are independent if $i \neq j$ or $l \neq l'$. We assume that the best design $i^* = \arg\max_i \mu_i$ is unique.

The decision-maker approaches the problem from a Bayesian perspective, modeling the unknown means and variances $\left(\boldsymbol{\mu},\boldsymbol{\sigma}^2\right)$ as a random vector with the joint prior density $\pi^{0}(\boldsymbol{\phi},\boldsymbol{\psi})$, where $\boldsymbol{\phi} \in \Rb^k$ and $\boldsymbol{\psi} \in \Rb^k_{+}$ denote possible values of $\boldsymbol{\mu}$ and $\boldsymbol{\sigma}^2$. Suppose that $N_i$ samples are collected for each design $i$, with $n = \sum_i N_i$ being the total budget. Let $f\left(\cdot\mid \phi_i,\psi_i\right)$ denote the $\mathcal{N}\left(\phi_i,\psi_i\right)$ density. Then, the posterior joint density $\pi^{n}(\boldsymbol{\phi},\boldsymbol{\psi})$ can be calculated by applying Bayes' rule to the prior $\pi^{0}\left(\boldsymbol{\phi},\boldsymbol{\psi}\right)$ and the joint conditional likelihood $L^n\left(\phi_i,\psi_i\right) = \prod^k_{i=1} \prod^{N_i}_{l=1} f\left(X_{i,j}\mid \phi_i,\psi_i\right)$ of the observations.

Let
\begin{equation*}
	\hat{\mu}^n_i = \int_{\Rb^{k}} \int_{\Rb^{k}_+} \phi_i \pi^n(\boldsymbol{\phi},\boldsymbol{\psi}) d \boldsymbol{\psi} d \boldsymbol{\phi}, \qquad \left(\hat{\sigma}^{n}_i\right)^2 = \int_{\Rb^{k}} \int_{\Rb^{k}_+} \psi_i \pi^n(\boldsymbol{\phi},\boldsymbol{\psi}) d \boldsymbol{\psi} d \boldsymbol{\phi}
\end{equation*}
represent the posterior beliefs about $\mu_i$ and $\sigma^2_i$, respectively, and denote by $i^{*,n} = \arg\max_i \hat{\mu}^n_i$ the design believed to be the best. The Bayesian PCS is defined as
\begin{align}
	PCS^n_{B} \triangleq&  \int_{\Rb^{k}} \int_{\Rb^{k}_+} \1\{ \cap_{i \ne i^{*,n}} \{\phi_{i^{*,n}} > \phi_i\} \}  \pi^n (\boldsymbol{\phi},\boldsymbol{\psi}) d \boldsymbol{\psi} d \boldsymbol{\phi}, \label{eq:pcs_bayes}
\end{align}
where $\1\{\cdot\}$ is the indicator function. The quantity $PCS^n_B$ represents the decision-maker's belief (according to the posterior after $n$ total samples) about how likely the estimated best design $i^{*,n}$ is to actually be the best.

The \textit{optimal computing budget allocation} (OCBA) problem, for this Bayesian setting, can be formulated as
\begin{align}\label{eq:bayesPCS}
	\max_{N_1,\dots,N_k} PCS^n_{B} \quad \mbox{s.t. } \sum_{i=1}^k N_i = n, N_i \in \mathbb{Z}_+ \text{ for }i=1,2,...,k.
\end{align}
In words, the decision-maker divides the samples between designs ahead of time to maximize the posterior probability that the design with the highest posterior mean value also has the highest true value. The OCBA literature \citep{chen2000,chen2011} typically recasts the decision variables in terms of proportions $\alpha_i = \frac{N_i}{n}$ that are allowed to take values in $\left[0,1\right]$ and satisfy $\sum_i \alpha_i = 1$. This literature also takes $n$ to be very large, so that the continuous relaxation of $N_i$ is not an issue. We will also adopt this approach in our analysis.

\subsection{Comparison With Frequentist Formulation}\label{sec:comparison}

It is relevant to compare (\ref{eq:bayesPCS}) with a frequentist variant of OCBA in which $PCS^n_B$ is replaced by
\begin{align}\label{ineq:pcs_freq}
	PCS^n_{F} \triangleq&  \Eb_{F} [ \1\{ \cap_{i \ne i^*} \{ \bar{X}^n_{i^*}
	> \bar{X}^n_i \} \} ],
\end{align}
where $\bar{X}^n_i = \frac{1}{N_i}\sum^{N_i}_{l=1} X_{i,l}$ is the sample average for design $i = 1,...,k$. The expectation $\Eb_F$ is taken over the joint distribution of the sample averages. Note that this distribution is completely different from the posterior density used in (\ref{eq:pcs_bayes}). However, like its Bayesian counterpart, $PCS^n_F$ is not expressible in closed form.

\cite{glynn2004} pioneered an analytical approach, based on large deviations theory, for optimizing the asymptotic convergence rate of $PCS^n_F$. Like OCBA, this approach also takes $n$ to be large and works with the proportions $\boldsymbol{\alpha} = \left(\alpha_1,...,\alpha_k\right)$, each one satisfying $\alpha_i = \frac{N_i}{n}$, rather than with the sample sizes $N_i$ directly. In brief, one first derives a function
\begin{equation}\label{glynn}
	\mathcal{G}\left(\boldsymbol{\alpha}\right) = -\lim_{n\rightarrow\infty} \frac{1}{n}\log\left(1- PCS^n_F\right).
\end{equation}
As long as $\alpha_i > 0$ for all $i$, the quantity $\mathcal{G}\left(\boldsymbol{\alpha}\right)$ is strictly positive, meaning that the probability of \textit{false} selection (PFS) converges to zero at an exponential rate with exponent $-\mathcal{G}\left(\boldsymbol{\alpha}\right)$. One can speed up convergence by choosing the allocation $\boldsymbol{\alpha}$ to maximize $\mathcal{G}$, giving rise to the OCBA-like problem
\begin{equation}\label{eq:glynnocba}
	\max_{\boldsymbol{\alpha}} \mathcal{G}\left(\boldsymbol{\alpha}\right) \quad \mbox{s.t. } \sum_{i=1}^k \alpha_i = 1, \alpha_i \geq 0\text{ for }i=1,2,...,k.
\end{equation}
Unlike the PCS itself, $\mathcal{G}\left(\boldsymbol{\alpha}\right)$ has a clean expression under the assumption of normal sampling distributions. Example 1 of \cite{glynn2004} shows, in that setting, that
\begin{equation}\label{eq:freqrate}
	\mathcal{G}\left(\boldsymbol{\alpha}\right) = \min_{i\neq i^*} \frac{(\mu_{i^*}-\mu_i)^2}{2(\sigma_{i^*}^2/\alpha_{i^*} + \sigma_i^2/\alpha_i)}.
\end{equation}
Essentially, the convergence rate of the probability of false selection is identical to the slowest convergence rate among the probabilities $P\left(\bar{X}^n_{i^*}\le \bar{X}^n_i\right)$ for pairwise comparisons between $i^*$ and individual $i$. To put it a different way, these probabilities all converge to zero at exponential rates with different exponents, and the smallest among the exponents determines the asymptotic behavior of the overall PCS. This smallest exponent is precisely the right-hand side of (\ref{eq:freqrate}).

Plugging (\ref{eq:freqrate}) into (\ref{eq:glynnocba}), it is possible to show that the optimal $\alpha_i$ values are completely characterized by the equations
\begin{eqnarray}
	\sum_{i \ne i^*} \frac{\sigma^2_{i^*}/\alpha^2_{i^*}}{\sigma_i^2/\alpha^2_i} &=& 1,  \label{glynn1}\\
	\frac{(\mu_{i^*}-\mu_i)^2}{2(\sigma_{i^*}^2/\alpha_{i^*} + \sigma_i^2/\alpha_i)} &=&  \frac{(\mu_{i^*}-\mu_j)^2}{2(\sigma_{i^*}^2/\alpha_{i^*} + \sigma_j^2/\alpha_j)}, \ i,j \ne i^*. \label{glynn2}
\end{eqnarray}
In a special case where $\alpha_{i^*}\gg \alpha_i$ for $i\neq i^*$, (\ref{glynn2}) reduces to
\begin{align}
	\frac{(\mu_{i^*}-\mu_i)^2}{\sigma_i^2/\alpha_i} =  \frac{(\mu_{i^*}-\mu_j)^2}{\sigma_j^2/\alpha_j}, \ i,j \ne i^*,\label{glynn3}
\end{align}
which matches the well-known OCBA equations, originally derived by \cite{chen2000} using an approximation of $PCS^n_F$.

This analysis is attractive because it provides us with a very strong, yet analytically tractable notion of optimality. Of course, if we do not know the true means and variances, we cannot directly solve (\ref{glynn1})-(\ref{glynn2}). However, it is possible to design sequential sampling algorithms \citep{gao2017,chen2023balancing} that replace the unknown quantities in these equations with plug-in estimators and allocate samples to designs one at a time in a manner that provably converges to the optimal allocation. In this way, the convergence rate analysis of $PCS^n_F$ provides principled guidance for the development of practical, computationally efficient procedures.

The drawback of this approach is that, because the rate is derived for fixed $\mu_i$ and $\sigma^2_i$, it does not involve any sense of parameter uncertainty. However, the strength of the results that can be obtained has motivated efforts to study the Bayesian version of the problem in a similar way. It should be understood, however, that such analyses (as well as the analysis in this paper) are not fully Bayesian, but rather, a hybrid of frequentist and Bayesian concepts. As in a frequentist analysis, one treats the true parameter values as fixed quantities. Strictly speaking, $\pi^{n}$ is then no longer a posterior distribution, because $\left(\boldsymbol{\mu},\boldsymbol{\sigma}^2\right)$ are non-random. Nonetheless, one can still define and study the measure $\pi^{n}$ obtained by applying Bayesian updating to the observations; as each design is sampled more and more, this measure will concentrate around the fixed true values.

Using this style of analysis, \cite{ryzhov2016convergence} found that EI, a popular Bayesian allocation procedure dating back to \cite{jones1998efficient}, asymptotically solves (\ref{glynn3}). \cite{chen2011} optimized an approximation of $PCS^n_B$ and arrived at (\ref{glynn1}) and (\ref{glynn3}), similarly to frequentist OCBA. \cite{chen2019complete} studied a modification of EI that solved (\ref{glynn1})-(\ref{glynn2}). Most notably, \cite{russo2020simple} found that, when the variances $\sigma^2_i$ are \textit{known}, an analog of (\ref{glynn}) holds for $PCS^n_B$ and, in fact, coincides with the result obtained by \cite{glynn2004}. In other words, equations (\ref{glynn1})-(\ref{glynn2}) describe the optimal allocation in both a frequentist setting \textit{and} a Bayesian setting with known variance. However, as we will show in the remainder of this paper, this is no longer the case when the sampling variance is unknown.

\section{Optimal Sampling Allocations Under Unknown Sampling Variance}

This section characterizes the optimal budget allocation for the Bayesian PCS. Similarly to the discussion in Section \ref{sec:comparison}, we adopt a partially Bayesian style of analysis where $\mu_i$ and $\sigma^2_i$ are fixed (though unknown to the decision-maker).

Our analysis works for a general prior that potentially allows correlations across designs, as in, e.g., \cite{frazier2009knowledge} or \cite{HaRyDe16}. Section \ref{sec:prior} states the assumptions that we impose on such a prior. Section \ref{sec:pcsrate} then derives a law of the form (\ref{glynn}), but for $PCS^n_B$.  Section \ref{sec:opt_cond} presents optimality conditions for the allocation $\boldsymbol{\alpha}$.

\subsection{Assumptions on the Prior}\label{sec:prior}

For the setting of normally distributed rewards with unknown mean and variance, two conjugate Bayesian models (namely, normal-gamma and normal-inverse-gamma) are available \citep{degroot2005optimal}. If one chooses either of these models, the assumptions below become unnecessary and one can skip ahead to Section \ref{sec:pcsrate}. However, these models assume independent beliefs across designs. If one wishes to allow the use of correlated beliefs as in, e.g., \cite{frazier2009knowledge}, some regularity conditions on the prior density are required.

We assume that the prior joint density is bounded above, i.e., there exists a constant $\bar{c} > 0$ such that $\pi^{0}(\boldsymbol{\phi},\boldsymbol{\psi}) \le \bar{c}$ for all $(\boldsymbol{\phi},\boldsymbol{\psi}) \in \Rb^k \times \Rb_+^k$. For ease of presentation, we additionally require
\begin{align}\label{ineq:prior_int}
	\int_{\Rb^k} \int_{\Rb^k_+} \pi^{0}(\boldsymbol{\phi},\boldsymbol{\psi}) d\boldsymbol{\phi} d\boldsymbol{\psi} = 1.
\end{align}
Technically, this rules out the use of improper or noninformative priors, but conceivably one could collect a small number of initial samples to update such a prior into a density that integrates to one.

We further assume that there exists a constant $\underline{c} > 0$ such that $\pi^{0}(\boldsymbol{\phi},\boldsymbol{\psi}) \ge \underline{c}$ for $\left(\boldsymbol{\phi},\boldsymbol{\psi}\right)\in H_w$, where $H_w \subseteq \Rb^k \times \Rb_+^k$ contains the unknown true values. Specifically, we can let
\begin{align}\label{ineq:prior_reg}
	H_{w} \triangleq [ \mu_{\min} - \bar\epsilon, \mu_{\max} + \bar\epsilon]^k  \times [\sigma_{\min}^2 - \bar\epsilon, \sigma_{\max}^2 + \bar\epsilon + ( \mu_{\max} - \mu_{\min} + 2 \bar\epsilon )^2 ]^k,
\end{align}
where $\mu_{\min}$ and $\mu_{\max}$ (respectively, $\sigma^2_{\min}$ and $\sigma^2_{\max}$) are the smallest and largest of the true means (respectively, sampling variances), and $\bar\epsilon < \sigma_{\min}^2/2$ is a small positive constant.

Under these conditions, Bayesian updating will produce consistent estimators of the true values. We show that, for sufficiently large $n$, the posterior estimates $\hat{\mu}^n_i$ and $\left(\sigma^n_i\right)^2$ will become arbitrarily close to, respectively, the sample mean $\bar{X}^n_i$ and the maximum likelihood estimator $\left(\tilde{S}^n_i\right)^2$ of the variance. Of course, the specific $n$ that is ``sufficiently large'' depends on the sample path. Consistency is standard if one uses a conjugate prior, and one certainly expects it to hold more generally, but for completeness we provide a proof in the Electronic Companion (EC).

\begin{lemma}\label{lem:bayes_consist}
	Suppose that $N_i \to \infty$ as $n \to \infty$. Then, for $\varepsilon \le \min\{ 1 / (2k),\bar{\epsilon}/4 \}$, we have $|\hat{\mu}^n_i - \bar{X}^n_i| \le \varepsilon$ and $|\left(\hat{\sigma}^n_i\right)^2 - \frac{N_i}{N_i-5} \left(\tilde{S}^n_i\right)^2| \le \varepsilon$ for all sufficiently large $n$.
\end{lemma}

We also show that the posterior estimates for design $i$ have finite limits when $i$ is not sampled infinitely often. This is obvious if the prior is independent across designs, but less so when correlated beliefs are present since our beliefs about $i$ can be updated when we sample other designs. We provide a proof under the additional assumption that the prior density is uniformly continuous.

\begin{lemma}\label{lem:posterior_non}
	Suppose the prior density $\pi^0(\boldsymbol{\phi},\boldsymbol{\psi})$ is uniformly continuous in $\Rb^k \times \Rb_+^k$ and $N_i \ge 6$. If $N_i$ does not increase to infinity as $n \to \infty$, then $\hat{\mu}^n_i$ and $\left(\sigma^n_i\right)^2$ still converge to finite limits.
\end{lemma}

\subsection{Convergence Rate of the Bayesian PCS}\label{sec:pcsrate}

The main result of our paper is that $PCS^n_B$ obeys a large deviations law under unknown mean and variance. We state and discuss this result, then provide a sketch of the proof.

\begin{theorem}
	\label{thm:pcs}
	The Bayesian probability $PCS^n_B$ of correct selection satisfies
	\begin{align}
		&-\lim_{n \to \infty} \frac{1}{n}  \log (1-PCS^n_{B})
		=  \min_{i \ne i^*}   \Vcal_{i}(\alpha_i,\alpha_{i^*}) \label{eq:bayesPCSrate}
	\end{align}
	where
	\begin{align}\label{eq:vdef}
		\Vcal_{i}(\alpha_i,\alpha_{i^*}) \triangleq \min_{\phi_{i} } \left( \frac{\alpha_i}{2} \log \left(1 + \frac{(\mu_i-\phi_i)^2}{\sigma_i^2} \right) + \frac{\alpha_{i^*}}{2} \log \left(1 + \frac{(\mu_{i^*}-\phi_{i})^2}{\sigma_{i^*}^2} \right) \right).
	\end{align}
\end{theorem}

Recall that, when the sampling variance is known, the convergence rate of $PCS^n_B$ is identical to its frequentist analog in (\ref{eq:freqrate}). However, when the variance is uncertain, the structure of the rate function changes dramatically. The difference between these two settings is best seen through the following comparison. \cite{russo2020simple} showed that, with known sampling variance, (\ref{eq:bayesPCSrate}) can be written as
\begin{align*}
	-\lim_{n \to \infty} \frac{1}{n}  \log (1-PCS^n_{B})  =& \min_{i\neq i^*} \min_{\phi_i,\phi_{i^*}: \ \phi_{i^*} \le \phi_i}\alpha_i D^{KL}\left(\mu_i,\sigma^2_i\mid\phi_i,\sigma^2_i\right) + \alpha_{i^*}D^{KL}\left(\mu_{i^*},\sigma^2_{i^*}\mid\phi_{i^*},\sigma^2_{i^*}\right) \\
	=& \min_{i\neq i^*} \min_{\phi_i}\alpha_i D^{KL}\left(\mu_i,\sigma^2_i\mid\phi_i,\sigma^2_i\right) + \alpha_{i^*}D^{KL}\left(\mu_{i^*},\sigma^2_{i^*}\mid\phi_{i},\sigma^2_{i^*}\right),
\end{align*}
where $D^{KL}\left(\mu,\sigma^2\mid \tilde{\mu},\tilde{\sigma}^2\right)$ denotes the Kullback-Leibler (KL) divergence between two normal distributions with respective parameters $\left(\mu,\sigma^2\right)$ and $\left(\tilde{\mu},\tilde{\sigma}^2\right)$.

When the variance is unknown, it can be shown with some tedious algebra that (\ref{eq:bayesPCSrate})-(\ref{eq:vdef}) can also be written in terms of a weighted average of KL divergences, but this time
\begin{align*}
	\Vcal_{i}(\alpha_i,\alpha_{i^*}) =& \min_{\phi_i,\psi_i,\phi_{i^*},\psi_{i^*}:\  \phi_{i^*} \le \phi_i }\alpha_i D^{KL}\left(\mu_i,\sigma^2_i\mid\phi_i,\psi_i\right) + \alpha_{i^*}D^{KL}\left(\mu_{i^*},\sigma^2_{i^*}\mid\phi_{i^*},\psi_{i^*}\right)\\
	=& \min_{\phi_i}\alpha_i D^{KL}\left(\mu_i,\sigma^2_i\mid\phi_i,\sigma^2_i + \left(\mu_i-\phi_i\right)^2\right) + \alpha_{i^*}D^{KL}\left(\mu_{i^*},\sigma^2_{i^*}\mid\phi_{i},\sigma^2_{i^*}+\left(\mu_{i^*}-\phi_i\right)^2\right).
\end{align*}
Crucially, the KL divergence now depends on the variable $\phi_i$ through both the mean parameter and the variance parameter. While the weighted average, for fixed $\phi_i$, can be written in closed form as in (\ref{eq:vdef}), the minimum over $\phi_i$ is no longer analytically tractable. In fact, unlike the known-variance case, the weighted average is no longer convex in $\phi_i$ (what is more, it is not even unimodal) and must be optimized numerically. The non-convexity issue will be explored further in Section \ref{sec:opt_cond}.

Although the proof of Theorem \ref{thm:pcs} is inspired by \cite{russo2020simple}, that paper focuses on sampling distributions that belong to one-parameter exponential families, and also makes the additional assumption that the prior and posterior densities have compact support. Neither restriction is present in our setting, and significant modifications to the analytical technique are necessary. Our proof, which is given in the EC, does not use the KL divergence characterization; the main ideas are sketched out as follows. For any $i\neq i^*$, define a set
\begin{align*}
	&\Xi_i \triangleq \{ (\boldsymbol{\phi},\boldsymbol{\psi}) \in \Rb^{k} \times \Rb^{k}_+: \phi_{i^*} \le \phi_i \}.
\end{align*}	
Then, $\Pb_{B}(\Xi_i) \triangleq \int  \int_{ \Xi_i }   \pi^n (\boldsymbol{\phi},\boldsymbol{\psi}) d \boldsymbol{\psi} d \boldsymbol{\phi}$ represents our belief about the probability of design $i$ being better than $i^*$. Similarly to the analysis of frequentist PCS, we find
\begin{align*}
	-\lim_{n \to \infty} \frac{1}{n} \log(1-PCS_{B}^n) = \min_{i \ne i^*}\left( -\lim_{n \to \infty} \frac{1}{n} \log\Pb_{B}(\Xi_i)\right),
\end{align*}
that is, the asymptotic behavior of PCS is driven by the slowest convergence rate among the pairwise comparisons between $i^*$ and $i\neq i^*$. We then write
\begin{align}\label{eq:post_probi}
	\Pb_{B}(\Xi_i) =  \frac{ \int  \int_{ \Xi_i } \pi^{0}(\boldsymbol{\phi},\boldsymbol{\psi}) \prod_{j=1}^{k} \left(L^{n} (\phi_{j},\psi_{j}) / L^{n} \left(\bar{X}^n_{j},\left(\tilde{S}^n_{j}\right)^2\right)\right) d \boldsymbol{\psi} d \boldsymbol{\phi} }{  \int_{\Rb^k} \int_{\Rb^k_+}  \pi^{0}(\boldsymbol{\phi}',\boldsymbol{\psi}') \prod_{j=1}^{k} \left(L^{n} (\phi_{j}',\psi_{j}') / L^{n} \left(\bar{X}^n_{j},\left(\tilde{S}^n_{j}\right)^2\right)\right) d\boldsymbol{\psi}' d\boldsymbol{\phi}'}.
\end{align}
The two integrals in \eqref{eq:post_probi} primarily concentrate around the points where the integrands achieve their respective maxima over $\Xi_i$ and $\Rb^k \times \Rb^k_+$, and the impact of the prior density $\pi^{0}(\boldsymbol{\phi},\boldsymbol{\psi})$ vanishes as $n \to \infty$. By analyzing the rate at which the maximum values of the two integrands converge, we establish the claim of Theorem \ref{thm:pcs}.

\subsection{Optimality Conditions}\label{sec:opt_cond}

Analogously to (\ref{eq:glynnocba}), we formulate the optimization problem
\begin{equation}\label{eq:newocba}
	\max_{\boldsymbol{\alpha}} \min_{i\neq i^*} \mathcal{V}_i\left(\alpha_i,\alpha_{i^*}\right) \quad \mbox{s.t. } \sum_{i=1}^k \alpha_i = 1, \alpha_i \geq 0\text{ for }i=1,2,...,k.
\end{equation}
For our analysis, it is sometimes convenient to rewrite $\mathcal{V}_i$ as a function of a single variable $r=\frac{\alpha_i}{\alpha_{i^*}}$. Define
\begin{align}
	&g_i( \phi_i, r ) \triangleq r  \log (1 + (\mu_i-\phi_i)^2/\sigma_i^2 ) +  \log (1 + (\mu_{i^*}-\phi_{i})^2/\sigma_{i^*}^2 ), \nonumber \\
	&\Wcal_i( r ) \triangleq \min_{\phi_{i} } g_i( \phi_i, r ). \label{eq:taudef}
\end{align}
It is easy to see that $\Vcal_i\left(\alpha_i,\alpha_{i^*}\right) = \frac{\alpha_{i^*}}{2} \Wcal_i( r )$, and the same $\phi_i$ optimizes both quantities.

Either way, the minimization problem in (\ref{eq:vdef}) and (\ref{eq:taudef}) is non-convex. Observe that
\begin{align}
	\frac{\partial g_i\left(\phi_i, r\right)}{\partial \phi_i } = \frac{ r (\phi_{i}-\mu_i) ( \sigma_{i^*}^2 + (\phi_{i}-\mu_{i^*})^2 ) + (\phi_{i}-\mu_{i^*}) (\sigma_i^2 + (\phi_{i}-\mu_i)^2) }{ (\sigma_i^2 + (\phi_{i}-\mu_i)^2) (\sigma_{i^*}^2 + (\phi_{i}-\mu_{i^*})^2) }.  \label{eq:v_deriv}
\end{align}
The numerator of (\ref{eq:v_deriv}) is a cubic polynomial of $\phi_i$ and may have up to three real roots. Consequently, the minimization problem in (\ref{eq:taudef}) may have multiple local optimal solutions, which, however, could have the same objective value. The following example provides an illustration.

\begin{example}\label{rem:disc}
	Let $\mu_i=0$, $\mu_{i^*} = 10$, and $\sigma_i=\sigma_{i^*}=1$. In Figure \ref{fig:gplot}, we plot $g_i( \phi_i, r )$ with $r = 0.9$, $1$ and $1.1$ respectively. We can observe that: (i) when $r = 1$, both $\phi_i \approx 0.101$ and $\phi_i \approx 9.899$ are global optimal solutions and have the same objective value; (ii) when $r = 0.9$, the global optimal solution is unique and near $\mu_{i^*} = 10$; and (iii) when $r = 1.1$, the global optimal solution is unique and near $\mu_{i} = 0$.
	
	\begin{figure}[htbp]
		\caption{Plot of $g_i( \phi_i, r )$ for the three cases considered in Example \ref{rem:disc}.}\label{fig:gplot}	\includegraphics[width=0.99\textwidth]{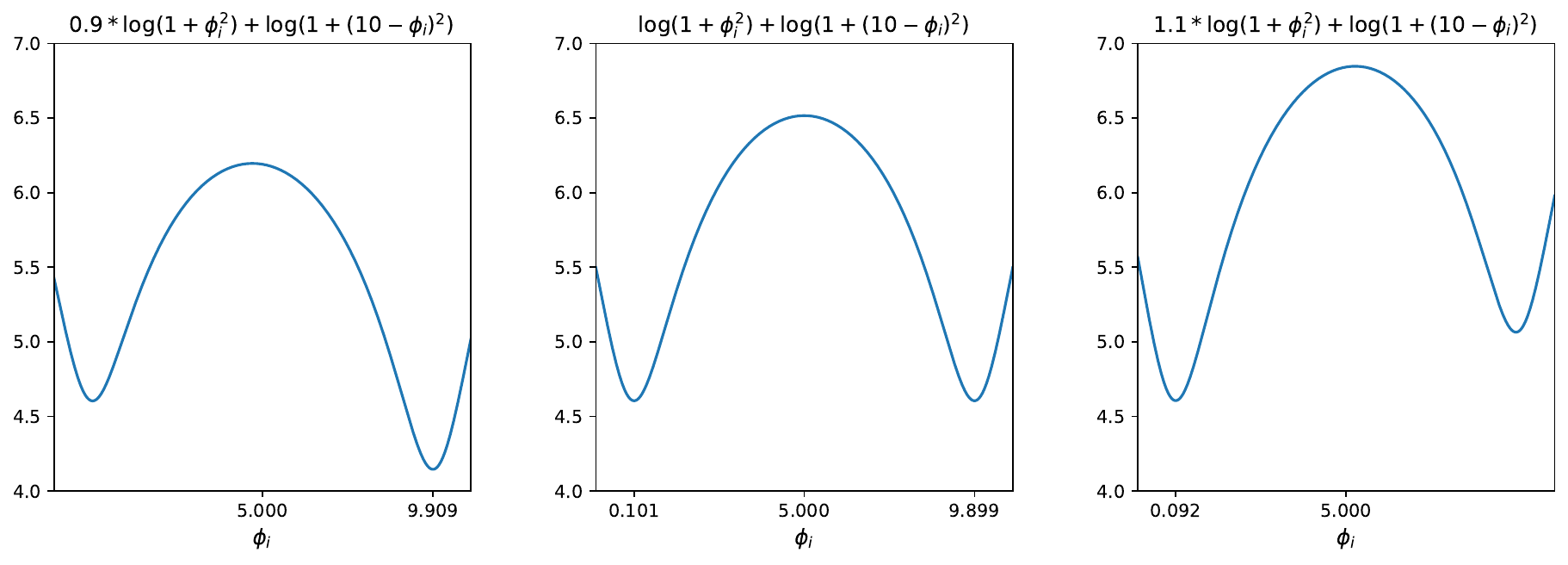}
		
	\end{figure}
	
	Observe that the presence of two globally optimal solutions creates a discontinuity: when $r = 1$, $\arg\min_{\phi_i} g_i\left(\phi_i,r\right)$ is a set containing two values, which are approximately $0.101$ and $9.899$. Then, when $r < 1$, we have $\arg\min_{\phi_i} g_i\left(\phi_i,r\right) > 9.899$, and when $r > 1$, we have $\arg\min_{\phi_i} g_i\left(\phi_i,r\right) < 0.101$.
\end{example}

The following technical result formally establishes the existence of the discontinuity observed in Example \ref{rem:disc}. The jump in $\arg\min_{\phi_i} g_i\left(\phi_i,r\right)$ does not have to occur at $r = 1$, as in Example \ref{rem:disc}.

\begin{lemma}\label{lem:phimono0}
	Let $\phi^{\min}_i\left(r\right)$ and $\phi^{\max}_i\left(r\right)$ be the smallest and largest elements, respectively, of the set $\arg\min_{\phi_i} g_i\left(\phi_i,r\right)$. The following properties hold:
	\begin{itemize}	
		\item[i)] If $r > r'$, then $\phi^{\max}_i(r) < \phi^{\min}_i(r')$.
		\item[ii)] Both $\phi^{\min}_i\left(r\right)$ and $\phi^{\max}_i\left(r\right)$ are in the interval $\left[\mu_i,\mu_{i^*}\right]$, with
		\begin{eqnarray*}
			\phi^{\min}_i\left(0\right) = \phi^{\max}_i\left(0\right) &=& \mu_{i^*},\\
			\phi^{\min}_i\left(\infty\right) = \phi^{\max}_i\left(\infty\right) &=& \mu_{i}.
		\end{eqnarray*}
		\item[iii)] Suppose that $r \leq \frac{1}{b}$ for some $b \geq 1$. Then, there exists $\eta\left(b\right)$ such that $\lim_{b\rightarrow\infty} \eta\left(b\right) = 0$ and $\mu_{i^*} - \phi^{\min}_i\left(r\right) \leq \eta\left(b\right)\left(\mu_{i^*}-\mu_i\right)$. Analogously, if $r \geq b \geq 1$, then $\phi^{\max}_i\left(r\right) - \mu_i \leq \eta\left(b\right)\left(\mu_{i^*}-\mu_i\right)$.
		\item[iv)] Let $\left\{r^l\right\}^{\infty}_{l=0}$ be a sequence. If $r^l \searrow r$, then $\phi^{\max}_i\left(r^l\right) \rightarrow \phi^{\min}_i\left(r\right)$. If $r^l \nearrow r$, then $\phi^{\min}_i\left(r^l\right)\rightarrow\phi^{\max}_i\left(r\right)$.
	\end{itemize}
\end{lemma}

Since $\phi^{\min}_i\left(r\right) \leq \phi^{\max}_i\left(r\right)$ by definition, Lemma \ref{lem:phimono0}(i) implies that both $\phi^{\min}_i\left(r\right)$ and $\phi^{\max}_i\left(r\right)$ are decreasing in $r$. Thus, in Example \ref{rem:disc}, $\phi^{\max}_i\left(r\right) \ge \phi^{\min}_i\left(r\right) \geq 9.899$ for $r < 1$, and $\phi^{\min}_i\left(r\right) \le \phi^{\max}_i\left(r\right) \leq 0.101$ for $r > 1$. Both functions have a jump from $9.899$ to $0.101$ occurring at $r=1$. Note that, by Lemma \ref{lem:phimono0}(ii)-(iii), both functions always take values in $[\mu_i,\mu_{i^*}]$, approaching $\mu_{i^*}$ when the ratio $r$ is very small and $\mu_i$ when it is large.

The discontinuity of $\phi^{\min}_i$ does not carry over to the individual rate functions $\Vcal_{i}$, which are always continuous in $\left(\alpha_i,\alpha_{i^*}\right)$. We can provide an upper bound on the gap between $\Vcal_i(\alpha_i,\alpha_{i^*})$ and $\Vcal_i(\alpha'_i,\alpha'_{i^*})$ in terms of the difference between $(\alpha_i,\alpha_{i^*})$ and $(\alpha'_i,\alpha'_{i^*})$. This and other useful properties of $\Vcal_i$ and $\Wcal_i$ are stated below.

\begin{lemma}\label{lem:vmono0}
	The following properties hold:
	\begin{itemize}
		\item[i)] $\Vcal_{i} (\alpha_i,\alpha_{i^*})$ is continuous in $(\alpha_i,\alpha_{i^*})$.	
		\item[ii)] Consider two pairs $\left(\alpha_i,\alpha_{i^*}\right)$ and $\left(\alpha'_i,\alpha'_{i^*}\right)$. For $\phi^* \in\left\{\phi^{\min}_i\left(\alpha_i/\alpha_{i^*}\right),\phi^{\max}_i\left(\alpha_i/\alpha_{i^*}\right)\right\}$,
		\begin{align*}
			2\Vcal_i(\alpha'_i,\alpha'_{i^*}) \le 2\Vcal_i (\alpha_i,\alpha_{i^*}) +  (\alpha'_i-\alpha_i ) \log \left(1 + \frac{(\mu_i-\phi^*)^2}{\sigma_i^2} \right) +  (\alpha'_{i^*}-\alpha_{i^*} ) \log \left(1 + \frac{(\mu_{i^*}-\phi^*)^2}{\sigma_{i^*}^2} \right).
		\end{align*}
		Moreover, the inequality is strict when $\phi^* \notin \arg\min_{\phi_i} g_i\left(\phi_i,\alpha_i'/\alpha_{i^*}'\right)$.
		\item[iii)] The function $\Wcal_i\left(r\right)$ is strictly increasing in $r$ with $\Wcal_i\left(0\right) = 0$.
	\end{itemize}
\end{lemma}

While the continuity of $\Vcal_i$ holds out hope for the tractability of (\ref{eq:newocba}), it is difficult to approach this problem using the standard tools of convex optimization, that is, by deriving the first-order conditions of the problem as done in \cite{glynn2004} and many other papers. Since $\Vcal_i$ depends on $\left(\alpha_i,\alpha_{i^*}\right)$ through $\phi^{\min}_i$, we cannot differentiate it. Instead, we proceed in two steps. First, we take an arbitrary constant $0 < \bar{\alpha}_{i^*} < 1$ and solve (\ref{eq:newocba}) under the additional constraint $\alpha_{i^*} = \bar{\alpha}_{i^*}$. The solution is characterized in the following results. The proofs are deferred to the EC.

\begin{lemma}\label{lem:vequal}
	Let $0 < \bar{\alpha}_{i^*} < 1$. The following statements hold:
	
	\begin{itemize}
		\item[i)] There exists a unique allocation $\boldsymbol{\alpha}^f\left(\bar{\alpha}_{i^*}\right)$ satisfying $\alpha_i^f\left(\bar{\alpha}_{i^*}\right) > 0$ for all $i$, $\sum_i \alpha_i^f\left(\bar{\alpha}_{i^*}\right) = 1$, and
		\begin{eqnarray}
			\Vcal_{i}(\alpha_{i}^f\left(\bar{\alpha}_{i^*}\right),\alpha_{i^*}^f\left(\bar{\alpha}_{i^*}\right)) &=& \Vcal_{i'}(\alpha_{i'}^f\left(\bar{\alpha}_{i^*}\right),\alpha_{i^*}^f\left(\bar{\alpha}_{i^*}\right), \  i,i' \ne i^*,\label{eq:balance}\\
			\alpha_{i^*}^f\left(\bar{\alpha}_{i^*}\right) &=& \bar{\alpha}_{i^*}.\label{eq:alphaconstraint}
		\end{eqnarray}
		\item[ii)] Take $\bar{\alpha}'_{i^*}$ such that $0 < \bar{\alpha}_{i^*} - \bar{\alpha}_{i^*}' \le \Delta$. Then $0 < \alpha_i^f\left(\bar{\alpha}'_{i^*}\right) - \alpha_i^f\left(\bar{\alpha}_{i^*}\right) \le \Delta$ for all $i \neq i^*$.
		
		\item[iii)] $\boldsymbol{\alpha}^f\left(\bar{\alpha}_{i^*}\right)$ is the unique optimal solution to the maximization problem
		\begin{align}\label{eq:rateoptimization_toptwo}
			\max_{\balpha} \min_{i \ne i^*} \mathcal{V}_i\left(\alpha_i,\alpha_{i^*}\right)
		\end{align}
		subject to the constraints $\boldsymbol{\alpha} \geq 0$, $\sum_i \alpha_i = 1$, and $\alpha_{i^*} = \bar{\alpha}_{i^*}$.
	\end{itemize}
\end{lemma}

The existence and uniqueness of $\boldsymbol{\alpha}^f\left(\bar{\alpha}_{i^*}\right)$ shown in Lemma \ref{lem:vequal}(i) are obtained by applying the intermediate value theorem to the difference between $\Vcal_i$ and $\Vcal_{i'}$, $i \ne i'$. As a byproduct, we obtain Lemma \ref{lem:vequal}(ii) that shows the continuity and strict monotonicity of $\alpha_i^f\left(\bar{\alpha}_{i^*}\right)$ about $\bar{\alpha}_{i^*}$, $i \ne i^*$. Lemma \ref{lem:vequal}(iii) shows the optimality of $\boldsymbol{\alpha}^f\left(\bar{\alpha}_{i^*}\right)$ by applying the strict monotonicity of $\alpha_i^f\left(\bar{\alpha}_{i^*}\right)$ and $\Wcal_i$, with the implication that the optimal solution to (\ref{eq:newocba}) without the additional constraint $\alpha_{i^*} = \bar{\alpha}_{i^*}$ must satisfy \eqref{eq:balance}.

The second step is to remove the extra constraint. Let $\boldsymbol{\alpha}^*$ denote the optimal solution to (\ref{eq:newocba}). Then, $\boldsymbol{\alpha}^*$ is also the allocation obtained by setting $\bar{\alpha}_{i^*} = \alpha^*_{i^*}$ in Lemma \ref{lem:vequal}(i). It remains, therefore, to characterize the value of $\alpha^*_{i^*}$.

Let $r^f_i\left(\bar{\alpha}_{i^*}\right) \triangleq \alpha^f_i\left(\bar{\alpha}_{i^*}\right)/\bar{\alpha}_{i^*}$ for any $0 < \bar{\alpha}_{i^*} < 1$. Define
\begin{eqnarray}
	\Ucal^{\min}_i (r) &\triangleq & \log \left(1 + (\mu_{i}-\phi^{\min}_{i}(r))^2/\sigma_{i}^2 \right),\label{eq:ucaldef}\\
	\Ucal^{*,\min}_i(r) &\triangleq &  \log \left(1 + (\mu_{i^*}-\phi^{\min}_{i}(r))^2/\sigma_{i^*}^2 \right).\label{eq:ucalmaxdef}
\end{eqnarray}
Analogously, define $\Ucal^{\max}_i$ and $\Ucal^{*,\max}_i$ by replacing $\phi^{\min}_i$ in (\ref{eq:ucaldef})-(\ref{eq:ucalmaxdef}) by $\phi^{\max}_i$.

\begin{lemma}\label{lem:umono}
	For any $\bar{\alpha}_{i^*}$, let $\boldsymbol{\alpha}^f\left(\bar{\alpha}_{i^*}\right)$ be the allocation obtained from Lemma \ref{lem:vequal}(i). The following properties hold:		
	
	\begin{itemize}
		\item[i)] The mapping
		\begin{equation*}
			\bar{\alpha}_{i^*} \mapsto \sum_{i\neq i^*} \frac{\Ucal^{*,\min}_i\left( r^f_i\left(\bar{\alpha}_{i^*}\right) \right)}{\Ucal^{\min}_i\left( r^f_i\left(\bar{\alpha}_{i^*}\right) \right)}
		\end{equation*}
		is strictly decreasing in $\bar{\alpha}_{i^*}$.
		\item[ii)] The maximum
		\begin{align}\label{ocba_t}
			&\alpha_{i^*}^* \triangleq \max \left\{ \bar{\alpha}_{i^*}\,:\, \sum_{i\neq i^*} \frac{\Ucal^{*,\min}_i\left( r^f_i\left(\bar{\alpha}_{i^*}\right) \right)}{\Ucal^{\min}_i\left(r^f_i\left(\bar{\alpha}_{i^*}\right) \right)}\geq 1\right\}
		\end{align}
		exists and satisfies $0< \alpha_{i^*}^* < 1$.
		\item[iii)] The value $\alpha_{i^*}^*$ obtained in (\ref{ocba_t}) satisfies
		\begin{equation*}
			\sum_{i\neq i^*} \frac{\Ucal^{*,\max}_i\left( r^f_i\left(\alpha^*_{i^*}\right) \right)}{\Ucal^{\max}_i\left( r^f_i\left(\alpha^*_{i^*}\right) \right)}\leq 1.
		\end{equation*}
	\end{itemize}
\end{lemma}

The quantity $\alpha^*_{i^*}$ obtained from Lemma \ref{lem:umono}(ii) is precisely the optimal allocation to design $i^*$. We formally state the result, with the proof deferred to the EC. Essentially, we use Lemmas \ref{lem:phimono0}-\ref{lem:umono} to show that, for any $\bar{\alpha}_{i^*} \neq \alpha^*_{i^*}$,
\begin{equation*}
	\Vcal_i\left(\alpha^f_i\left(\bar{\alpha}_{i^*}\right),\bar{\alpha}_{i^*}\right) < \Vcal_i\left(\alpha^f_i\left(\alpha^*_{i^*}\right),\alpha^*_{i^*}\right), \qquad i\neq i^*,
\end{equation*}
which means that $\boldsymbol{\alpha}^f\left(\bar{\alpha}_{i^*}\right)$ achieves a lower objective value in problem (\ref{eq:newocba}) than $\boldsymbol{\alpha}^f\left(\alpha^*_{i^*}\right)$.

\begin{theorem}\label{th:ocba}
	The optimal solution $\boldsymbol{\alpha}^*$ to (\ref{eq:newocba}) is unique and satisfies
	\begin{eqnarray}
		\alpha_{i^*}^* &=& \max \left\{\bar{\alpha}_{i^*} \,:\, \sum_{i\neq i^*} \frac{\Ucal^{*,\min}_i\left( r^f_i\left(\bar{\alpha}_{i^*}\right) \right)}{\Ucal^{\min}_i\left( r^f_i\left(\bar{\alpha}_{i^*}\right) \right)}\geq 1\right\}, \label{eq:totalbalance}\\
		\Vcal_{i}( \alpha^*_i,\alpha^*_{i^*} ) &=& \Vcal_{j}( \alpha^*_{j},\alpha^*_{i^*} ), \qquad \forall i ,j \ne i^*.\label{eq:individualbalance}
	\end{eqnarray}
\end{theorem}

The optimality conditions (\ref{eq:totalbalance})-(\ref{eq:individualbalance}) are the analog of (\ref{glynn1})-(\ref{glynn2}) for the case of unknown variance. It is important to note that, unlike (\ref{glynn1}), condition (\ref{eq:totalbalance}) involves an additional maximization. This behavior arises in situations where $\phi^{\min}_j\big( r^f_j\big(\alpha_{i^*}^*\big) \big) < \phi^{\max}_j\big( r^f_j\big(\alpha_{i^*}^*\big) \big)$ for some $j \ne i^*$. If this occurs, we may have
\begin{equation*}
	\sum_{i\neq i^*} \frac{\Ucal^{*,\min}_i\Big( r^f_i\big(\alpha_{i^*}^*\big) \Big)}{\Ucal^{\min}_i\Big( r^f_i\big(\alpha_{i^*}^*\big) \Big)}> 1, \qquad \sum_{i\neq i^*} \frac{\Ucal^{*,\max}_i\Big( r^f_i\big(\alpha_{i^*}^*\big) \Big)}{\Ucal^{\max}_i\Big( r^f_i\big(\alpha_{i^*}^*\big) \Big)}< 1.
\end{equation*}
However, if $\phi^{\min}_i\left(\alpha^*_i/\alpha^*_{i^*}\right) = \phi^{\max}_i\left(\alpha^*_i/\alpha^*_{i^*}\right)$ for all $i \ne i^*$, then (\ref{eq:totalbalance}) will simplify to
\begin{equation}\label{eq:simplerbalance}
	\sum_{i\neq i^*} \frac{\log\left(1 + \left(\mu_{i^*} - \phi^{\min}_i\left(\alpha^*_i/\alpha^*_{i^*}\right)\right)^2 / \sigma^2_{i^*}\right)}{\log\left(1 + \left(\mu_{i} - \phi^{\min}_i\left(\alpha^*_i/\alpha^*_{i^*}\right)\right)^2 / \sigma^2_{i}\right)} =1,
\end{equation}
which has more structural resemblance to (\ref{glynn1}).

We may now summarize the similarities and differences between the known- and unknown-variance settings. In both settings, we maximize the slowest-converging probability of a false comparison made between $i^*$ and some $i\neq i^*$. In both settings, the optimal allocation balances the convergence rates of these probabilities so that they are all equal. In the known-variance setting, the balancing condition is (\ref{glynn2}), whereas in the unknown-variance setting, the condition is (\ref{eq:individualbalance}). While the structure of the two conditions is similar, the rate function has a very different form in the unknown-variance setting. Next, the known-variance setting has an additional condition (\ref{glynn1}) relating the proportion of the budget set aside for design $i^*$ to all the other proportions simultaneously. In some cases, we may have a similar condition (\ref{eq:simplerbalance}) in the unknown-variance setting, but in general, it may not be possible to achieve (\ref{eq:simplerbalance}) exactly because the optimization problem characterizing the rate exponent $\mathcal{V}_i$ is non-convex, though it would be convex in the case of known variance. In such a situation, (\ref{eq:totalbalance}) describes the optimal proportion $\alpha^*_{i^*}$.

\subsection{Insights Into Previous Work}\label{sec:insights}

Our results shed additional light on some observations made in the literature. First, let us consider a situation where $\alpha^*_{i^*} \gg \alpha^*_i$, analogous to the classic OCBA method for the frequentist setting. Then, the ratio $\alpha^*_i/\alpha^*_{i^*} \approx 0$ and, by Lemma \ref{lem:phimono0}(iii), we have $\phi^{\min}_i\left(\alpha^*_i/\alpha^*_{i^*}\right) \approx \mu_{i^*}$. In this special case, condition (\ref{eq:individualbalance}) becomes
\begin{equation*}
	\frac{\alpha^*_i}{\alpha^*_j} \approx \frac{\log\left(1 + \left(\mu_j - \mu_{i^*}\right)^2 / \sigma^2_j\right)}{\log\left(1 + \left(\mu_i - \mu_{i^*}\right)^2 / \sigma^2_i\right)}.
\end{equation*}
These are precisely the asymptotic sampling ratios derived in Theorem 2 of \cite{ryzhov2016convergence} for a variant of the EI method specialized to the setting of unknown variance. The same paper showed that, under known variance, the allocation achieved by EI behaves like the optimal allocation in the regime where $i^*$ receives much more samples than the other designs. It was conjectured that the sampling ratios under unknown variance behave similarly, but since the optimal allocation for that setting had not been established at the time, it was not possible to make a direct comparison. With our work filling that gap, we see that the conjecture indeed holds.

Second, we point out a connection between our results in Theorem \ref{th:ocba} and those of \cite{jourdan2023dealing}. This recent paper considers a frequentist setting where the design $i$ is assigned a proportion $\alpha_i$ of the budget, but the total number of samples is not fixed to be $n$, but rather determined according to a dynamic stopping rule that depends on $\boldsymbol{\alpha}$. Information collection terminates at some random time $\mathcal{T}_\delta\left(\boldsymbol{\alpha}\right)$, and the PCS at termination satisfies $PCS^{\mathcal{T}_{\delta}\left(\boldsymbol{\alpha}\right)}_F \geq 1-\delta$. The optimization problem is to choose $\boldsymbol{\alpha}$ to minimize $\liminf_{\delta\rightarrow 0} \frac{\mathbb{E}_F\left[\mathcal{T}_{\delta}\left(\boldsymbol{\alpha}\right)\right]}{\log\left(1/\delta\right)}$. By comparing the results in \cite{jourdan2023dealing} with ours, it can be seen that their objective $\liminf_{\delta\rightarrow 0} \frac{\mathbb{E}_F\left[\mathcal{T}_{\delta}\left(\boldsymbol{\alpha}\right)\right]}{\log\left(1/\delta\right)}$ is equivalent to $\left( \min_{i\neq i^*} \mathcal{V}_i\left(\alpha_i,\alpha_{i^*}\right)\right)^{-1}$, where $\mathcal{V}_i$ is as in (\ref{eq:vdef}). Consequently, the allocation obtained from Theorem \ref{th:ocba} is asymptotically optimal for both fixed-budget and fixed-precision settings.

It is also important to note that \cite{jourdan2023dealing} implicitly assumed the differentiability of $\phi^{\min}_i$, arriving at (\ref{eq:simplerbalance}) and (\ref{eq:individualbalance}) as the optimality conditions. Unfortunately, since $\phi^{\min}_i$ can be discontinuous, these equations do not always have a solution. In contrast, the allocation described by Theorem \ref{th:ocba} always exists.

\section{A Sequential OCBA Procedure for Learning the Optimal Allocation}\label{sec:algorithms}

The optimality conditions (\ref{eq:totalbalance})-(\ref{eq:individualbalance}) cannot be solved directly because they depend on the unknown problem parameters. In this section, we develop a sequential algorithm that learns the optimal allocation sequentially and can be guaranteed to converge to $\boldsymbol{\alpha}^*$ as $n\rightarrow\infty$. The algorithmic literature explicitly modeling uncertain variance is very limited. \cite{chick2010} and \cite{ryzhov2016convergence} study variants of EI that handle this type of uncertainty, but neither method converges to the optimal allocation. To our knowledge, our work is the first to offer this guarantee.

Suppose that we have collected $m$ samples, of which $N^m_i$ had been assigned to design $i$. As before, we denote by $\pi^m$ the density over $\left(\boldsymbol{\mu},\boldsymbol{\sigma}^2\right)$ that represents the decision-maker's posterior belief.\footnote[3]{Again, it is not quite accurate to say that $\pi^m$ is the posterior distribution of $(\boldsymbol{\mu},\boldsymbol{\sigma}^2)$ since our analysis views these quantities as nonrandom.} We plug the posterior estimates $\hat{\mu}^m_i$ and $\left(\hat{\sigma}^m_i\right)^2$ and the sampling proportions $\hat{\alpha}_i^m \triangleq N_i^m / m$ into the definitions of $\Vcal_i$, $\Ucal^{*,\min}_i$ and $\Ucal^{\min}_i$ instead of the unknown true values. Thus, we let $i^{*,m} \triangleq \arg\max_i \hat{\mu}^m_i$ be the design believed to be the best\footnote[4]{If $\arg\max_i \hat{\mu}^m_i$ is not unique, $i^{*,m}$ is taken to be the design in $\arg\max_i \hat{\mu}^m_i$ that has received the smallest number of samples up to time $m$.} based on $m$ samples and
\begin{equation}\label{eq:vcalm}
	\hat{\Vcal}^m_{i}  \triangleq \min_{\phi_{i} } \left( \frac{\hat{\alpha}_i^m }{2} \log \left(1 + \frac{(\hat{\mu}^m_i-\phi_{i})^2}{\left(\hat{\sigma}^m_i\right)^2} \right) + \frac{\hat{\alpha}_{i^{*,m}}^m}{2} \log \left(1 + \frac{(\hat{\mu}^m_{j}-\phi_{i})^2}{\big(\hat{\sigma}^m_{j}\big)^2} \right) \right),
\end{equation}
with $\hat{\phi}^{m}_{i}$ being the value of $\phi_{i}$ that achieves the argmin (or the smallest element of the argmin) in (\ref{eq:vcalm}). Since the algorithm requires only the smallest element of the argmin, we drop the notation ``$\min$'' from the superscript to reduce clutter. We then define $\hat{\Ucal}^{m}_i$, $\hat{\Ucal}^{*,m}_i$ by (\ref{eq:ucaldef})-(\ref{eq:ucalmaxdef}) with $\mu_i$, $\sigma_i$, and $\phi^{\min}_i$ replaced by $\hat{\mu}^m_i$, $\hat{\sigma}^m_i$ and $\hat{\phi}^m_{i}$, respectively. In our implementation, we compute $\hat{\phi}^m_{i}$ by numerically finding all roots of the relevant first-order equation.

\begin{algorithm}[t]
	\caption{$\mathcal{OCBA}^\mathcal{U}$ Algorithm.}
	\label{alg:ocba}
	\begin{algorithmic}
		\STATE {\bfseries Input:} Prior distribution $\pi$, initial sample size $n_0$, total budget $n$.
		\STATE {\bfseries Initialization:} Iteration counter $m\leftarrow0$. Perform $n_0$ replications for each design $i$ and compute $\pi^0$ from $\pi$ and the initial sample. Set $N^m_{i}=n_0$, $N^{m}= \sum_{i=1}^k N^m_{i}$ and $\hat{\alpha}^m_{i}=N^m_{i}/N^{m}$.
		\REPEAT
		\STATE $m\leftarrow m+1$.
		
		\STATE \textbf{Step 1.}  If $\arg\max_{i=1,\dots,k} \hat{\mu}^{m-1}_{i}$ is not unique, let $i^m = i^{*,m-1}$. Go to Step 4.
		
		\STATE \textbf{Step 2.} Compute $j^{m-1}\in\arg\min_{i\neq i^{*,m-1}}\hat{\Vcal}^{m-1}_{i}$, with ties broken arbitrarily.
		\STATE \textbf{Step 3.} If
		\begin{equation}\label{eq:estratio}
			\sum_{i \ne i^{*,m-1}}\hat{\Ucal}^{*,m-1}_{i}/\hat{\Ucal}^{m-1}_{i} > 1,
		\end{equation}
		set $i^m=i^{*,m-1}$; otherwise, set $i^m=j^{m-1}$.
		
		\STATE \textbf{Step 4.} Provide one more replication to design $i^m$ and update the posterior density $\pi^m$.
		\UNTIL{$N^{m}=n$.}
		
		\STATE {\bfseries Output:} Estimated best design $i^{*,m}$.
	\end{algorithmic}
\end{algorithm}

Our OCBA algorithm proceeds sequentially according to the rate-balancing principle put forth by \cite{gao2017} and made rigorous by \cite{chen2023balancing}. We approximate the optimality conditions (\ref{eq:totalbalance})-(\ref{eq:individualbalance}) by replacing $\Vcal_i$, $\Ucal^{\min}_i$ and $\Ucal^{*,\min}_i$ by their plug-in estimates. Instead of solving the equations for $\boldsymbol{\alpha}$, however, we evaluate them at $\hat{\boldsymbol{\alpha}}^m$, with each $\hat{\alpha}^m_i$ being the empirical proportion that design $i$ received out of the total budget spent thus far. Thus, $i^*$ in the optimality conditions is replaced by $i^{*,m}$, $\alpha^*_{i^*}$ is replaced by $\hat{\alpha}^m_{i^{*,m}}$, and all other $\alpha^*_i$ are replaced by $\hat{\alpha}^m_i$. We then look for imbalances in the equations and use these to determine the design that should receive the next simulation. The formal statement is given in Algorithm \ref{alg:ocba}.

Informally, the intuition behind the structure of the algorithm is that the individual rates $\hat{\Vcal}^{m-1}_{i}$ have a tendency to increase when $i$ is sampled (though with possible fluctuations due to changes in the estimated parameters as well as the index $i^{*,m-1}$), while the ratio on the left-hand side of (\ref{eq:estratio}) has a tendency to decrease when $i^{*,m-1}$ is sampled (once $m$ is large enough that $i^{*,m-1} = i^*$). Thus, we should sample $i^{*,m-1}$ when the left-hand side of (\ref{eq:estratio}) exceeds $1$, as this will tend to bring that quantity closer to the target value in (\ref{eq:totalbalance}). When we sample $i \neq i^{*,m-1}$, we should choose the smallest of the individual rate functions so that they increase at the same rate, attaining (\ref{eq:individualbalance}) in the limit.

We make this intuition more formal in the following proof outline, with the full technical details deferred to the EC. The proof has a four-part structure, analogously to \cite{chen2023balancing}, though the analytical technique significantly departs from that work, particularly in the later steps. We begin by showing the consistency of Algorithm \ref{alg:ocba}, i.e., that it samples each design infinitely often as the budget becomes large. If one wishes to use a general prior $\pi^0$, this analysis requires the assumptions of Section \ref{sec:prior}, including the uniform continuity needed for Lemma \ref{lem:posterior_non}.

\begin{lemma}\label{lem:const}
	For all designs $i=1,\dots,k$, $N^m_i \to \infty$ as $m \to \infty$.
\end{lemma}

The next step shows that the designs are not only sampled infinitely often, the sampling rates are of the same order. This step is distinct from consistency since Lemma \ref{lem:const} is first needed to largely eliminate the estimation error from the decisions made by the algorithm.	

\begin{lemma}\label{lem:one}
	For sufficiently large $m$, $\arg\max_{i} \hat{\mu}^m_{i}$ is unique. Furthermore, there exist numbers $b_{\alpha L},b_{\alpha U} > 0$ and $b_{\alpha L 2},b_{\alpha U 2} > 0$, possibly dependent on the sample path, such that
	\begin{align*}
		b_{\alpha L} \le \hat{\alpha}^m_{i}/\hat{\alpha}^m_{j} \le b_{\alpha U}, \qquad b_{\alpha L 2} \le \hat{\alpha}^m_{i} \le b_{\alpha U 2}
	\end{align*}
	for all $i,j = 1,2,\dots,k$ and all sufficiently large $m$.
\end{lemma}

The next step shows that the empirical estimates of the individual rate functions $\Vcal_i$ become balanced as $m$ becomes large. This is not quite the same as verifying (\ref{eq:individualbalance}), but is an important intermediate step toward this goal. Again, we find it more convenient to work with the empirical estimate $\hat{\Wcal}^m_{i} = \frac{2}{\hat{\alpha}^m_{i^{*,m}}}\hat{\Vcal}^{m}_{i}$ of $\Wcal_i$ rather than $\Vcal_i$. This result is one place where our analysis differs significantly from \cite{chen2023balancing}, which approached the individual rate functions as the last step. In particular, the discontinuity of $\hat{\phi}^m_{i}$ makes it impossible to rely on differentiability assumptions and necessitates the use of a different technique.

\begin{proposition}\label{proposv}
	Let $\bar\varepsilon$ be a small positive constant. For the budget allocation $\hat{\alpha}^m_{i}$ generated by the $\ocbau$ algorithm and any $\varepsilon$ with $0<\varepsilon<\bar{\varepsilon}$, there exists a random time $M(\varepsilon) < \infty$ such that $\max_{i,j\neq i^*} \left| \hat{\Wcal}^m_{i} - \hat{\Wcal}^m_{j}\right| \le \varepsilon$ for any $m \ge M(\varepsilon)$ almost surely.
\end{proposition}

In the last step, we complete the proof of the main convergence result showing that $\ocbau$ asymptotically learns the allocation $\boldsymbol{\alpha}^*$ that solves (\ref{eq:totalbalance})-(\ref{eq:individualbalance}).

\begin{theorem}\label{th:algconv}
	Let $\bar\varepsilon$ be a small positive constant. For the budget allocation $\hat{\alpha}^m_{i}$ generated by the $\ocbau$ algorithm and any $\varepsilon$ with $0<\varepsilon<\bar{\varepsilon}$, there exists a random time $M'(\varepsilon)$ with $M\left(\varepsilon\right) \leq M'(\varepsilon) < \infty$ such that $\max_{i =1,\dots,k} | \hat{\alpha}^m_{i} - \alpha_{i}^* | \le \varepsilon$ for any $m \ge M'(\varepsilon)$ almost surely.
\end{theorem}

In addition to this theoretical guarantee, we emphasize two appealing qualities of the algorithm. First, it is computationally efficient, as it does not require us to solve systems of nonlinear equations, which would be especially cumbersome with the complicated form of (\ref{eq:totalbalance}). Second, the algorithm learns the optimal allocation without resorting to tunable parameters or forced exploration. These elements are often used as workarounds for technical difficulties and continue to be present in many papers on optimal allocations; for example, the top-two method of \cite{russo2020simple} involves a tunable parameter controlling how often to sample $i^{*,m}$, while \cite{garivier2016optimal} uses forced exploration.

\section{Numerical Experiments}

We conduct three sets of experiments to evaluate the empirical performance of the $\mathcal{OCBA}^\mathcal{U}$ algorithm and also to demonstrate the importance of modeling uncertainty around the sampling variance. The first two sets of experiments are based on the following four synthetic instances:
\begin{itemize}
	\item Instance 1: $\mu_i = -1.5(i-1)$ and $\sigma_i^2 = 2^2$, $i=1,\dots,k$.	
	\item Instance 2: $\mu_i = -1.5(i-1)$ and $\sigma_i^2 = 5^2$, $i=1,\dots,k$.
	\item Instance 3: $\mu_i = -0.5(i-1)$ and $\sigma_i^2 = 2^2$, $i=1,\dots,k$.
	\item Instance 4: $\mu_i = -0.5(i-1)$ and $\sigma_i^2 = 5^2$, $i=1,\dots,k$.
\end{itemize}
We set $k=10$ for all instances, with design 1 as the best design. The means of the designs in Instances 1 and 2 are identical and evenly spaced, with higher sampling variances in Instance 2. Instances 3 and 4 are similarly designed, but the means of the designs are closer together, making it more difficult to distinguish between them.

In the first set of experiments, we do not conduct any sampling or learning. We assume that the true problem parameters are known, and compare the optimal allocation solving (\ref{eq:totalbalance})-(\ref{eq:individualbalance}) with its counterpart (\ref{glynn1})-(\ref{glynn2}). We are not looking at empirical performance at the moment; we simply wish to view the difference between the two allocations. Figure \ref{fig:opt_alloc} shows that, in general, the top two designs (i.e., designs $1$ and $2$) would receive larger shares of the budget when the decision-maker knows the sampling variances. On the other hand, when the variance is unknown, designs $3$-$10$ should be sampled more often. We see this behavior in all four instances, though the differences between the two allocations are less pronounced when the sampling variances are larger (Instance 2 vs. 1, or Instance 4 vs. 3). In these situations, we need to sample the top two designs more often to distinguish between them, leaving us with less wiggle room to allocate more samples to the other designs.

\begin{figure}
	\centering
	\caption{Optimal allocations for known and unknown sampling variances in the four instances.}
	\includegraphics[width=0.95\textwidth]{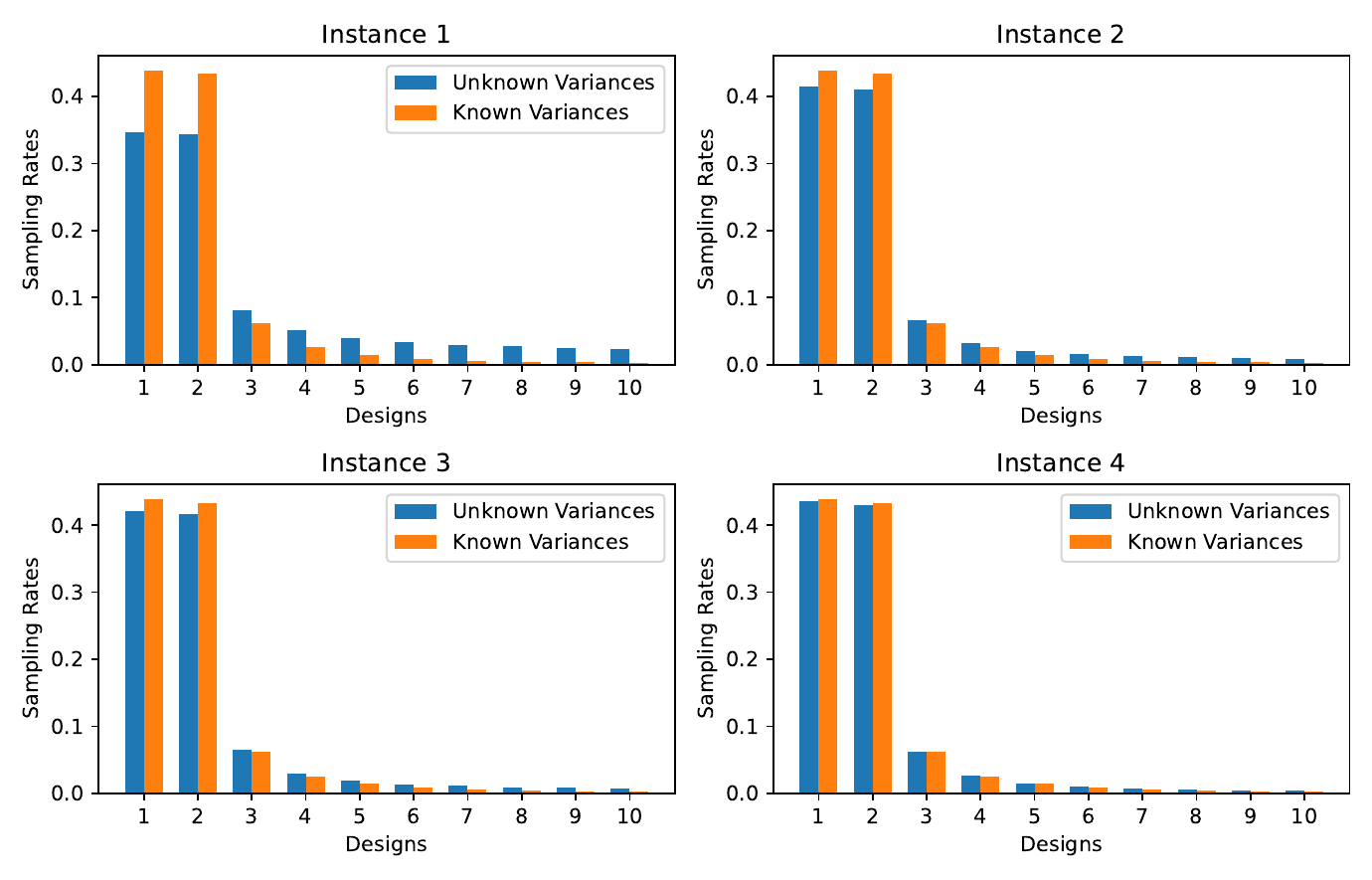}
	\label{fig:opt_alloc}
\end{figure}

The similarity of the allocations in Instance 4 can also be explained from a theoretical perspective. By Lemma \ref{lem:phimono0}, the value that achieves the minimum in \eqref{eq:vdef} lies in the interval $[\mu_i, \mu_{i^*}]$. When $\mu_{i^*} - \mu_i$ is small relative to $\sigma_i^2$, the function $\log(1 + (\mu_i - \phi_i)^2 / \sigma_i^2)$ can be approximated by $(\mu_i - \phi_i)^2 / \sigma_i^2$ for all $\phi_i \in [\mu_i, \mu_{i^*}]$. Substituting this approximation into (\ref{eq:vdef}) and \eqref{eq:bayesPCSrate} allows the minimization to be solved in closed form, yielding precisely the rate function (\ref{eq:freqrate}) that arises in the case of known variances. Thus, the two allocations should be very similar in this situation. However, as $(\mu_{i^*} - \mu_i)^2 / \sigma_i^2$ increases, the approximation becomes less accurate, and $\Vcal_{i}(\alpha_i,\alpha_{i^*})$ deviates more from its known-variance counterpart.

Our second set of experiments focuses on empirical performance in the four synthetic instances. We implement $\ocbau$ as well as the following benchmark methods:
\begin{itemize}
	\item \textit{Oracle allocation with unknown variance (Oracle-U)}. This method is given knowledge of $\mu_i$ and $\sigma^2_i$ and solves (\ref{eq:totalbalance})-(\ref{eq:individualbalance}) via brute force. It serves as a benchmark to evaluate any loss incurred by $\ocbau$ in finite sample.
	\item $\mathcal{OCBA}^\mathcal{K}$. We use a sequential method by \cite{li2023convergence} that converges to the optimal allocation under known variance, i.e., the solution of (\ref{glynn1})-(\ref{glynn2}). The method does not know the true variances, but replaces them with plug-in estimators.
	\item \textit{Expected improvement (EI)}. This is a classical EI method described in Sec. 5.1 of \cite{ryzhov2016convergence}. Unlike $\mathcal{OCBA}^\mathcal{K}$, EI explicitly models uncertainty around the variance, but does not converge to the optimal allocation. We have discussed its limiting behavior in Section \ref{sec:insights}.
	\item \textit{Equal allocation}. This method divides the budget equally across all designs and is a common benchmark in the R\&S literature.
\end{itemize}
After all samples have been collected, we select the design with the largest posterior mean. To make a fair comparison across methods, we use the same statistical model to estimate $\mu_i$. Namely, for each $\left(\mu_i,\sigma^2_i\right)$, we create an independent normal-inverse-gamma prior
\begin{align*}
	\pi(\phi_i,\psi_i) \propto  \frac{1}{\psi_i^{1/2}} \exp \left( -\frac{\lambda_{i}^0}{2 \psi_i} (\phi_i - \hat{\mu}_{i}^0)^2 \right) \frac{1}{\psi_i^{a_{i}^0+1}} \exp \left( - \frac{2b_{i}^0}{2 \psi_i} \right)
\end{align*}
with the parameters of the marginal inverse-gamma distribution set to $a^0_i = \frac{1}{2}$ and $b^0_i = 0$, the location parameter set to $\hat{\mu}^0_i = 0$, and the fourth parameter set to $\lambda^0_i = 0$. We use $n_0 = 3$ as the initial sample size. Under this prior, the posterior means and variances are equivalent to the sample means and variances. The performance measure $PFS^n_B \triangleq 1 - PCS^n_B$ can be estimated using Monte Carlo simulation, by generating samples from the posterior distribution of $\mu_i$ for each $i$ and computing the proportion of these simulations where the estimated best design has the largest actual value. We compute $2.5 \times 10^5$ such samples in each simulation run. Additionally, we report the frequentist performance measure $PFS^n_F \triangleq 1 - PCS^n_F$, which simply checks whether the design with the largest posterior mean coincides with $i^*$. The values of both performance metrics are averaged over $1,000$ macroreplications.

\begin{figure}
	\centering
	\caption{Performance comparison for $\ocbau$ and the benchmark methods in Instances 1-4.}
	\begin{minipage}{\textwidth}
		\includegraphics[width=1.0\textwidth]{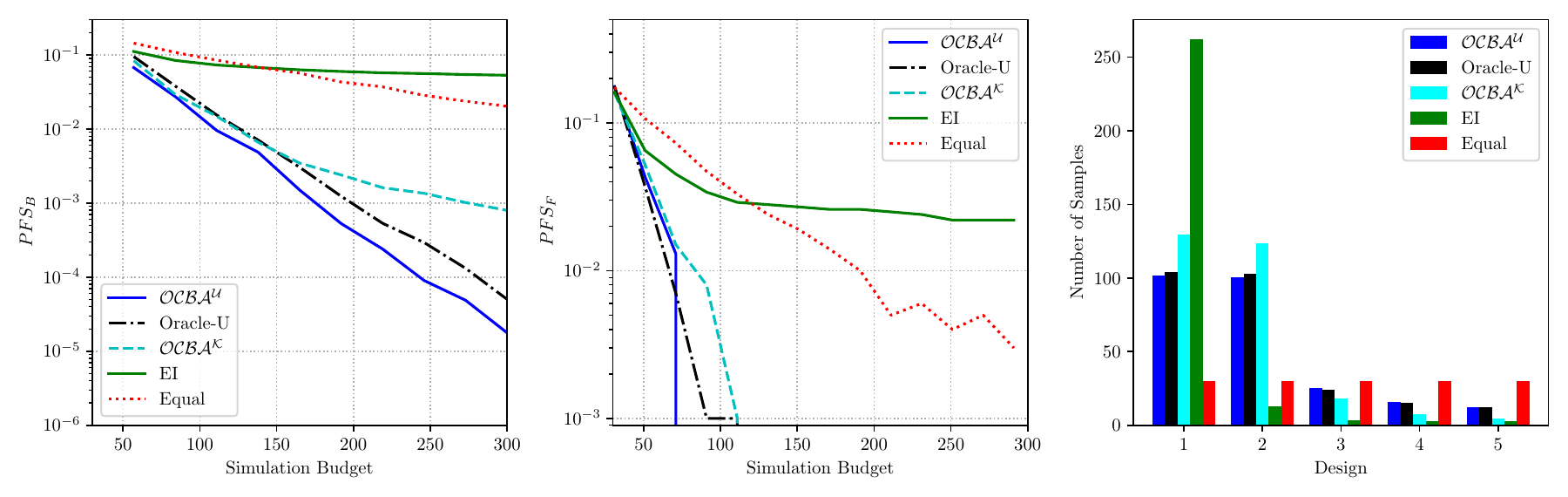}
	\end{minipage}
	
	\begin{minipage}{\textwidth}
		\includegraphics[width=1.0\textwidth]{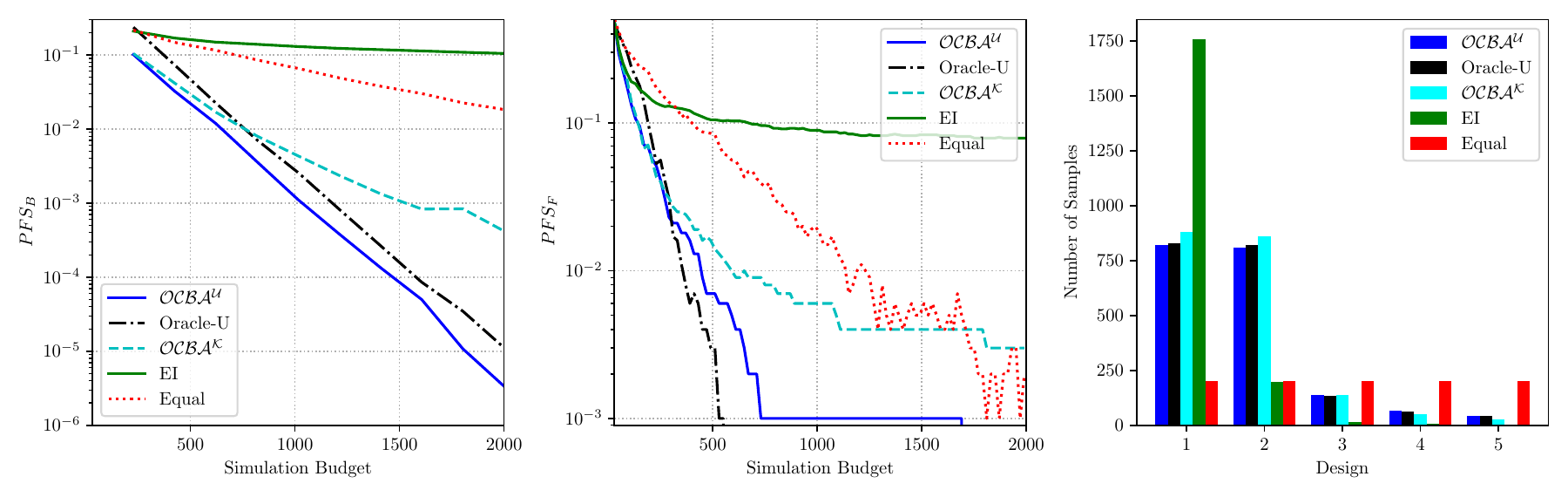}
	\end{minipage}
	
	\begin{minipage}{\textwidth}
		\includegraphics[width=1.0\textwidth]{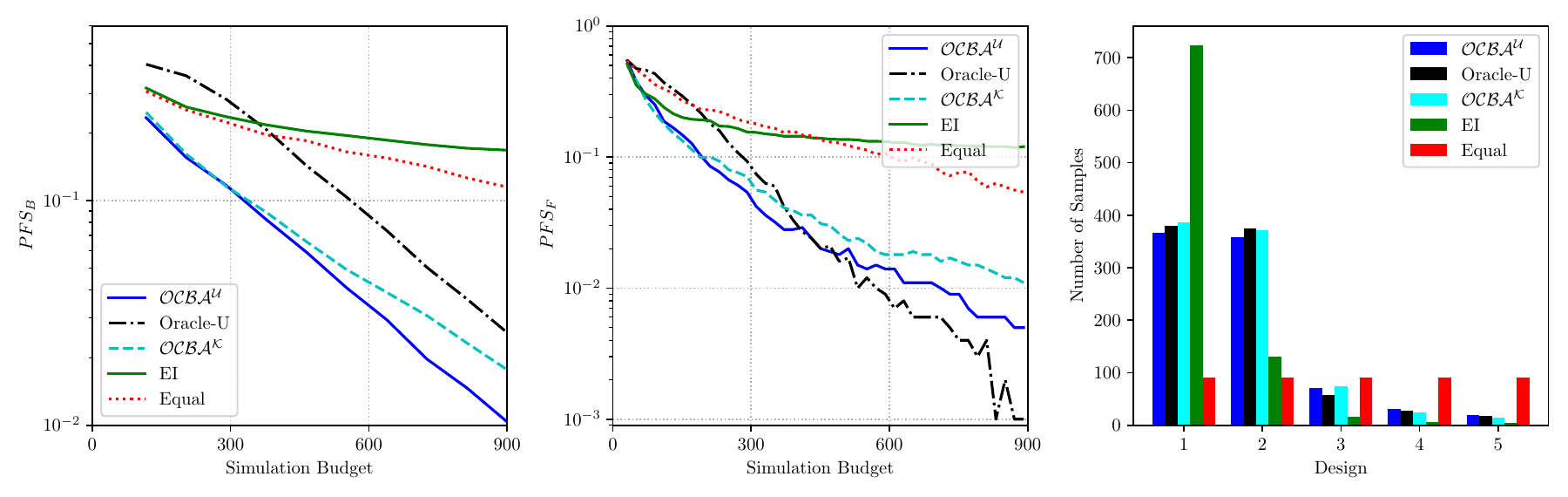}
	\end{minipage}
	
	\begin{minipage}{\textwidth}
		\includegraphics[width=1.0\textwidth]{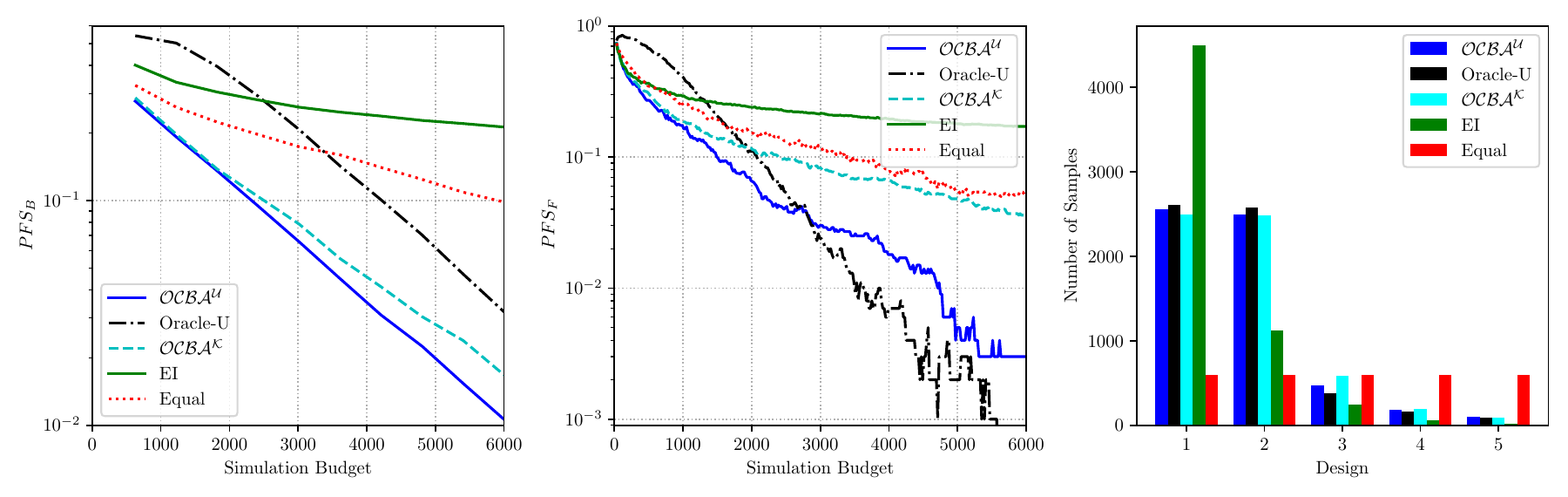}
	\end{minipage}
	\flushleft
	\vspace{0.2cm}
	{\small \emph{Notes. Rows 1-4 correspond to Instances 1-4, respectively.}}
	\label{fig:AlgTest}
\end{figure}

Figure \ref{fig:AlgTest} shows the trajectory for both types of PCS over time, as well as the final empirical allocation of the budget for designs 1-5. The total simulation budget $n$ is different for each of Instances $1$-$4$ due to the differences in difficulty between these problems. The most important observation is that $\ocbau$ significantly outperforms both $\mathcal{OCBA}^\mathcal{K}$ (which learns an optimal allocation, but for known variances) and EI (which learns a suboptimal allocation, but for unknown variances). Between these two, $\mathcal{OCBA}^\mathcal{K}$ performs much better, as EI tends to oversample the best design, consistently with the theory established in \cite{ryzhov2016convergence}. It is not surprising that $\ocbau$ outperforms $\mathcal{OCBA}^\mathcal{K}$ with respect to the Bayesian PCS, as that is the metric that it seeks to optimize, but it is interesting that $\ocbau$ mostly performs better even with respect to the \textit{frequentist} PCS, which is optimized by (\ref{glynn1})-(\ref{glynn2}). This happens because $\mathcal{OCBA}^\mathcal{K}$ does not actually know the sampling variances, it simply replaces them by plug-in estimators. This introduces additional statistical error that impedes the performance of the algorithm in finite sample, even though asymptotically it learns an allocation that is better for the frequentist PCS than the one produced by $\ocbau$. This clearly shows the value of including variance uncertainty in the allocation, even when the desired performance metric is not Bayesian.

Lastly, $\ocbau$ is generally comparable to the oracle allocation and even outperforms it sometimes, despite the fact that the oracle already knows the allocation that $\ocbau$ seeks to learn. This behavior has been observed before \citep{chen2006efficient} for OCBA methods, and arises when some designs have small values of $\alpha^*_i$. As a result, they receive very few samples under the oracle and their estimated values are unstable. Sequential procedures such as $\ocbau$ are somewhat more robust to this issue as they adapt to the observed values at every stage of sampling.

The third and final set of experiments is motivated by a realistic application setting, namely, the dose-finding problem \citep{yang2025stochastically,zhou2024sequential}. In this problem, the ``designs'' represent different dosage levels, and the effectiveness of each one can be assessed by conducting a clinical trial with an uncertain outcome. The mean and variance of the outcome are both unknown. The objective is to efficiently identify the best dose using a limited experimental budget.

\begin{figure}[htb]
	\centering
	\caption{Performance comparison for $\ocbau$ and the benchmark methods in Instances 5-6.}
	\begin{minipage}{\textwidth}
		\includegraphics[width=1.0\textwidth]{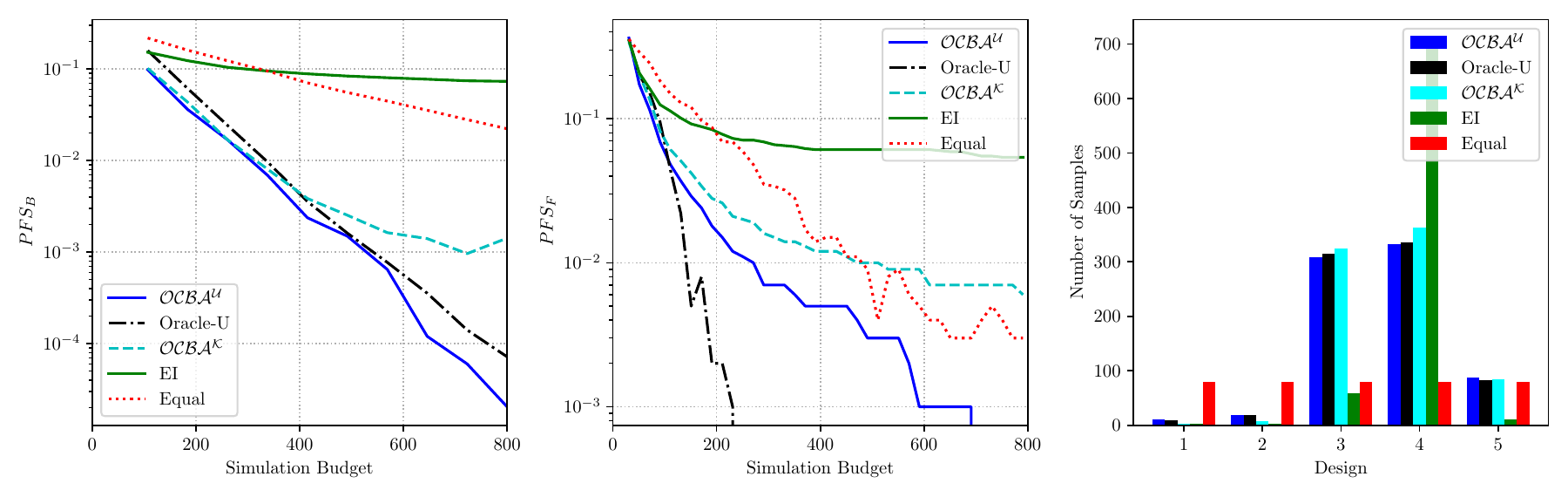}
	\end{minipage}
	
	\begin{minipage}{\textwidth}
		\includegraphics[width=1.0\textwidth]{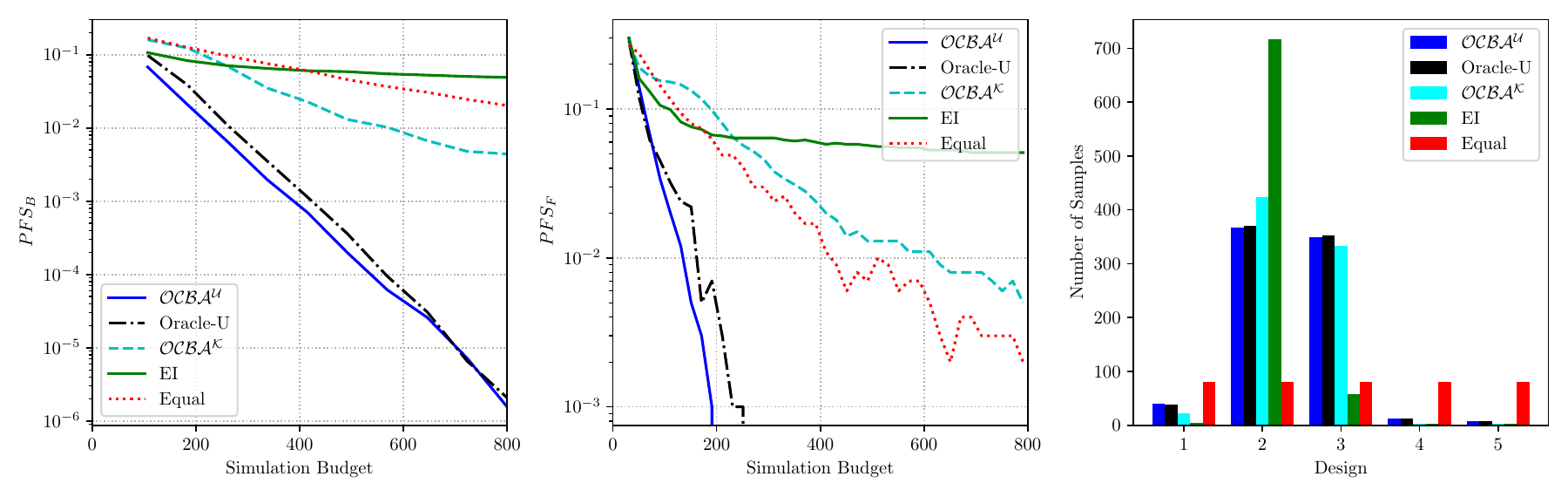}
	\end{minipage}
	\flushleft
	\vspace{0.2cm}
	{\small \emph{Notes. The first and second rows correspond to Instances 5 and 6, respectively.}}
	\label{fig:dose}
\end{figure}

We suppose that there are $10$ dosage levels to choose from. The effectiveness of dose $i$, $i=1,2,\dots,10$, is assumed to follow the Brain-Cousens model \citep{ritz2015dose} with parameter vector $\boldsymbol{c}=(c_1,c_2,\dots,c_5)$, so that
\begin{align}\label{eq:dose_eff}
	\mu_i = c_1+\frac{c_2-c_1+100c_3i}{1+\exp(c_5(\log(100i)-\log(c_4)))}.
\end{align}
The standard deviation for dose $i$ is set as $\sigma_i = 0.1 \mu_i$. We design two instances representing different patient groups:
\begin{itemize}
	\item Instance 5: $\boldsymbol{c}=(2,80,0.3,600,4)$, $i^* = 4$.
	\item Instance 6: $\boldsymbol{c}=(2,100,0.2,400,5)$, $i^*=2$.
\end{itemize}
As in the previous experiments, we apply $\ocbau$ and all benchmark methods to Instances 5 and 6, and present the numerical results in Figure~\ref{fig:dose}.

The observations from Figure \ref{fig:dose} are consistent with those in the second set of experiments: $\ocbau$ outperforms $\mathcal{OCBA}^\mathcal{K}$, EI and equal allocation under both Bayesian and frequentist PFS. The difference in performance is greater in Instance 6 as compared to Instance 5, and corresponds to a greater difference in the allocations learned by $\ocbau$ and $\mathcal{OCBA}^{\mathcal K}$. In fact, $\mathcal{OCBA}^\mathcal{K}$ does not do appreciably better than equal allocation in either instance (and even does worse in Instance 6), suggesting that the improvement achieved by $\ocbau$ comes not only from the differences in the limiting allocations but also in the sequence of sampling decisions made by each procedure. In other words, when the variance is unknown, we do better by adapting to imbalances in (\ref{eq:totalbalance})-(\ref{eq:individualbalance}) rather than in (\ref{glynn1})-(\ref{glynn2}).

% Figure~\ref{fig:dose} also shows that modeling variance uncertainty brings steady gains. $\ocbau$ puts more samples on doses near the current best where noise is higher, which facilitate the hardest pairwise comparisons. This set of experiments confirms that explicitly modeling variance uncertainty is useful beyond synthetic settings.

\section{Conclusions}

We have characterized the optimal budget allocation for R\&S problems where both the means and variances of the simulation output distributions are unknown. Past work has used the asymptotic convergence rate of the probability of correct selection as a performance measure to guide the allocation. However, the frequentist formulation of PCS in past work renders it incapable of distinguishing between known and unknown variance. We rectified this issue by considering a Bayesian formulation of PCS as the performance measure. We derived the rate function of this quantity, as well as the conditions describing the optimal allocation, and developed an efficient selection algorithm that provably learns this allocation without the need for tuning or forced exploration.

In our analysis, we identified several complications stemming from the presence of multiple uncertain parameters. The lack of differentiability properties that are taken for granted in the single-parameter case distinguishes our analysis from other work in this area, and is also likely to pose a challenge for future work on similar problems. We believe that the techniques developed in our paper will be of value in addressing other simulation optimization problems studied in the literature, such as subset selection \citep{chen2008,gao2016}, constrained selection \citep{lee2012}, contextual selection \citep{du2024contextual}, multi-objective selection \citep{lee2010}, and others. All of the prior work on these problems has treated the unknown sampling variance as known when developing optimal allocations.

\bibliographystyle{poms}
\bibliography{reference}

\newpage
\appendix
{\noindent \LARGE \textbf{Electronic Companion}}

\section{Proofs for Section \ref{sec:prior}}

Below, we prove Lemmas \ref{lem:bayes_consist} and \ref{lem:posterior_non}.

\subsection{Proof of Lemma \ref{lem:bayes_consist}}

The following facts will be used in the proof.

\begin{lemma}
	The following statements hold:
	\begin{itemize}
		\item If $N_i \geq 6$,
		\begin{equation}
			\frac{ \int_{\Rb_+} \int_{\Rb} (\phi_i - \bar{X}^n_i)^2  L^{n} (\phi_{i},\psi_{i}) d\phi_i d\psi_i }{  \int_{\Rb_+} \int_{\Rb}  L^{n} (\phi_{i}',\psi_{i}') d\phi_i' d\psi_i' } = \frac{\left(\tilde{S}^n_i\right)^2}{N_i-5}.\label{ineq:unif_prior_m}
		\end{equation}
		\item If $N_i \geq 8$,
		\begin{equation}
			\frac{ \int_{\Rb_+} \int_{\Rb} \left(\psi_i - \frac{N_i \left(\tilde{S}^n_i\right)^2}{N_i-5}\right)^2  L^{n} (\phi_{i},\psi_{i}) d\phi_i d\psi_i }{  \int_{\Rb_+} \int_{\Rb}  L^{n} (\phi_{i}',\psi_{i}') d\phi_i' d\psi_i' } =\frac{2 N_i^2 \left(\tilde{S}^n_i\right)^4 }{(N_i-5)^2(N_i-7)}.\label{ineq:unif_prior_v}
		\end{equation}
		\item Let $A_{\varepsilon,i} \triangleq \{ (\phi_i,\psi_i): |\phi_i-\bar{X}_i^n| < \varepsilon, |\psi_i - N_i(\tilde{S}_i^n)^2/(N_i-5)| < \varepsilon \}$. If $N_i \ge 8$,
		\begin{equation}
			\frac{ \int \int_{(\phi_i,\psi_i) \in A_{\varepsilon,i}} L^{n} (\phi_{i},\psi_{i}) d \phi_{i} d \psi_{i}  }{ \int_{\Rb_+} \int_{\Rb}  L^{n} (\phi_{i},\psi_{i}) d \phi_{i} d \psi_{i}   } \ge 1 - \frac{\left(\tilde{S}^n_i\right)^2}{\varepsilon^2 (N_i - 5)} - \frac{ 2 N_i^2 \left(\tilde{S}^n_i\right)^4 }{\varepsilon^2 (N_i-5)^2(N_i-7)}.\label{ineq:unif_prior_1}
		\end{equation}
	\end{itemize}
\end{lemma}

\begin{proof}{Proof:}
	Under a non-informative prior, the posterior for design $i$, given $N_i \ge 4$ samples, has the density $\frac{  L^{n} (\phi_{i},\psi_{i}) }{  \int_{\Rb_+} \int_{\Rb}  L^{n} (\phi_{i}',\psi_{i}') d\phi_i' d\psi_i' }$, which corresponds to the normal-inverse-gamma distribution
	$NIG(\bar{X}_i^n,N_i,$ $(N_i - 3)/2,N_i (\tilde{S}_i^n)^2/2 )$.
	Suppose the random vector $(\Phi_i,\Psi_i)$ follows this posterior distribution.
	The first fact is shown by noticing that $\frac{\sqrt{ N_i-3 }}{  \tilde{S}_i^n}   (\Phi_i - \bar{X}_i^n)$ follows Student's $t$-distribution with $N_i-3$ degrees of freedom. This $t$-distribution has mean $0$ and variance $\frac{N_i-3}{N_i-5}$, from which (\ref{ineq:unif_prior_m}) can be derived. The second fact is shown by noticing that $\Psi_i$ follows an inverse-gamma distribution with shape parameter $\frac{N_i-3}{2}$ and scale parameter $\frac{N_i}{2}(\tilde{S}^n_i)^2$. Such a $\Psi_i$ has mean $\frac{N_i (\tilde{S}^n_i)^2}{N_i-5}$ and variance $\frac{2 N_i^2 (\tilde{S}^n_i)^4 }{(N_i-5)^2(N_i-7)}$, which leads to (\ref{ineq:unif_prior_v}). Finally, (\ref{ineq:unif_prior_1}) is obtained by applying Chebyshev's inequality together with the first two facts. $\square$
\end{proof}

Consider a fixed sample path. Let $\Scal$ denote the set of designs such that $N_j$ does not diverge to infinity as $n \to \infty$. 		
For any given $0 < \varepsilon < \min\{1/(2k),\bar{\epsilon}/4\}$ where $\bar{\epsilon}$ has been introduced in \eqref{ineq:prior_reg} of the main text, with probability one, there exists $n_1$ large enough such that for all $j \notin \Scal$ and $n \ge n_1$,
\begin{enumerate}
	\item $|\mu_{j} - \bar{X}_{j}^n| \le \varepsilon$, $ |\sigma_{j}^2 - N_{j} (\tilde{S}_{j}^n)^2/ (N_{j}-5)| \le \varepsilon$ and $ (\tilde{S}_{j}^n)^2 \le 2 \sigma_{j}^2 $;
	
	\item $N_{j} \ge 15$ (which implies $N_{j} - 7 \ge N_{j}/2$), $N_{j} \ge 8 \sigma_{\max}^2/\varepsilon^3$  (which implies $(\tilde{S}_{j}^n)^2/(\varepsilon^2 (N_{j} - 5)) \le 4\sigma_{j}^2/(\varepsilon^2 N_{j}) \le \varepsilon/2$) and $N_{j} \ge 128  \sigma_{\max}^4  / \varepsilon^3$ (which implies $ 2 N_{j}^2 (\tilde{S}_{j}^n)^4 /(\varepsilon^2 (N_{j}-5)^2(N_{j}-7)) \le  64  \sigma_{j}^4  / (\varepsilon^2 N_{j}) \le \varepsilon/2$).
\end{enumerate}
Then by \eqref{ineq:unif_prior_1},
\begin{align}\label{ineq:unif_prior_t}
	\frac{\int \int_{(\phi_{j},\psi_{j}) \in A_{\varepsilon,j}} L^{n} (\phi_{j},\psi_{j}) d \psi_{j} d \phi_{j} }{\int_{\Rb_+} \int_{\Rb} L^{n} (\phi_{j},\psi_{j}) d \phi_{j} d \psi_{j} } \ge 1 - \frac{(\tilde{S}_{j}^n)^2}{\varepsilon^2 (N_{j} - 5)} - \frac{ 2 N_{j}^2 (\tilde{S}_{j}^n)^4 }{\varepsilon^2 (N_{j}-5)^2(N_{j}-7)} \ge 1-\varepsilon.
\end{align}
Let $n_2$ be $n_1$ if $\Scal = \emptyset$ and $n_2$ be large enough such that $N_{j}$ remains unchanged for all $j \in \Scal$ if $\Scal \ne \emptyset$. Consider any $n \ge n_2$. Let
$A_{\varepsilon,0} \triangleq [\mu_{\min}, \mu_{\max}] \times [ \sigma_{\min}^2, \sigma_{\max}^2 ]$,
\begin{align*}
	&\tilde{A}_{\varepsilon,j} \triangleq \left\{ \begin{array}{ll}
		A_{\varepsilon,j}, & \text{if } j \notin \Scal \\
		A_{\varepsilon,0},  & \text{if } j \in \Scal,
	\end{array} \right.  \
	c_B \triangleq \left\{ \begin{array}{ll}
		1/2, & \text{if } \Scal = \emptyset \\
		1/2 \prod_{j \in \Scal} \frac{ \int \int_{(\phi_{j},\psi_{j}) \in \tilde{A}_{\varepsilon,j}} L^{n_2} (\phi_{j},\psi_{j}) d \psi_{j} d \phi_{j} }{\int_{\Rb_+} \int_{\Rb} L^{n_2} (\phi_{j},\psi_{j}) d \phi_{j} d \psi_{j}}, & \text{if } \Scal \ne \emptyset.
	\end{array} \right.
\end{align*}
Let $A_{\varepsilon} = \{(\boldsymbol{\phi},\boldsymbol{\psi}): (\phi_j,\psi_j) \in \tilde A_{\varepsilon,j},j=1,\dots,k\}$. Notice that $ \varepsilon \le 1/(2k) $ such that $(1-\varepsilon)^k \ge 1-k\varepsilon \ge 1/2$.
If $\Scal = \emptyset$, we have
\begin{align}
	\frac{ \int \int_{(\boldsymbol{\phi},\boldsymbol{\psi}) \in A_{\varepsilon} }  \prod_{j=1}^{k} L^{n} (\phi_{j},\psi_{j}) d\boldsymbol{\phi} d\boldsymbol{\psi} }{ \int_{\Rb^k} \int_{\Rb^k_+}  \prod_{j=1}^{k} L^{n} (\phi_{j}',\psi_{j}') d\boldsymbol{\phi}' d\boldsymbol{\psi}' }  = \prod_{j=1}^k \frac{\int \int_{(\phi_j,\psi_j) \in A_{\varepsilon,j}} L^{n} (\phi_{j},\psi_{j}) d \psi_j d \phi_j }{\int_{\Rb_+} \int_{\Rb} L^{n} (\phi_{j},\psi_{j}) d \phi_{j} d \psi_{j} }
	\ge& (1-\varepsilon)^k\label{eq:techpostprob1}\\
	\ge& c_B,\label{ineq:post_prob_1}
\end{align}
where (\ref{eq:techpostprob1}) holds by \eqref{ineq:unif_prior_t}.
If $\Scal \ne \emptyset$, we have
\begin{align}
	\frac{ \int \int_{(\boldsymbol{\phi},\boldsymbol{\psi}) \in A_{\varepsilon} }  \prod_{j=1}^{k} L^{n} (\phi_{j},\psi_{j}) d\boldsymbol{\phi} d\boldsymbol{\psi} }{ \int_{\Rb^k} \int_{\Rb^k_+}  \prod_{j=1}^{k} L^{n} (\phi_{j}',\psi_{j}') d\boldsymbol{\phi}' d\boldsymbol{\psi}' }
	\ge (1-\varepsilon)^{k} \prod_{j \in \Scal} \frac{\int \int_{(\phi_j,\psi_j) \in \tilde{A}_{\varepsilon,j}} L^{n} (\phi_{j},\psi_{j}) d \psi_j d \phi_j }{\int_{\Rb_+} \int_{\Rb} L^{n} (\phi_{j},\psi_{j}) d \phi_{j} d \psi_{j} } \ge c_B. \label{ineq:post_prob_11}
\end{align}
Let $n_3$ be large enough such that for $j \notin \Scal$ and $n \ge n_3$, we have $N_j \ge  64\bar{c}\sigma_{\max}^2/(\underline{c}c_B\varepsilon^2)$ implying
\begin{align}
	&\bar{c} (\tilde{S}_j^n)^2/(\underline{c}c_B(N_j-5)) \le 4\bar{c} \sigma_{\max}^2/(\underline{c}c_BN_j) \le \varepsilon^2,  \label{ineq:post_n3_1} \\
	&2 \bar{c} N_j^2 (\tilde{S}_j^n)^4 /(\underline{c} c_B (N_j-5)^2(N_j-7)) \le 64 \bar{c}  \sigma_{\max}^4 / (\underline{c}c_B N_j) \le \varepsilon^2. \label{ineq:post_n3_2}
\end{align}

Consider any $n \ge n_3$.
For $j \notin \Scal$ and $(\phi_j, \psi_j) \in A_{\varepsilon,j}$, since $|\mu_j - \bar{X}_j^n| \le \varepsilon$, $ |\sigma_j^2 - N_j (\tilde{S}_j^n)^2/ (N_j-5)| \le \varepsilon$ and $\varepsilon \le \bar{\epsilon}/4$, we have  $(\phi_j, \psi_j) \in [\mu_{\min} - \bar{\epsilon}, \mu_{\max} + \bar{\epsilon}] \times [ \sigma_{\min}^2 - \bar{\epsilon}, \sigma_{\max}^2 + \bar{\epsilon} ]$. Meanwhile, for $j \in \Scal$ and $(\phi_j, \psi_j) \in A_{\varepsilon,0}$, it is straightforward that $(\phi_j, \psi_j) \in [\mu_{\min} - \bar{\epsilon}, \mu_{\max} + \bar{\epsilon}] \times [ \sigma_{\min}^2 - \bar{\epsilon}, \sigma_{\max}^2 + \bar{\epsilon} ]$ by definition. Then $A_{\varepsilon} \subset H_w$ where $H_w$ is defined in (11), which leads to $\pi^{0}(\boldsymbol{\phi},\boldsymbol{\psi}) \ge \underline{c}$ for $(\boldsymbol{\phi},\boldsymbol{\psi}) \in A_{\varepsilon}$. Then, for any $i \notin \Scal$,
\begin{align}
	&\Eb_B[ (\mu_i - \bar{X}_i^n)^2  ] = \frac{ \int_{\Rb_+^k} \int_{\Rb^k} (\phi_i-\bar{X}_i^n)^2 \pi^{0}(\boldsymbol{\phi},\boldsymbol{\psi})  \prod_{j=1}^{k} L^{n} (\phi_{j},\psi_{j}) d\boldsymbol{\phi} d\boldsymbol{\psi} }{ \int_{\Rb_+^k} \int_{\Rb^k} \pi^{0}(\boldsymbol{\phi}',\boldsymbol{\psi}') \prod_{j=1}^{k} L^{n} (\phi_{j}',\psi_{j}') d\boldsymbol{\phi}' d\boldsymbol{\psi}' } \nonumber\\
	\le& \bar{c}/\underline{c}  \frac{ \int_{\Rb_+^k} \int_{\Rb^k} (\phi_i-\bar{X}_i^n)^2  \prod_{j=1}^{k} L^{n} (\phi_{j},\psi_{j}) d\boldsymbol{\phi} d\boldsymbol{\psi} }{ \int \int_{(\boldsymbol{\phi},\boldsymbol{\psi}) \in A_{\varepsilon}}  \prod_{j=1}^{k} L^{n} (\phi_{j}',\psi_{j}') d\boldsymbol{\phi}' d\boldsymbol{\psi}' } \nonumber\\
	\le& \frac{\bar{c}(\tilde{S}_i^n)^2}{\underline{c}(N_i-5)} \frac{1}{c_B} \label{eq:techlong1}\\
	\le& \varepsilon^2,\label{eq:techlong2}
\end{align}
where (\ref{eq:techlong1}) holds by \eqref{ineq:unif_prior_m}, \eqref{ineq:post_prob_1} and \eqref{ineq:post_prob_11}, and \eqref{eq:techlong2} holds by \eqref{ineq:post_n3_1}.
By H\"older inequality,
$\Eb_B[ |\mu_i - \bar{X}_i^n|  ] \le \varepsilon
$,
which leads to $|\hat\mu_i^n - \bar{X}_i^n| = |\Eb_B (\mu_i) - \bar{X}_i^n| \le \varepsilon$.
Similarly,
\begin{eqnarray}
	\Eb_B[ (\sigma_i^2 - \frac{N_i}{N_i-5} (\tilde{S}_i^n)^2)^2  ]
	&\le& \bar{c}/\underline{c}  \frac{ \int_{\Rb_+^k} \int_{\Rb^k}  (\psi_i-\frac{N_i}{N_i-5} (\tilde{S}_i^n)^2)^2  \prod_{j=1}^{k} L^{n} (\phi_{j},\psi_{j}) d\boldsymbol{\phi} d\boldsymbol{\psi} }{ \int \int_{(\boldsymbol{\phi},\boldsymbol{\psi}) \in A_{\varepsilon}}  \prod_{j=1}^{k} L^{n} (\phi_{j}',\psi_{j}') d\boldsymbol{\phi}' d\boldsymbol{\psi}' }  \nonumber\\
	&\le& \frac{2 \bar{c} N_i^2 (\tilde{S}_i^n)^4 }{\underline{c}(N_i-5)^2(N_i-7)} \frac{1}{c_B} \label{eq:techlong3}\\
	&\le& \varepsilon^2,\label{eq:techlong4}
\end{eqnarray}
where (\ref{eq:techlong3}) holds by \eqref{ineq:unif_prior_v}, \eqref{ineq:post_prob_1} and \eqref{ineq:post_prob_11}, and \eqref{eq:techlong4} holds holds by \eqref{ineq:post_n3_2}. By the H\"older inequality,
$\Eb_B[ |\sigma_i^2 - \frac{N_i}{N_i-5} (\tilde{S}_i^n)^2|  ] \le \varepsilon$,
which leads to $|(\hat\sigma_i^n)^2 - \frac{N_i}{N_i-5} (\tilde{S}_i^n)^2|  \le \varepsilon$. $\square$

\subsection{Proof of Lemma \ref{lem:posterior_non}}

To simplify the notation, we assume $k=2$ in this proof such that $\boldsymbol{\phi}=(\phi_1,\phi_2)$ and $\boldsymbol{\psi}=(\psi_1,\psi_2)$. Suppose $N_1 \ge 6$ and $N_1$ does not increase to infinity as $n \to \infty$.   Meanwhile, suppose $N_2 \to \infty$ as $n \to \infty$. Let $c_d \triangleq 1/\int_{\Rb_+} \int_{\Rb} L^{n} (\phi_{2},\psi_{2}) d \phi_{2} d \psi_{2} $. We have by \eqref{ineq:unif_prior_t} that, when $n$ is large enough, $c_d \int \int_{(\phi_{2},\psi_{2}) \notin A_{\varepsilon,2}} L^{n} (\phi_{2},\psi_{2}) d \phi_{2} d \psi_{2}  \le \varepsilon$.
For any fixed value of $(\phi_1,\psi_1)$, we have
\begin{align*}
	&\left| c_d \int \int_{(\phi_{2},\psi_{2}) \notin A_{\varepsilon,2}} (\pi^0(\phi_1,\phi_2,\psi_1,\psi_2)-\pi^0(\phi_1,\mu_2,\psi_1,\sigma_2^2)) L^{n} (\phi_{2},\psi_{2}) d \phi_{2} d \psi_{2}  \right| \\
	\le& \bar{c} c_d \int \int_{(\phi_{2},\psi_{2}) \notin A_{\varepsilon,2}}  L^{n} (\phi_{2},\psi_{2}) d \phi_{2} d \psi_{2}  \le \bar{c} \varepsilon.
\end{align*}
On the other hand, by the uniform continuity of $\pi^0(\boldsymbol{\phi},\boldsymbol{\psi})$, we have when $(\phi_{2},\psi_{2}) \in A_{\varepsilon,2}$ that $|\pi^{0}(\phi_1,\phi_2,\psi_1,\psi_2) - \pi^{0}(\phi_1,\mu_{2},\psi_1,\sigma_2^2)| \le b_{\pi} \varepsilon$,
where $b_{\pi} $ is a constant, which leads to
\begin{align*}
	& \left| c_d \int \int_{(\phi_{2},\psi_{2}) \in A_{\varepsilon,2}} (\pi^0(\phi_1,\phi_2,\psi_1,\psi_2)-\pi^0(\phi_1,\mu_2,\psi_1,\sigma_2^2)) L^{n} (\phi_{2},\psi_{2}) d \phi_{2} d \psi_{2}  \right|   \le b_{\pi} \varepsilon.
\end{align*}
Combining the above two inequalities, $c_d \int_{\Rb_+} \int_{\Rb}  \pi^0(\phi_1,\phi_2,\psi_1,\psi_2) L^{n} (\phi_{2},\psi_{2}) d \phi_{2} d \psi_{2}$ can be bounded by $\pi^0(\phi_1,\mu_2,\psi_1,\sigma_2^2) \pm (b_{\pi}+\bar{c}) \varepsilon$.
Thus
\begin{align*}
	&\Big| c_d \int_{\Rb_+^2} \int_{\Rb^2} \phi_1 \pi^0(\boldsymbol{\phi},\boldsymbol{\psi}) \prod_{j=1}^k L^{n} (\phi_{j},\psi_{j}) d\boldsymbol{\phi} d\boldsymbol{\psi} - \int_{\Rb_+} \int_{\Rb}  \phi_1 \pi^0(\phi_1,\mu_2,\psi_1,\sigma_2^2) L^{n} (\phi_{1},\psi_{1})   d \phi_{1} d \psi_{1} \Big|  \\
	\le& (b_{\pi}+\bar{c}) \varepsilon \int_{\Rb_+} \int_{\Rb} \phi_1 L^{n} (\phi_{1},\psi_{1}) d \phi_{1} d \psi_{1}.
\end{align*}
Similarly,
\begin{align*}
	&\Big| c_d \int_{\Rb_+^2} \int_{\Rb^2}  \pi^0(\boldsymbol{\phi},\boldsymbol{\psi}) \prod_{j=1}^k L^{n} (\phi_{j},\psi_{j}) d\boldsymbol{\phi} d\boldsymbol{\psi} - \int_{\Rb_+} \int_{\Rb}   \pi^0(\phi_1,\mu_2,\psi_1,\sigma_2^2) L^{n} (\phi_{1},\psi_{1})   d \phi_{1} d \psi_{1} \Big|  \\
	\le& (b_{\pi}+\bar{c}) \varepsilon \int_{\Rb_+} \int_{\Rb}  L^{n} (\phi_{1},\psi_{1}) d \phi_{1} d \psi_{1}.
\end{align*}
Since $N_1 \ge 6$, both $\int_{\Rb_+} \int_{\Rb} \phi_1 L^{n} (\phi_{1},\psi_{1}) d \phi_{1} d \psi_{1}$ and $\int_{\Rb_+} \int_{\Rb}  L^{n} (\phi_{1},\psi_{1}) d \phi_{1} d \psi_{1}$ are finite because
$L^{n} (\phi_{1},\psi_{1}) / \int_{\Rb_+} \int_{\Rb}  L^{n} (\phi_{1},\psi_{1}) d \phi_{1} d \psi_{1}$
is the density function of  $ NIG(\bar{X}_i^n,N_i,(N_i - 3)/2,N_i (\tilde{S}_i^n)^2/2 )$ whose expectations are finite. Then both $(b_{\pi}+\bar{c}) \varepsilon \int_{\Rb_+} \int_{\Rb} \phi_1 L^{n} (\phi_{1},\psi_{1}) d \phi_{1} d \psi_{1}$ and $(b_{\pi}+\bar{c}) \varepsilon \int_{\Rb_+} \int_{\Rb}$ $  L^{n} (\phi_{1},\psi_{1}) d \phi_{1} d \psi_{1}$ can be very small for $\varepsilon$ small enough and $n$ large enough. Thus, when $n \to \infty$, the posterior mean of design 1 satisfies
\begin{align*}
	&\frac{ \int_{\Rb_+^k} \int_{\Rb^k} \phi_1 \pi^{0}(\boldsymbol{\phi},\boldsymbol{\psi})  \prod_{j=1}^{k} L^{n} (\phi_{j},\psi_{j}) d\boldsymbol{\phi} d\boldsymbol{\psi} }{ \int_{\Rb_+^k} \int_{\Rb^k} \pi^{0}(\boldsymbol{\phi}',\boldsymbol{\psi}') \prod_{j=1}^{k} L^{n} (\phi_{j}',\psi_{j}') d\boldsymbol{\phi}' d\boldsymbol{\psi}' } \to  \frac{\int_{\Rb_+} \int_{\Rb} \phi_1 \pi^0(\phi_1,\mu_2,\psi_1,\sigma_2^2) L^{n} (\phi_{1},\psi_{1}) d \phi_{1} d \psi_{1}}{ \int_{\Rb_+} \int_{\Rb}  \pi^0(\phi_1,\mu_2,\psi_1,\sigma_2^2) L^{n} (\phi_{1},\psi_{1}) d \phi_{1} d \psi_{1} }.
\end{align*}
Similarly, when $n \to \infty$, the posterior variance of design 1 satisfies
\begin{align*}
	&\frac{ \int_{\Rb_+^k} \int_{\Rb^k} \psi_1 \pi^{0}(\boldsymbol{\phi},\boldsymbol{\psi})  \prod_{j=1}^{k} L^{n} (\phi_{j},\psi_{j}) d\boldsymbol{\phi} d\boldsymbol{\psi} }{ \int_{\Rb_+^k} \int_{\Rb^k} \pi^{0}(\boldsymbol{\phi}',\boldsymbol{\psi}') \prod_{j=1}^{k} L^{n} (\phi_{j}',\psi_{j}') d\boldsymbol{\phi}' d\boldsymbol{\psi}' } \to  \frac{\int_{\Rb_+} \int_{\Rb} \psi_1 \pi^0(\phi_1,\mu_2,\psi_1,\sigma_2^2) L^{n} (\phi_{1},\psi_{1}) d \phi_{1} d \psi_{1}}{ \int_{\Rb_+} \int_{\Rb}  \pi^0(\phi_1,\mu_2,\psi_1,\sigma_2^2) L^{n} (\phi_{1},\psi_{1}) d \phi_{1} d \psi_{1} }. \ \square
\end{align*}

\section{Proof of Theorem \ref{thm:pcs}}\label{sec:appendixsec1}

We first present several technical lemmas that will be used in the proof. Below, we state these results; the proofs are given in Section \ref{sec:proof_appendixsec1}. For notational simplicity, let $\Delta \triangleq \min \{\min_{j \ne i^*} (\mu_{i^*} - \mu_{j}) / 8 ,  \sigma_{\min}^2/4 \}$.

\begin{lemma}\label{lem:mle_eoc_z}
	Let $\varepsilon$ denote any small positive constant satisfying $0< \varepsilon < \Delta$. For $i \ne i^*$, $\bar{\psi}_i \in [\sigma_{i}^2 - \varepsilon, \sigma_{i}^2+\varepsilon]$, $\bar{\psi}_{i^*} \in [\sigma_{i^*}^2 - \varepsilon, \sigma_{i^*}^2+\varepsilon]$, $\bar{\phi}_i \in [\mu_i-\varepsilon,\mu_i+\varepsilon]$, $\bar{\phi}_{i^*} \in [\mu_{i^*}-\varepsilon,\mu_{i^*}+\varepsilon]$, $\phi_i \in \Rb$ and $\phi_{i^*} \in \Rb$, define functions
	\begin{align}
		&w(\bar{\psi}_{i},\bar{\psi}_{i^*},\bar{\phi}_{i},\bar{\phi}_{i^*},\phi_{i},\phi_{i^*}) \triangleq  \frac{\alpha_{i}}{2} \log (1 + (\bar{\phi}_{i}-\phi_{i})^2/\bar{\psi}_{i}) + \frac{\alpha_{i^*}}{2} \log (1 + (\bar{\phi}_{i^*}-\phi_{i^*})^2/\bar{\psi}_{i^*}),  \nonumber \\
		&a(\bar{\psi}_{i},\bar{\psi}_{i^*},\bar{\phi}_{i},\bar{\phi}_{i^*}) \triangleq  -\min_{\phi_{i}-\phi_{i^*} \ge 0} w(\bar{\psi}_{i},\bar{\psi}_{i^*},\bar{\phi}_{i},\bar{\phi}_{i^*},\phi_{i},\phi_{i^*}).  \label{ineq:eoc_num_z}
	\end{align}
	Let $(\phi_{i}^{\min,i}(\bar{\psi}_{i},\bar{\psi}_{i^*},\bar{\phi}_{i},\bar{\phi}_{i^*}), \phi_{i^*}^{\min,i}(\bar{\psi}_{i},\bar{\psi}_{i^*},\bar{\phi}_{i},\bar{\phi}_{i^*}))$ denote an optimal solution to
	\begin{align*}
		\min_{\phi_{i},\phi_{i^*}: \  \phi_{i} \ge \phi_{i^*}} w(\bar{\psi}_{i},\bar{\psi}_{i^*},\bar{\phi}_{i},\bar{\phi}_{i^*},\phi_{i},\phi_{i^*}).
	\end{align*}
	The optimal solution satisfies
	\begin{align}
		&\phi_{i}^{\min,i}(\bar{\psi}_{i},\bar{\psi}_{i^*},\bar{\phi}_{i},\bar{\phi}_{i^*}) = \phi_{i^*}^{\min,i}(\bar{\psi}_{i},\bar{\psi}_{i^*},\bar{\phi}_{i},\bar{\phi}_{i^*}), \label{ineq:phiiistar}  \\
		&\phi_{i^*}^{\min,i}(\bar{\psi}_{i},\bar{\psi}_{i^*},\bar{\phi}_{i},\bar{\phi}_{i^*}),\phi_{i}^{\min,i}(\bar{\psi}_{i},\bar{\psi}_{i^*},\bar{\phi}_{i},\bar{\phi}_{i^*}) \in [ \mu_i - \varepsilon, \mu_{i^*}+\varepsilon ] . \label{ineq:eoc_phi_ep}
	\end{align}
	We can find $b_a > 0$ such that
	\begin{align}
		&\left| a(\bar{\psi}_{i},\bar{\psi}_{i^*},\bar{\phi}_{i},\bar{\phi}_{i^*}) - a(\sigma_{i}^2,\sigma_{i^*}^2,\mu_{i},\mu_{i^*}) \right| \le b_{a} \varepsilon.   \label{ineq:z_est}
	\end{align}
\end{lemma}

The next technical lemma uses the set
$
\Xi_{i} \triangleq \{ (\boldsymbol{\phi},\boldsymbol{\psi}) \in \Rb^{k} \times \Rb^{k}_+: \phi_i \ge \phi_{i^*} \}
$, $i \ne i^*$,
that has been defined in Section 3.2 of the main text.

\begin{lemma}\label{lem:mle}
	The following facts about the maximum likelihood estimation (MLE) hold:
	\begin{itemize}
		\item The optimal solution (denoted by $(\boldsymbol{\phi}^{*,i^*},\boldsymbol{\psi}^{*,i^*})$) to $\max_{(\boldsymbol{\phi}, \boldsymbol{\psi}) \in \Rb^{k} \times \Rb^{k}_+} \sum_{i=1}^k \sum_{l=1}^{N_i} \log f (X_{il} | \phi_{i},\psi_{i})$
		is $\boldsymbol{\phi}^{*,i^*} = \bar{\boldsymbol{X}}^n$ and $\boldsymbol{\psi}^{*,i^*} = (\tilde{\boldsymbol{S}}^n)^2$. Moreover,
		\begin{align*}
			&\max_{(\boldsymbol{\phi}, \boldsymbol{\psi}) \in \Rb^{k} \times \Rb^{k}_+} \sum_{i=1}^k \sum_{l=1}^{N_i} \log f (X_{il} | \phi_{i},\psi_{i})
			= - \frac{n}{2} \left( \log (2\pi) + 1 \right) - \sum_{i=1}^k \frac{N_i}{2} \log (\tilde{S}_i^n)^2.
		\end{align*}
		
		\item Suppose $n$ is large enough such that  $| \bar{X}_j^n - \mu_j | \le \varepsilon$ and $| (\tilde{S}_j^n)^2 - \sigma_j^2 | \le \varepsilon$, $j=1,\dots,k$, where $\varepsilon$ is any small positive constant satisfying $\varepsilon < \Delta$. For $i \ne i^*$ and $0 \le \delta < \Delta$, the constrained maximum likelihood estimation problem can be simplified with
		\begin{align*}
			\max_{(\boldsymbol{\phi}, \boldsymbol{\psi}) \in \Xi_{i}} \sum_{j=1}^k \sum_{l=1}^{N_j} \log f &(X_{jl} |  \phi_{j},\psi_{j})
			= - \frac{n}{2} \left( \log (2\pi) + 1 \right) - \sum_{j\ne i^*, i}  \left( \frac{N_j}{2}\log  (\tilde{S}_j^n)^2   \right) \\
			& - \min_{ \phi_{i} } \left(   \frac{N_{i}}{2}\log \left( (\tilde{S}_{i}^n)^2 +  (\bar{X}_{i}^n-\phi_{i})^2 \right)    + \frac{N_{i^*}}{2}\log \left( (\tilde{S}_{i^*}^n)^2 + (\bar{X}_{i^*}^n-\phi_{i})^2 \right)   \right)
		\end{align*}
		and the optimal solution (denoted by $(\boldsymbol{\phi}^{*,i},\boldsymbol{\psi}^{*,i})$) to this constrained MLE satisfies
		\begin{itemize}			
			\item $\phi_{i}^{*,i} = \phi_{i^*}^{*,i}$ and $\phi_{i}^{*,i}, \phi_{i^*}^{*,i} \in  [\mu_{i} - \varepsilon, \mu_{i^*} + \varepsilon]$,
			
			\item $\psi_{i}^{*,i}  = (\tilde{S}_{i}^n)^2 + (\bar{X}_{i}^n - \phi_{i}^{*,i})^2$,
			
			\item $\psi_{i^*}^{*,i} = (\tilde{S}_{i^*}^n)^2 + (\bar{X}_{i^*}^n - \phi_{i^*}^{*,i})^2$,
			
			\item $\phi_j^{*,i} = \bar{X}_j^n$ and $\psi_j^{*,i} = (\tilde{S}_j^n)^2$ for $j \ne i^*,i$.
		\end{itemize}
	\end{itemize}
\end{lemma}

The following technical lemma identifies a subset of $\Xi_{i}$ in which the log-likelihood is close to its optimal value.

\begin{lemma}\label{lem:xi_ai}
	Let $b_{\iota}$ denote a large constant. Suppose $n$ is large enough such that  $| \bar{X}_j^n - \mu_j | \le \varepsilon$ and $|(\tilde{S}_j^n)^2 - \sigma_j^2| \le \varepsilon$, $j=1,\dots,k$, where $\varepsilon$ is any small positive constant satisfying $\varepsilon < \min \{\bar{\epsilon}/(2(b_{\iota}+1)), \Delta  \}$.
	\begin{itemize}
		\item If $( \boldsymbol{\phi}, \boldsymbol{\psi}) \in H_{i^*}$ with
		$H_{i^*} \triangleq \{ ( \boldsymbol{\phi}, \boldsymbol{\psi}) \in \Rb^k \times \Rb_+^k: \| ( \boldsymbol{\phi}, \boldsymbol{\psi}) -  ( \boldsymbol{\phi}^{*,i^*}, \boldsymbol{\psi}^{*,i^*})\|_{\infty} \le b_{\iota} \varepsilon \}$,
		then $( \boldsymbol{\phi}, \boldsymbol{\psi}) \in H_w$, the volume of $H_{i^*}$ is $(2b_{\iota}\varepsilon )^{2k}$ and
		\begin{align}\label{ineq:his}
			\frac{1}{n}\sum_{j=1}^k \sum_{l=1}^{N_j} \log f (X_{jl} | \phi_{j}^{*,i^*},\psi_{j}^{*,i^*}) - \frac{1}{n}\sum_{j=1}^k \sum_{l=1}^{N_j} \log f (X_{jl} | \phi_{j},\psi_{j}) \le \varepsilon.
		\end{align}
		
		\item For $i \ne i^*$, if $( \boldsymbol{\phi}, \boldsymbol{\psi}) \in H_{i}$ with
		$H_{i} \triangleq \{ ( \boldsymbol{\phi}, \boldsymbol{\psi}) \in \Xi_{i}: \| ( \boldsymbol{\phi}, \boldsymbol{\psi}) -  ( \boldsymbol{\phi}^{*,i}, \boldsymbol{\psi}^{*,i})\|_{\infty} \le b_{\iota} \varepsilon \}$,
		then $( \boldsymbol{\phi}, \boldsymbol{\psi}) \in H_w$, the volume of $H_{i}$ is $(2b_{\iota}\varepsilon )^{2k}/2$ and
		\begin{align}\label{ineq:hi}
			\frac{1}{n}\sum_{j=1}^k \sum_{l=1}^{N_j} \log f (X_{jl} | \phi_{j}^{*,i},\psi_{j}^{*,i}) - \frac{1}{n}\sum_{j=1}^k \sum_{l=1}^{N_j} \log f (X_{jl} | \phi_{j},\psi_{j}) \le \varepsilon.
		\end{align}
	\end{itemize}
\end{lemma}

Note that, by Lemma \ref{lem:xi_ai}, both $H_{i}$ and $H_{i^*}$ are subsets of $H_w$, so the lower bound $\underline{c}$ on the prior is valid. Moreover, the volumes of these subsets are independent of $n$. Now we can complete the proof of Theorem \ref{thm:pcs} as follows.

For $\varepsilon \le \min\{\bar{\epsilon}/(2(b_{\iota}+1)), \Delta,1/(2k), \bar{\epsilon}/4\}$, suppose $n$ is large enough such that $| \bar{X}^n_j - \mu_j | \le \varepsilon$ and $|(\tilde{S}^n_j)^2 - \sigma_j^2| \le \varepsilon$ for each design $j$ and the results from Lemma \ref{lem:bayes_consist} can be applied. Then  $|\hat{\mu}^n_{j} - \bar{X}^n_j| \le \varepsilon$, which together with $| \bar{X}^n_j - \mu_j | \le \varepsilon$ leads to $|\hat{\mu}^n_{j} - \mu_j| \le 2 \varepsilon$ for each design $j$. Since $\Delta\le \min_{i \ne i^*} (\mu_{i^*} - \mu_{i}) / 8 $, we have $\hat{\mu}^n_j \le \mu_j + (\mu_{i^*} - \mu_{j}) / 4 < \mu_{i^*} - (\mu_{i^*} - \mu_{j}) / 4 \le \hat{\mu}^n_{i^*}$. Then $i^{*,n} = i^*$.

Since $| \bar{X}^n_j - \mu_j | \le \varepsilon$ and $|(\tilde{S}_j^n)^2 - \sigma_j^2| \le \varepsilon$, $j=1,\dots,k$, the results from Lemmas \ref{lem:mle} and \ref{lem:xi_ai} can be applied.
Notice that $\1\{ \cap_{j \ne i^{*}} \{\phi_{i^{*}} > \phi_j\} \} = 0$ if and only if $\boldsymbol{\phi} \in \cup_{i \ne i^*} \Xi_{i}$. Thus,
\begin{eqnarray*}
	1-PCS^n_{B}
	&=& \int  \int_{\Rb^{k} \times \Rb^{k}_+}   \pi^n (\boldsymbol{\phi},\boldsymbol{\psi}) d \boldsymbol{\psi} d \boldsymbol{\phi}-\int  \int_{\Rb^{k} \times \Rb^{k}_+} \1\{ \cap_{i \ne i^{*}} \{\phi_{i^{*}} > \phi_i\} \}  \pi^n (\boldsymbol{\phi},\boldsymbol{\psi}) d \boldsymbol{\psi} d \boldsymbol{\phi}    \\
	&\le& \int  \int_{ \cup_{i \ne i^*} \Xi_{i} }   \pi^n (\boldsymbol{\phi},\boldsymbol{\psi}) d \boldsymbol{\psi} d \boldsymbol{\phi}\\
	&\le& (k-1) \max_{i\ne i^*} \int \int_{ \Xi_{i} } \pi^n (\boldsymbol{\phi},\boldsymbol{\psi}) d \boldsymbol{\psi} d \boldsymbol{\phi}.
\end{eqnarray*}
Meanwhile, $1-PCS_{B}^n \ge \max_{i\ne i^*} \int \int_{ \Xi_{i} } \pi^n (\boldsymbol{\phi},\boldsymbol{\psi}) d \boldsymbol{\psi} d \boldsymbol{\phi}$.
Then, we have
\begin{align*}
	\lim_{n \to \infty} \frac{1}{n} \log (1-PCS_{B}^n) \le& \lim_{n \to \infty} \frac{1}{n} \log (k-1) + \lim_{n \to \infty} \frac{1}{n} \log \left( \max_{i\ne i^*} \int \int_{ \Xi_{i} } \pi^n (\boldsymbol{\phi},\boldsymbol{\psi}) d \boldsymbol{\psi} d \boldsymbol{\phi} \right)  \\
	=& \max_{i\ne i^*} \lim_{n \to \infty} \frac{1}{n} \log \left(  \int \int_{ \Xi_{i} } \pi^n (\boldsymbol{\phi},\boldsymbol{\psi}) d \boldsymbol{\psi} d \boldsymbol{\phi} \right),
\end{align*}
as well as
\begin{align*}
	\lim_{n \to \infty} \frac{1}{n} \log (1-PCS_{B}^n) \ge&  \max_{i\ne i^*} \lim_{n \to \infty} \frac{1}{n} \log \left(  \int \int_{ \Xi_{i} } \pi^n (\boldsymbol{\phi},\boldsymbol{\psi}) d \boldsymbol{\psi} d \boldsymbol{\phi} \right)
\end{align*}
which jointly imply $\lim_{n \to \infty} \frac{1}{n} \log (1-PCS_{B}^n) =  \max_{i\ne i^*} \lim_{n \to \infty} \frac{1}{n} \log \left(  \int \int_{ \Xi_{i} } \pi^n (\boldsymbol{\phi},\boldsymbol{\psi}) d \boldsymbol{\psi} d \boldsymbol{\phi} \right)$.
In the following, we analyze $\lim_{n \to \infty} \frac{1}{n} \log \left(  \int \int_{ \Xi_{i} } \pi^n (\boldsymbol{\phi},\boldsymbol{\psi}) d \boldsymbol{\psi} d \boldsymbol{\phi} \right)$, $i \ne i^*$. Notice that for $i\ne i^*$,
\begin{align}
	\int \int_{ \Xi_{i} } \pi^n (\boldsymbol{\phi},\boldsymbol{\psi}) d \boldsymbol{\psi} d \boldsymbol{\phi}
	=&  \frac{ \int \int_{ \Xi_{i} } \pi^{0}(\boldsymbol{\phi},\boldsymbol{\psi}) \prod_{j=1}^{k} (L^{n} (\phi_{j},\psi_{j}) / L^{n} (\bar{X}^n_{j},(\tilde{S}^n_{j})^2)) d \boldsymbol{\psi} d \boldsymbol{\phi} }{ \int_{\Rb^k} \int_{\Rb^k_+} \pi^{0}(\boldsymbol{\phi}',\boldsymbol{\psi}') \prod_{j=1}^{k} (L^{n} (\phi_{j}',\psi_{j}') / L^{n} (\bar{X}^n_{j},(\tilde{S}^n_{j})^2)) d\boldsymbol{\phi}' d\boldsymbol{\psi}' }, \label{eq:themainfraction}
\end{align}
and
\begin{eqnarray}
	&\,& \log \max_{(\boldsymbol{\phi}, \boldsymbol{\psi}) \in \Xi_{i}} \prod_{j=1}^{k} (L^{n} (\phi_{j},\psi_{j})/ L^{n} (\bar{X}^n_{j},(\tilde{S}^n_{j})^2))\nonumber\\
	&=& - n \Big( \min_{ \phi_{i} }   \frac{\alpha_{i}}{2}\log \big( 1 + (\bar{X}^n_{i}-\phi_{i})^2/(\tilde{S}^n_{i})^2 \big)  + \frac{\alpha_{i^*}}{2}\log \big( 1 + (\bar{X}^n_{i^*}-\phi_{i})^2/(\tilde{S}^n_{i^*})^2 \big)     \Big)\label{eq:maxlikeratio1}\\
	&=& n  a((\tilde{S}^n_{i})^2,(\tilde{S}^n_{i^*})^2,\bar{X}^n_{i},\bar{X}^n_{i^*}), \label{ineq:max_like_ratio}
\end{eqnarray}
where (\ref{eq:maxlikeratio1}) holds by Lemma \ref{lem:mle} and (\ref{ineq:max_like_ratio}) holds by the definition of $a((\tilde{S}^n_{i})^2,(\tilde{S}^n_{i^*})^2,\bar{X}^n_{i},\bar{X}^n_{i^*},0)$ in Lemma \ref{lem:mle_eoc_z}. By \eqref{ineq:z_est} of Lemma \ref{lem:mle_eoc_z}, we have
\begin{equation*}
	|a(\sigma_{i}^2,\sigma_{i^*}^2,\mu_{i},\mu_{i^*}) - a((\tilde{S}^n_{i})^2,(\tilde{S}^n_{i^*})^2,\bar{X}^n_{i},\bar{X}^n_{i^*})| \le b_a \varepsilon,
\end{equation*}
which leads to
\begin{align}\label{ineq:a_est}
	\Big|  a(\sigma_{i}^2,\sigma_{i^*}^2,\mu_{i},\mu_{i^*}) - \frac{1}{n} \log \max_{(\boldsymbol{\phi}, \boldsymbol{\psi}) \in \Xi_{i}} \prod_{j=1}^{k} (L^{n} (\phi_{j},\psi_{j})/ L^{n} (\bar{X}^n_{j},(\tilde{S}^n_{j})^2)) \Big| \le   b_a \varepsilon .
\end{align}
For the numerator of (\ref{eq:themainfraction}), we have
\begin{align*}
	& \int \int_{ \Xi_{i} } \pi^{0}(\boldsymbol{\phi},\boldsymbol{\psi}) \prod_{j=1}^{k} (L^{n} (\phi_{j},\psi_{j})/ L^{n} (\bar{X}^n_{j},(\tilde{S}^n_{j})^2)) d \boldsymbol{\psi} d \boldsymbol{\phi} \\
	\le& \max_{(\boldsymbol{\phi}, \boldsymbol{\psi}) \in \Xi_{i}} \prod_{j=1}^{k} (L^{n} (\phi_{j},\psi_{j})/ L^{n} (\bar{X}^n_{j},(\tilde{S}^n_{j})^2)) \int \int_{ \Xi_{i} } \pi^{0}(\boldsymbol{\phi},\boldsymbol{\psi})  d \boldsymbol{\psi} d \boldsymbol{\phi}  \\
	\le& \exp \Big(  n \Big( a(\sigma_i^2,\sigma_{i^*}^2,\mu_i,\mu_{i^*}) + b_a \varepsilon \Big)   \Big)  \int \int_{ \Xi_{i} } \pi^{0}(\boldsymbol{\phi},\boldsymbol{\psi})  d \boldsymbol{\psi} d \boldsymbol{\phi},
\end{align*}
where the last inequality holds by \eqref{ineq:a_est}. For the denominator of (\ref{eq:themainfraction}), we have
\begin{align}
	& \int_{\Rb^k} \int_{\Rb^k_+} \pi^{0}(\boldsymbol{\phi}',\boldsymbol{\psi}') \prod_{j=1}^{k} (L^{n} (\phi_{j}',\psi_{j}') / L^{n} (\bar{X}^n_{j},(\tilde{S}^n_{j})^2)) d\boldsymbol{\phi}' d\boldsymbol{\psi}'  \nonumber \\
	\ge& \underline{c} \text{Volume}(H_{i^*}) \min_{(\boldsymbol{\phi}, \boldsymbol{\psi}) \in H_{i^*}} \Big(\prod_{j=1}^{k} (L^{n} (\phi_{j}',\psi_{j}') / L^{n} (\bar{X}^n_{j},(\tilde{S}^n_{j})^2))\Big)  \label{ineq:pcs_proof_dlb0} \\
	=& \underline{c} \text{Volume}(H_{i^*}) \min_{(\boldsymbol{\phi}, \boldsymbol{\psi}) \in H_{i^*}} \exp \left( \sum_{j=1}^k \sum_{l=1}^{N_j} \log f (X_{jl} | \phi_{j},\psi_{j}) - \sum_{i=1}^k \sum_{l=1}^{N_i} \log f (X_{il} | \bar{X}^n_{i},(\tilde{S}^n_{i})^2) \right)  \nonumber \\
	\ge& \underline{c} \text{Volume}(H_{i^*}) \exp \left( - n\varepsilon  \right), \label{ineq:pcs_proof_dlb}
\end{align}
where \eqref{ineq:pcs_proof_dlb0} holds because $H_{i^*} \subset H_w$ by Lemma \ref{lem:xi_ai} such that $\pi^{0}(\boldsymbol{\phi}',\boldsymbol{\psi}') \ge \underline{c}$ for any $(\boldsymbol{\phi}',\boldsymbol{\psi}') \in H_{i^*}$ and \eqref{ineq:pcs_proof_dlb} holds by \eqref{ineq:his}. Thus,
\begin{align}
	&\lim_{n \to \infty} \frac{1}{n} \log \int \int_{ \Xi_{i} } \pi^n (\boldsymbol{\phi},\boldsymbol{\psi}) d \boldsymbol{\psi} d \boldsymbol{\phi} \nonumber \\
	\le& \lim_{n \to \infty} \frac{1}{n} \log \frac{ \int \int_{ \Xi_{i} } \pi^{0}(\boldsymbol{\phi},\boldsymbol{\psi})  d \boldsymbol{\psi} d \boldsymbol{\phi} }{\underline{c} \text{Volume}(H_{i^*}) }    + \lim_{n \to \infty} \frac{1}{n} \log \frac{ \exp   (n ( a(\sigma_i^2,\sigma_{i^*}^2,\mu_i,\mu_{i^*}) + b_a \varepsilon ) ) }{ \exp \left(  - n\varepsilon  \right)}  \nonumber \\
	=& a(\sigma_i^2,\sigma_{i^*}^2,\mu_i,\mu_{i^*}) + (b_a+1) \varepsilon \nonumber \\
	=& (b_a+1) \varepsilon -  \min_{\phi_{i} } \left( \frac{\alpha_i}{2} \log \left(1 + (\mu_i-\phi_i)^2/\sigma_i^2 \right) + \frac{\alpha_{i^*}}{2} \log \left(1 + (\mu_{i^*}-\phi_{i})^2/\sigma_{i^*}^2 \right) \right). \label{ineq:uprate}
\end{align}

It remains to derive the lower bound on the convergence rate of $\int \int_{ \Xi_{i} } \pi^n (\boldsymbol{\phi},\boldsymbol{\psi}) d \boldsymbol{\psi} d \boldsymbol{\phi}$. Since
\begin{align*}
	&\int_{\Rb^k} \int_{\Rb^k_+} \pi^{0}(\boldsymbol{\phi}',\boldsymbol{\psi}') \prod_{j=1}^{k} (L^{n} (\phi_{j}',\psi_{j}') / L^{n} (\bar{X}^n_{j},(\tilde{S}^n_{j})^2)) d\boldsymbol{\phi}' d\boldsymbol{\psi}'  \nonumber \\
	\le& \max_{\boldsymbol{\phi}, \boldsymbol{\psi}} \prod_{j=1}^{k} (L^{n} (\phi_{j},\psi_{j})/ L^{n} (\bar{X}^n_{j},(\tilde{S}^n_{j})^2)) \int_{\Rb^k} \int_{\Rb^k_+} \pi^{0}(\boldsymbol{\phi}',\boldsymbol{\psi}')  d\boldsymbol{\phi}' d\boldsymbol{\psi}'\nonumber \\
	=& \int_{\Rb^k} \int_{\Rb^k_+} \pi^{0}(\boldsymbol{\phi}',\boldsymbol{\psi}')  d\boldsymbol{\phi}' d\boldsymbol{\psi}'=1
\end{align*}
and
\begin{align*}
	& \int \int_{ \Xi_{i} } \pi^{0}(\boldsymbol{\phi},\boldsymbol{\psi}) \prod_{j=1}^{k} (L^{n} (\phi_{j},\psi_{j}) / L^{n} (\bar{X}^n_{j},(\tilde{S}^n_{j})^2)) d \boldsymbol{\psi} d \boldsymbol{\phi}  \\
	\ge & \underline{c} \text{Volume}(H_{i}) \min_{(\boldsymbol{\phi}, \boldsymbol{\psi}) \in H_{i}} \exp \Big( \sum_{j=1}^k \sum_{l=1}^{N_j} \log f (X_{jl} | \phi_{j},\psi_{j}) - \sum_{i=1}^k \sum_{l=1}^{N_i} \log f (X_{il} | \bar{X}^n_{i},(\tilde{S}^n_{i})^2) \Big) \\
	\ge &\underline{c} \text{Volume}(H_{i})  \exp \Big( \max_{(\boldsymbol{\phi}, \boldsymbol{\psi}) \in \Xi_{i}}  \sum_{j=1}^k \sum_{l=1}^{N_j} \log f (X_{jl} | \phi_{j},\psi_{j}) - \sum_{i=1}^k \sum_{l=1}^{N_i} \log f (X_{il} | \bar{X}^n_{i},(\tilde{S}^n_{i})^2) - n \varepsilon \Big)  \\
	\ge &\underline{c} \text{Volume}(H_{i}) \exp \Big(  n \Big( a(\sigma_i^2,\sigma_{i^*}^2,\mu_i,\mu_{i^*}) - b_a \varepsilon - \varepsilon \Big)  \Big),
\end{align*}
where the first inequality holds because $H_{i} \subset H_w$ by Lemma \ref{lem:xi_ai} such that $\pi^{0}(\boldsymbol{\phi},\boldsymbol{\psi}) \ge \underline{c}$ for any $(\boldsymbol{\phi},\boldsymbol{\psi}) \in H_{i}$, the second inequality holds by \eqref{ineq:hi} and the last inequality holds by \eqref{ineq:a_est}.
Thus,
\begin{align}
	&\lim_{n \to \infty} \frac{1}{n} \log \int \int_{ \Xi_{i} } \pi^n (\boldsymbol{\phi},\boldsymbol{\psi}) d \boldsymbol{\psi} d \boldsymbol{\phi} \nonumber \\
	\ge& \lim_{n \to \infty} \frac{1}{n} \log \underline{c} \text{Volume}(H_{i})  + \lim_{n \to \infty} \frac{1}{n} \log \exp \Big(  n \Big( a(\sigma_i^2,\sigma_{i^*}^2,\mu_i,\mu_{i^*}) - b_a \varepsilon - \varepsilon \Big)  \Big)   \nonumber \\
	=& - (b_a+1) \varepsilon  -  \min_{\phi_{i} } \left( \frac{\alpha_i}{2} \log \left(1 + (\mu_i-\phi_i)^2/\sigma_i^2 \right) + \frac{\alpha_{i^*}}{2} \log \left(1 + (\mu_{i^*}-\phi_{i})^2/\sigma_{i^*}^2 \right) \right). \label{ineq:lowrate}
\end{align}
The result then follows from \eqref{ineq:uprate} and \eqref{ineq:lowrate} when we take $\varepsilon \to 0$. $\square$

\section{Proofs for Section \ref{sec:opt_cond}}

In the following, we prove Lemmas \ref{lem:phimono0}-\ref{lem:umono} and Theorem \ref{th:ocba}.

\subsection{Proof of Lemma \ref{lem:phimono0}}

For notational simplicity, let
\begin{align*}
	&\Ical( \phi, \mu, \sigma^2 ) \triangleq \frac{d \log (1 + (\mu-\phi_1)^2/\sigma^2 ) }{ d \phi_1 } \Big|_{\phi_1 = \phi}  = \frac{ 2(\phi-\mu) }{ \sigma^2 + (\phi-\mu)^2 }, \\
	&\Hcal( \phi, \mu, \sigma^2 ) \triangleq \frac{d^2 \log (1 + (\mu-\phi_1)^2/\sigma^2 ) }{ d \phi_1^2 } \Big|_{ \phi_1 = \phi } = \frac{ 2\sigma^2  - 2(\phi-\mu)^2 }{ (\sigma^2 + (\phi-\mu)^2)^2 }.
\end{align*}
The stationarity property of $\phi^{\min}_{i}(r)$ requires
\begin{align*}
	&\frac{d  g_{i}  (\phi_i, r ) }{d \phi_i } \Big|_{ \phi_i =  \phi^{\min}_{i}(r)}
	= r \Ical( \phi^{\min}_{i}(r), \mu_i, \sigma_i^2 ) + \Ical( \phi^{\min}_{i}(r), \mu_{i^*}, \sigma_{i^*}^2 )
	= 0.
\end{align*}
The stationarity property is the only place we use the derivative.

% This is item 3 which I commented out
%
%\item If $1/b_0 \le  \alpha_{i}/\alpha_{i^*} \le b_0$ for any $b_0 > 1$, then $ \phi^{\min}_{i}(r),\phi^{\max}_{i}(\alpha_i/\alpha_{i^*}) \in [ \mu_{i} + \kappa(b_0) , \mu_{i^*} - \kappa(b_0)]$ where $\kappa(b_0) \triangleq   \sigma_{\min}^2 \min_{j \ne i^*} (\mu_{i^*} - \mu_{j})  / (2 b_0 (\sigma_{\max}^2 + \max_{j} (\mu_{i^*}-\mu_{j})^2)) $.

% we need to prove item (ii) first for subsequent proof

Below, we first establish (ii), as this result will be used in the subsequent proof of (i).

\textbf{Proof of (ii).} Since both $\log (1 + (\mu_i-\phi_i)^2/\sigma_i^2 )$ and $\log (1 + (\mu_{i^*}-\phi_{i})^2/\sigma_{i^*}^2 )$ decrease with $\phi_i$ when $\phi_i < \mu_i$ and increase with $\phi_i$ when $\phi_i > \mu_{i^*}$, we have $\mu_{i} \le \phi^{\min}_{i}(r) \le \phi^{\max}_{i}(r) \le \mu_{i^*}$.
If $r = 0$, then $r  \log (1 + (\mu_i-\phi_i)^2/\sigma_i^2 ) +  \log (1 + (\mu_{i^*}-\phi_{i})^2/\sigma_{i^*}^2 ) = \log (1 + (\mu_{i^*}-\phi_{i})^2/\sigma_{i^*}^2 )$
and $\phi_i = \mu_{i^*}$ is the unique optimal solution. Notice that $\phi^{\min}_{i}(r)$ and $\phi^{\max}_{i}(r)$ are also the optimal solutions to $\min_{\phi_{i} } (  \log (1 + (\mu_i-\phi_i)^2/\sigma_i^2 ) + (\alpha_{i^*}/\alpha_i) \log (1 + (\mu_{i^*}-\phi_{i})^2/\sigma_{i^*}^2 ) )$.
If $r = \infty$ such that $\alpha_{i^*}/\alpha_i = 1/r = 0$, then $\phi_i = \mu_{i}$ is the unique optimal solution.

\textbf{Proof of (i).}
Let $r^{(1)} \triangleq \alpha_i^{(1)}/\alpha_{i^*}^{(1)} $ and $r^{(2)} \triangleq \alpha_i^{(2)}/\alpha_{i^*}^{(2)}$. Suppose $r^{(1)} > r^{(2)}$.  For any $\phi_i > \phi_{i}^{\min}(r^{(2)})$,
\begin{align*}
	& g_{i} \big(\phi_i, r^{(1)}\big)  - g_{i} \big(\phi_{i}^{\min}(r^{(2)}), r^{(1)}\big) - (g_{i} \big(\phi_i, r^{(2)}\big)  - g_{i} \big(\phi_{i}^{\min}(r^{(2)}), r^{(2)}\big))  \\
	=& \big( r^{(1)} - r^{(2)} \big) \big( \log \big(1 + (\mu_i-\phi_i)^2/\sigma_i^2 \big) - \log \big(1 + (\mu_i-\phi_{i}^{\min}(r^{(2)}))^2/\sigma_i^2 \big) \big)   > 0,
\end{align*}
where the last inequality holds by $\phi_i > \phi_{i}^{\min}(r^{(2)}) \ge \mu_i$.
Meanwhile,
\begin{align*}
	g_{i} \big(\phi_i, r^{(2)}\big)  - g_{i} \big(\phi_{i}^{\min}(r^{(2)}), r^{(2)}\big) \ge 0
\end{align*}
due to the optimality of $\phi_{i}^{\min}(r^{(2)})$. Thus, $g_{i} \big(\phi_i, r^{(1)}\big)  - g_{i} \big(\phi_{i}^{\min}(r^{(2)}), r^{(1)}\big) > 0$.
This shows that any $\phi_i > \phi_{i}^{\min}(r^{(2)})$ cannot be the optimal solution $\phi_i^{\max}(r^{(1)})$. Thus, $\phi_i^{\max}(r^{(1)}) \le \phi_{i}^{\min}(r^{(2)})$. The reason of $\phi_i^{\max}(r^{(1)}) < \phi_{i}^{\min}(r^{(2)})$ is as follows. The derivative of $g_{i} \big(\phi_i, r^{(1)}\big)$
at $\phi_i = \phi_{i}^{\min}(r^{(2)})$ is
\begin{align*}
	&r^{(1)} \Ical( \phi_{i}^{\min}(r^{(2)}),\mu_{i},\sigma_i^2 ) + \Ical( \phi_{i}^{\min}(r^{(2)}),\mu_{i^*},\sigma_{i^*}^2 )  \\
	=& (r^{(1)}-r^{(2)}) \Ical( \phi_{i}^{\min}(r^{(2)}),\mu_{i},\sigma_i^2 ) + r^{(2)} \Ical( \phi_{i}^{\min}(r^{(2)}),\mu_{i},\sigma_i^2 ) + \Ical( \phi_{i}^{\min}(r^{(2)}),\mu_{i^*},\sigma_{i^*}^2 )  .
\end{align*}
Since $\phi_{i}^{\min}(r^{(2)}) > \mu_i$,
\begin{align*}
	(r^{(1)}-r^{(2)}) \Ical( \phi_{i}^{\min}(r^{(2)}),\mu_{i},\sigma_i^2 ) = (r^{(1)}-r^{(2)}) \frac{ 2(\phi_{i}^{\min}(r^{(2)})-\mu_{i}) }{ \sigma_{i}^2 + (\phi_{i}^{\min}(r^{(2)})-\mu_{i})^2 } > 0.
\end{align*}
Meanwhile, $\phi_{i}^{\min}(r^{(2)})$ is an optimal solution and satisfies the stationary condition.
Then, the derivative at $\phi_i = \phi_{i}^{\min}(r^{(2)})$ is strictly positive.
This yields that the function $g_{i} \big(\phi_i, r^{(1)}\big)$
can be further reduced by letting $\phi_i < \phi_{i}^{\min}(r^{(2)})$. Thus, $\phi_{i}^{\max}(r^{(1)}) < \phi_{i}^{\min}(r^{(2)})$.

\textbf{Proof of (iii).}  We show the result for $\eta(b) =  (\sigma_{\max}^2 + (\mu_{i^*}-\mu_{i})^2)/  (b  \sigma_{\min}^2) $. Solutions $\phi_{i}^{\min}(r)$ and $\phi_{i}^{\max}(r)$ should satisfy the stationary condition
\begin{align}\label{eq:station_diff0}
	-\frac{\Ical ( \phi_{i}^{\min}(r), \mu_{i^*}, \sigma_{i^*}^2 )}{\Ical( \phi_{i}^{\min}(r), \mu_{i}, \sigma_{i}^2 )} = -\frac{\Ical ( \phi_{i}^{\max}(r), \mu_{i^*}, \sigma_{i^*}^2 )}{\Ical( \phi_{i}^{\max}(r), \mu_{i}, \sigma_{i}^2 )} = r.
\end{align}
Suppose $r \le 1/b$ and
$\mu_{i^*} - \phi_{i}^{\min}(r) > \eta(b) (\mu_{i^*}-\mu_{i})$. Then
\begin{align*}
	-\Ical ( \phi_{i}^{\min}(r), \mu_{i^*}, \sigma_{i^*}^2 ) =& \frac{ 2(\mu_{i^*}-\phi_{i}^{\min}(r)) }{ \sigma_{i^*}^2 + (\phi_{i}^{\min}(r)-\mu_{i^*})^2 } >  \frac{ 2 \eta(b) (\mu_{i^*}-\mu_{i}) }{ \sigma_{i^*}^2 + (\mu_{i^*}-\mu_{i})^2}, \\
	\Ical( \phi_{i}^{\min}(r), \mu_{i}, \sigma_{i}^2 ) =& \frac{ 2(\phi_{i}^{\min}(r)-\mu_{i}) }{ \sigma_{i}^2 + (\phi_{i}^{\min}(r)-\mu_{i})^2 } < \frac{ 2 (\mu_{i^*}-\mu_{i}) }{ \sigma_{i}^2  },
\end{align*}
which, together with $\eta(b) = (\sigma_{\max}^2 + (\mu_{i^*}-\mu_{i})^2)/  (b  \sigma_{\min}^2) \ge  (\sigma_{i^*}^2 + (\mu_{i^*}-\mu_{i})^2) /  b  \sigma_{i}^2 $, yields
\begin{align*}
	-\frac{\Ical ( \phi_{i}^{\min}(r), \mu_{i^*}, \sigma_{i^*}^2 )}{\Ical( \phi_{i}^{\min}(r), \mu_{i}, \sigma_{i}^2 )} > \frac{ 2 \eta(b) (\mu_{i^*}-\mu_{i}) }{ \sigma_{i^*}^2 + (\mu_{i^*}-\mu_{i})^2} \frac{ \sigma_{i}^2  }{ 2 (\mu_{i^*}-\mu_{i}) } = \frac{  \eta(b)  \sigma_{i}^2 }{ \sigma_{i^*}^2 + (\mu_{i^*}-\mu_{i})^2} \ge \frac{1}{b}  \ge r,
\end{align*}
contradicting \eqref{eq:station_diff0}. Thus, $\mu_{i^*} - \phi_{i}^{\min}(r) \le \eta(b) (\mu_{i^*}-\mu_{i})$. Similarly, we can show $ \phi_{i}^{\max}(r) - \mu_i \le \eta(b) (\mu_{i^*}-\mu_{i})$ when $r \ge b$.

\textbf{Proof of (iv).} 		
Let $\{r^{(l)},l=1,\dots,\infty\}$ denote a monotonically decreasing sequence that converges to $r$. We will focus only on this case, as the result for the increasing sequence is proved similarly.

By (i), $\phi_{i}^{\max}(r^{(l)})$ is strictly increasing. Since $\phi_{i}^{\max}(r^{(l)}) \le \mu_{i^*}$, by monotone convergence theorem, there exists $\bar{\phi}_i$ such that $\lim_{l \to \infty} \phi_{i}^{\max}(r^{(l)}) = \bar{\phi}_i$. Then $\bar{\phi}_i = \phi_{i}^{\min}(r)$ must hold.

Otherwise, if $\bar{\phi}_i \ne \phi_{i}^{\min}(r)$, then $\bar{\phi}_i < \phi_{i}^{\min}(r)$ because $\phi_{i}^{\max}(r^{(l)}) < \phi_{i}^{\min}(r)$ by (i) such that $\bar{\phi}_i \le \phi_{i}^{\min}(r)$.
Let $\Delta_i = \min_{\phi_i \le \bar{\phi}_i} g_{i} (\phi_i, r) - g_{i} (\phi_{i}^{\min}(r), r)$. Since $\phi_{i}^{\min}(r)$ is the smallest optimal solution for $\min_{\phi_i} g_{i} (\phi_i, r)$, we have
\begin{align}\label{ineq:phi_con_delta}
	\min_{\phi_i \le \bar{\phi}_i} g_{i} (\phi_i, r) - \min_{\phi_i} g_{i} (\phi_i, r) = \Delta_i > 0.
\end{align}
Notice that for any $\phi_i$, $g_{i} (\phi_i, r^{(l)}) = g_{i} (\phi_i, r ) + \log (1 + (\mu_i-\phi_i)^2/\sigma_i^2)  (r^{(l)} - r )$.
Since $r^{(l)} \ge r$, we know
$g_{i} (\phi_{i}^{\max}(r^{(l)}), r) \le g_{i} (\phi_{i}^{\max}(r^{(l)}), r^{(l)} )$. Due to the optimality of $\phi_{i}^{\max}(r^{(l)})$, we have $g_{i} (\phi_{i}^{\max}(r^{(l)}), r^{(l)} ) \le g_{i} (\phi_{i}^{\min}(r), r^{(l)} )$.
Then
\begin{align*}
	&g_{i} (\phi_{i}^{\max}(r^{(l)}), r)
	\le g_{i} (\phi_{i}^{\min}(r), r^{(l)} )
	\le g_{i} (\phi_{i}^{\min}(r), r ) + \log (1 + (\mu_i-\mu_{i^*})^2/\sigma_i^2)  (r^{(l)} - r ).
\end{align*}
Since $r^{(l)} \to r$ as $l \to \infty$, there exists $l_0$ large enough such that $\log (1 + (\mu_i-\mu_{i^*})^2/\sigma_i^2)  (r^{(l)} - r ) \le \Delta_i/2$.
Then $g_{i} (\phi_{i}^{\max}(r^{(l_0)}), r) \le g_{i} (\phi_{i}^{\min}(r), r ) + \Delta_i/2$ which contradicts \eqref{ineq:phi_con_delta} because $\phi_{i}^{\max}(r^{(l_0)}) \le \bar{\phi}_i$. Thus $\lim_{l \to \infty} \phi_{i}^{\max}(r^{(l)}) = \phi_{i}^{\min}(r)$.

Similarly, if $\{r^{(l)},l=1,\dots,\infty\}$ is a monotonically increasing sequence that converges to $r$, we can show $\lim_{l \to \infty} \phi_{i}^{\min}(r^{(l)}) = \phi_{i}^{\max}(r)$. $\square$

\subsection{Proof of Lemma \ref{lem:vmono0}}

\textbf{Proof of (i).} Since $\Vcal_{i} (\alpha_i,\alpha_{i^*})$ is the minimum of a continuous function of $(\alpha_i,\alpha_{i^*})$, $\Vcal_{i} (\alpha_i,\alpha_{i^*})$ is continuous in $(\alpha_i,\alpha_{i^*})$.

\textbf{Proof of (ii).} By the definition of $\Vcal_{i}$,
\begin{align*}
	&2\Vcal_{i}(\alpha_i^{(1)},\alpha_{i^*}^{(1)})
	=  \alpha_i^{(1)} \log (1 + (\mu_i-\phi^*)^2/\sigma_i^2 ) + \alpha_{i^*}^{(1)} \log (1 + (\mu_{i^*}-\phi^*)^2/\sigma_{i^*}^2 ) \\
	=& \alpha_i^{(3)} \log (1 + (\mu_i-\phi^*)^2/\sigma_i^2 ) + \alpha_{i^*}^{(3)} \log (1 + (\mu_{i^*}-\phi^*)^2/\sigma_{i^*}^2 )  + (\alpha_i^{(1)} - \alpha_i^{(3)}) \log (1 + (\mu_i-\phi^*)^2/\sigma_i^2 )  \\
	& + (\alpha_{i^*}^{(1)}-\alpha_{i^*}^{(3)}) \log (1 + (\mu_{i^*}-\phi^*)^2/\sigma_{i^*}^2 ) \\
	\ge&  2\Vcal_{i}(\alpha_i^{(3)},\alpha_{i^*}^{(3)}) + (\alpha_i^{(1)} - \alpha_i^{(3)}) \log (1 + (\mu_i-\phi^*)^2/\sigma_i^2 ) + (\alpha_{i^*}^{(1)}-\alpha_{i^*}^{(3)}) \log (1 + (\mu_{i^*}-\phi^*)^2/\sigma_{i^*}^2 )
\end{align*}
where the inequality holds because $\Vcal_{i}(\alpha_i^{(3)},\alpha_{i^*}^{(3)})$ is the minimum value. In addition, if $\phi^* \notin \arg\min_{\phi_i} g_i(\phi_i,\alpha_i^{(3)}/\alpha_{i^*}^{(3)}) $, then
\begin{align*}
	\alpha_i^{(3)} \log (1 + (\mu_i-\phi^*)^2/\sigma_i^2 ) + \alpha_{i^*}^{(3)} \log (1 + (\mu_{i^*}-\phi^*)^2/\sigma_{i^*}^2 ) > 2\Vcal_{i}(\alpha_i^{(3)},\alpha_{i^*}^{(3)}),
\end{align*}
implying that the inequality is strict. This concludes the proof.

\textbf{Proof of (iii).} It is straightforward to see $\Wcal_i( 0 )  = 0$.
Suppose $r^{(1)} > r^{(2)}$.
\begin{align}
	\Wcal_{i}(r^{(1)})
	=& r^{(2)} \log (1 + (\mu_i-\phi^{\max}_{i}(r^{(1)}))^2/\sigma_i^2 ) + \log (1 + (\mu_{i^*}-\phi^{\max}_{i}(r^{(1)}))^2/\sigma_{i^*}^2 ) \nonumber \\
	& + (r^{(1)} - r^{(2)}) \log (1 + (\mu_i-\phi^{\max}_{i}(r^{(1)}))^2/\sigma_i^2 )  \nonumber \\
	>& \Wcal_{i}(r^{(2)}) + (r^{(1)} - r^{(2)}) \log (1 + (\mu_i-\phi^{\max}_{i}(r^{(1)}))^2/\sigma_i^2 ) \label{ineq:wlb}
\end{align}
where \eqref{ineq:wlb} holds because $\phi^{\max}_{i}(r^{(1)}) < \phi^{\min}_{i}(r^{(2)})$. Thus $\Wcal_{i}(r^{(1)}) > \Wcal_{i}(r^{(2)})$. $\square$

\subsection{Proof of Lemma \ref{lem:vequal}}

We prove Lemma \ref{lem:vequal}(i) and (ii) first. Without loss of generality, let $i^* = 1$ for notational simplicity. The following statement can be shown by induction: given $h \in \{1,\dots,k-2\}$ and for any $0 < \alpha_1 < 1$ and $0 < c \le 1-\alpha_1$, among all $(\alpha_{k-h}, \alpha_{k-h+1}, \dots,\alpha_{k}) $ with $\sum_{i = k-h}^k \alpha_{i} = c$, there exists a unique $(\tilde{\alpha}_{k-h} (c,h), \tilde{\alpha}_{k-h+1} (c,h), \dots,\tilde{\alpha}_{k} (c,h)) $  satisfying
\begin{align*}
	\Vcal_{i}(\tilde{\alpha}_{i} (c,h),\alpha_1) = \Vcal_{k}(\tilde{\alpha}_{k} (c,h),\alpha_1), \  i = k-h,k-h+1,\dots,k-1.
\end{align*}
Moreover, $\tilde{\alpha}_{i} (c,h) > 0$ and is continuous in $c$ such that $0< \tilde{\alpha}_i(c',h) - \tilde{\alpha}_i(c,h) \le \Delta$ if $0 < c' - c \le \Delta$, $i = k-h,k-h+1,\dots,k$. The lemma is immediate by setting $h=k-2$ and $c = 1-\alpha_{1}$.

To begin with, we consider $h=1$. The proof is similar to the following proof for $h > 1$ and thus omitted.

Let $h = 2,\dots,k-2$. Suppose that the statement is true for $h-1$, and we show the claim is also true for $h$. Consider function $d_h(\delta) = \Vcal_{k-h}(\delta,\alpha_1) - \Vcal_{k}(\tilde{\alpha}_{k} (c-\delta,h-1),\alpha_1)$. By assumption, $\tilde{\alpha}_{k}(c-\delta,h-1)$ is continuous in $c-\delta$ and thus in $\delta$, which, together with the continuity of $\Vcal_k(\alpha_k,\alpha_1)$ in $\alpha_k$, implies that $\Vcal_{k}(\tilde{\alpha}_{k} (c-\delta,h-1),\alpha_1)$ is continuous in $\delta$. Meanwhile,  $\Vcal_{k-h}(\delta,\alpha_1)$ is also continuous in $\delta$. Moreover, by assumption, $\tilde{\alpha}_{k}(c-\delta,h-1)$ is strictly increasing with $c-\delta$ and thus strictly decreasing with $\delta$ given $c$. Then, $\Vcal_{k}(\tilde{\alpha}_{k} (c-\delta,h-1),\alpha_1)$ is strictly decreasing with $\delta$, and $\Vcal_{k-h}(\delta,\alpha_1)$ is strictly increasing with $\delta$. Thus, $d_h(\delta)$ is continuous and strictly increasing with $\delta$.

If $\delta=\delta_{1}$ where $\delta_{1} \le \varepsilon$ for $\varepsilon$ small enough, then $\Vcal_{k-h}(\delta_1,\alpha_1)$ is small enough. Meanwhile, there exists $j \in \{ k-h+1,\dots,k \}$ such that $\tilde{\alpha}_{j} (c-\delta_1,h-1) \ge \sum_{i=k-h+1}^k \tilde{\alpha}_{i} (c-\delta_1,h-1)/h = (c-\delta_1)/h $. Then $\Vcal_{k}(\tilde{\alpha}_{k}(c-\delta_1,h-1),\alpha_1) = \Vcal_{j}(\tilde{\alpha}_{j}(c-\delta_1,h-1),\alpha_1) \ge \Vcal_{j}((c-\delta_1)/h,\alpha_1)$, which implies $d_h(\delta_1) = \Vcal_{k-h}(\delta_1,\alpha_1) - \Vcal_{k}(\tilde{\alpha}_{k}(c-\delta_1,h-1),\alpha_1) < 0$ for $\varepsilon$ small enough. Similarly, if $\delta=\delta_{2}$ where $c - \varepsilon \le \delta_{2} \le c$ for $\varepsilon$ small enough, we have $d_h(\delta_2) > 0$.
Then there exists a unique $\tilde\alpha_{k-h}(c,h) \in (0,c)$ with
$d_h( \tilde\alpha_{k-h}(c,h) ) = 0$.
Let
\begin{align}\label{eq:h_alpha_def}
	\tilde{\alpha}_{i} (c,h) = \tilde{\alpha}_{i}(c-\tilde\alpha_{k-h}(c,h),h-1),
\end{align}
$i=k-h+1,\dots,k$, and we have shown its existence and uniqueness. Since $\tilde\alpha_{k-h}(c,h) \in (0,c)$,  $\tilde{\alpha}_{i} (c,h) = \tilde{\alpha}_{i}(c-\tilde\alpha_{k-h}(c,h),h-1) > 0$.

The continuity of $\tilde{\alpha}_{i} (c,h)$ in $c$ is shown as follows. Suppose $c$ and $c'$ satisfy $c' - c = \Delta$ for any feasible $\Delta > 0$. Notice that $\tilde{\alpha}_{k}(c,h) = \tilde{\alpha}_{k}(c-\tilde{\alpha}_{k-h}(c,h),h-1)$ by \eqref{eq:h_alpha_def}. Consider the function $\tilde{d}_h(\delta) = \Vcal_{k-h}(\tilde{\alpha}_{k-h}(c,h)+\Delta-\delta,\alpha_1) - \Vcal_{k}(\tilde{\alpha}_{k}(c-\tilde{\alpha}_{k-h}(c,h)+\delta,h-1),\alpha_1)$. Since $\Vcal_{k-h}(\tilde{\alpha}_{k-h}(c,h),\alpha_1) = \Vcal_{k}(\tilde{\alpha}_{k}(c,h),\alpha_1)$, we have
\begin{align*}
	\tilde{d}_h(\Delta) =&\Vcal_{k-h}(\tilde{\alpha}_{k-h}(c,h),\alpha_1) - \Vcal_{k}(\tilde{\alpha}_{k}(c-\tilde{\alpha}_{k-h}(c,h)+\Delta,h-1),\alpha_1) \\
	<& \Vcal_{k-h}(\tilde{\alpha}_{k-h}(c,h),\alpha_1) - \Vcal_{k}(\tilde{\alpha}_{k}(c-\tilde{\alpha}_{k-h}(c,h),h-1),\alpha_1)\\
	=& 0,
\end{align*}
where the inequality holds by the assumption that $\tilde{\alpha}_{k}(c,h-1) < \tilde{\alpha}_{k}(c',h-1)$ if $c < c'$. Similarly $\tilde{d}_h(0) > 0$. There must exist $\delta_h \in (0,\Delta)$ such that $\tilde{d}_h(\delta_h) = 0$. Then, $\tilde{\alpha}_{k-h}(c',h) = \tilde{\alpha}_{k-h}(c,h)+\Delta-\delta_h$ and $\tilde{\alpha}_{k}(c',h) = \tilde{\alpha}_{k}(c-\tilde{\alpha}_{k-h}(c,h)+\delta_h,h-1)$. By the inductive hypothesis, we know $0< \tilde{\alpha}_{k}(c-\tilde{\alpha}_{k-h}(c,h)+\delta_h,h-1) - \tilde{\alpha}_{k}(c-\tilde{\alpha}_{k-h}(c,h),h-1) \le \delta_h$. Thus, $0 < \tilde{\alpha}_{k}(c',h) - \tilde{\alpha}_{k}(c,h) \le \Delta$. Similarly, $0 < \tilde{\alpha}_{i}(c',h) - \tilde{\alpha}_{i}(c,h) \le \Delta$, $i=k-h+1,\dots,k-1$. This completes the proof of Lemma \ref{lem:vequal}(i)-(ii).

Now we show Lemma \ref{lem:vequal}(iii). Consider any $\balpha$ with $\balpha \ne \balpha^f(\bar{\alpha}_{i^*})$ and $\alpha_{i^*} = \bar{\alpha}_{i^*}$. Then there must exist $i_1,i_2 \ne i^*$ with $\alpha_{i_1} < \alpha_{i_1}^f(\bar{\alpha}_{i^*})$ and $\alpha_{i_2} > \alpha_{i_2}^f(\bar{\alpha}_{i^*})$ due to the constraint $\sum_{i=1}^k \alpha_{i} = 1$. By Lemma \ref{lem:vmono0}(iii), $\Vcal_{i_1}( \alpha_{i_1}, \bar{\alpha}_{i^*}) = \frac{\bar{\alpha}_{i^*}}{2} \Wcal_{i_1}( \alpha_{i_1} / \bar{\alpha}_{i^*}) < \frac{\bar{\alpha}_{i^*}}{2} \Wcal_{i_1}( \alpha_{i_1}^f(\bar{\alpha}_{i^*}) / \bar{\alpha}_{i^*}) =  \Vcal_{i_1}( \alpha_{i_1}^f(\bar{\alpha}_{i^*}), \bar{\alpha}_{i^*})$,
which leads to
\begin{align*}
	\min_{i \ne i^*} \Vcal_i( \alpha_{i}, \alpha_{i^*} ) \le \Vcal_{i_1}( \alpha_{i_1}, \bar{\alpha}_{i^*}) < \Vcal_{i_1}( \alpha_{i_1}^f(\bar{\alpha}_{i^*}), \bar{\alpha}_{i^*}) = \min_{i \ne i^*} \Vcal_{i}( \alpha_{i}^f(\bar{\alpha}_{i^*}), \bar{\alpha}_{i^*} ).
\end{align*}
Thus, $\balpha^f(\bar{\alpha}_{i^*})$ is the unique optimal solution to problem \eqref{eq:rateoptimization_toptwo} of the main text. $\square$		 

\subsection{Proof of Lemma \ref{lem:umono}}

Let $r_i^f(\bar{\alpha}_{i^*}) \triangleq \alpha_i^f( \bar{\alpha}_{i^*} ) / \bar{\alpha}_{i^*}$ for any $i \ne i^*$ and $0<\bar{\alpha}_{i^*}<1$. Consider two possible values $\alpha_{i^*}',\alpha_{i^*}''$ of $\bar{\alpha}_{i^*}$ with $\alpha_{i^*}' < \alpha_{i^*}''$. By Lemma \ref{lem:vequal}(ii), we have $\alpha_i^f( \alpha_{i^*}' ) > \alpha_i^f( \alpha_{i^*}'' )$ for any $i \ne i^*$, which yields
\begin{align}\label{ineq:alpha_non_mono}
	r_{i}^f( \alpha_{i^*}' ) > r_{i}^f( \alpha_{i^*}'' ) .
\end{align}
By the monotonicity of $\phi_{i}^{\min}(r)$ and $\phi_{i}^{\max}(r)$ shown in Lemma \ref{lem:phimono0}(i), we have $\mu_{i} \le \phi_{i}^{\min}(r_{i}^f( \alpha_{i^*}' )) \le \phi_{i}^{\max}(r_{i}^f( \alpha_{i^*}' )) < \phi_{i}^{\min}(r_{i}^f( \alpha_{i^*}'' )) \le \mu_{i^*}$,
which yields
\begin{align*}
	\sum_{i \ne i^*} \frac{\Ucal^{*,\min}_{i}( r_{i}^f(\alpha_{i^*}') )}{\Ucal^{\min}_i( r_{i}^f(\alpha_{i^*}') )} = \sum_{i \ne i^*} \frac{ \log (1 + (\mu_{i^*}-\phi^{\min}_{i}(r_{i}^f(\alpha_{i^*}')))^2/\sigma_{i^*}^2 ) }{ \log (1 + (\mu_i-\phi^{\min}_{i}(r_{i}^f(\alpha_{i^*}')))^2/\sigma_i^2 ) } > \sum_{i \ne i^*} \frac{\Ucal^{*,\min}_{i}( r_{i}^f(\alpha_{i^*}'') )}{\Ucal^{\min}_i( r_{i}^f(\alpha_{i^*}'') )} .
\end{align*}
Similarly, we also have
\begin{align*}
	\sum_{i \ne i^*} \frac{\Ucal^{*,\max}_{i}( r_{i}^f(\alpha_{i^*}') )}{\Ucal^{\max}_i( r_{i}^f(\alpha_{i^*}') )} > \sum_{i \ne i^*} \frac{\Ucal^{*,\max}_{i}( r_{i}^f(\alpha_{i^*}'') )}{\Ucal^{\max}_i( r_{i}^f(\alpha_{i^*}'') )} .
\end{align*}

Let $\Acal^{\min}$ and $\Acal^{\max}$ denote the set of possible values of $\bar{\alpha}_{i^*}$ as
\begin{align*}
	\Acal^{\min} \triangleq \left\{\bar{\alpha}_{i^*}: \sum_{i \ne i^*}  \frac{\Ucal^{*,\min}_{i}( r_{i}^f(\bar{\alpha}_{i^*}) ) }{ \Ucal^{\min}_{i}( r_{i}^f(\bar{\alpha}_{i^*}) ) }  \ge 1 \right\}, \ \Acal^{\max} \triangleq \left\{\bar{\alpha}_{i^*}: \sum_{i \ne i^*}  \frac{\Ucal^{*,\max}_{i}( r_{i}^f(\bar{\alpha}_{i^*}) ) }{ \Ucal^{\max}_{i}( r_{i}^f(\bar{\alpha}_{i^*}) ) }  \le 1 \right\}.
\end{align*}
Let $\alpha_{i^*}^{\min}$ be the supremum of $\bar{\alpha}_{i^*}$ in $\Acal^{\min}$. When $\bar{\alpha}_{i^*} \to 0$, there must exist $i^\circ \ne i^*$ with $r_{i^\circ}^f(\bar{\alpha}_{i^*})  \to \infty$ such that $\phi_{i^\circ}^{\min}( r_{i^\circ}^f(\bar{\alpha}_{i^*}) ) \to \mu_{i^\circ}$. Then
$\sum_{i \ne i^*} \Ucal^{*,\min}_{i}( r_{i}^f(\bar{\alpha}_{i^*}) )/\Ucal^{\min}_{i}( r_{i}^f(\bar{\alpha}_{i^*}) )  \to \infty$. Thus $\Acal^{\min}$ is non-empty and $\alpha_{i^*}^{\min} > 0$.
Let $\alpha_{i^*}^{\max}$ denote the infimum of $\bar{\alpha}_{i^*}$ in $\Acal^{\max}$. By a similar argument to that of $\alpha_{i^*}^{\min}$, $\Acal^{\max}$ is non-empty and $\alpha_{i^*}^{\max} < 1$.

If $\alpha_{i^*}^{\min} < \alpha_{i^*}^{\max}$, then for any $\bar{\alpha}_{i^*}$ with $\alpha_{i^*}^{\min} < \bar{\alpha}_{i^*} < \alpha_{i^*}^{\max}$, since $\bar{\alpha}_{i^*} < \alpha_{i^*}^{\max}$, we have
\begin{align*}
	\sum_{i \ne i^*}  (\Ucal^{*,\max}_{i}( r_{i}^f(\bar{\alpha}_{i^*}) ) / \Ucal^{\max}_{i}( r_{i}^f(\bar{\alpha}_{i^*}) )) > 1,
\end{align*}
such that
\begin{align*}
	\sum_{i \ne i^*}  (\Ucal^{*,\min}_{i}( r_{i}^f(\bar{\alpha}_{i^*}) ) / \Ucal^{\min}_{i}( r_{i}^f(\bar{\alpha}_{i^*}) )) \ge \sum_{i \ne i^*}  (\Ucal^{*,\max}_{i}( r_{i}^f(\bar{\alpha}_{i^*}) ) / \Ucal^{\max}_{i}( r_{i}^f(\bar{\alpha}_{i^*}) )) > 1,
\end{align*}
implying $\bar{\alpha}_{i^*} \le \alpha_{i^*}^{\min}$. This contradicts the assumption that $\alpha_{i^*}^{\min} < \bar{\alpha}_{i^*} < \alpha_{i^*}^{\max}$. Meanwhile, if $\alpha_{i^*}^{\min} > \alpha_{i^*}^{\max}$, then we will also find a contradiction by similar arguments.
Thus, $\alpha_{i^*}^{\min} = \alpha_{i^*}^{\max}$.

In the following, we show
\begin{align}\label{ineq:alpha1_u}
	\sum_{i \ne i^*} (\Ucal^{*,\min}_{i}( r_{i}^f(\alpha_{i^*}^{\min}) ) / \Ucal^{\min}_{i}( r_{i}^f(\alpha_{i^*}^{\min}) )) \ge 1.
\end{align}
Let $\{\alpha_{i^*}^{(l)}, l =1,2,\dots\}$ denote a monotonically increasing sequence with $\lim_{l \to \infty} \alpha_{i^*}^{(l)} = \alpha_{i^*}^{\min}$. By \eqref{ineq:alpha_non_mono}, we know $\{r_i^f(\alpha_{i^*}^{(l)}), l =1,2,\dots\}$ is a decreasing sequence. Since $\alpha_i^f(\alpha_{i^*})$ is continuous in $\alpha_{i^*}$ by Lemma \ref{lem:vequal}(ii), we have $\lim_{l \to \infty} r_i^f(\alpha_{i^*}^{(l)}) = \lim_{l \to \infty} \alpha_i^f(\alpha_{i^*}^{(l)})/\alpha_{i^*}^{(l)} = \alpha_i^f(\alpha_{i^*}^{\min})/\alpha_{i^*}^{\min} = r_i^f(\alpha_{i^*}^{\min})$, which, by Lemma \ref{lem:phimono0}(iv), leads to $\lim_{l \to \infty} \phi^{\max}_{i}(r_{i}^f(\alpha_{i^*}^{(l)})) = \phi^{\min}_{i}(r_i^f(\alpha_{i^*}^{\min}))$. Meanwhile, $\alpha_{i^*}^{\min} = \alpha_{i^*}^{\max}$ such that $\alpha_{i^*}^{\min}$ is the infimum of $\Acal^{\max}$. Since $\alpha_{i^*}^{(l)} < \alpha_{i^*}^{\min}$ for any $l$ such that $\alpha_{i^*}^{(l)} \notin \Acal^{\max}$, which implies
\begin{align*}
	\sum_{i \ne i^*} ( \Ucal^{*,\max}_{i}( r_{i}^f(\alpha_{i^*}^{(l)}) ) / \Ucal^{\max}_{i}( r_{i}^f(\alpha_{i^*}^{(l)}) ) ) > 1,
\end{align*}
which, together with $\lim_{l \to \infty} \phi^{\max}_{i}(r_{i}^f(\alpha_{i^*}^{(l)})) = \phi^{\min}_{i}(r_i^f(\alpha_{i^*}^{\min}))$, leads to \eqref{ineq:alpha1_u}.
Similarly, we can show
\begin{align*}
	\sum_{i \ne i^*} (\Ucal^{*,\max}_{i}( r_{i}^f(\alpha_{i^*}^{\min}) ) / \Ucal^{\max}_{i}( r_{i}^f(\alpha_{i^*}^{\min}) )) \le 1.
\end{align*}

Since \eqref{ineq:alpha1_u} holds, we have $\alpha_{i^*}^{\min} \in \Acal^{\min}$. Then the supremum $\alpha_{i^*}^{\min}$ is also the maximum value of $\bar{\alpha}_{i^*}$ in $\Acal^{\min}$. Thus, $\alpha_{i^*}^{\min}$ is the desired $\alpha_{i^*}^*$ in \eqref{ocba_t} of the main text. Since $\alpha_{i^*}^{\min} > 0$ and $\alpha_{i^*}^{\min} = \alpha_{i^*}^{\max} < 1$, we have $0 < \alpha_{i^*}^* < 1$. $\square$

\subsection{Proof of Theorem \ref{th:ocba}} \label{sec:proof_ocba}

It is sufficient to show that $\boldsymbol{\alpha}^*$, where $\alpha_i^* = \alpha_{i}^{f}( \alpha_{i^*}^* )$, $i=1,\dots,k$, exists and is the unique optimal solution to \eqref{eq:newocba} of the main text. To elaborate, we know by Lemma \ref{lem:umono} that $\alpha_{i^*}^*$ exists and $0 < \alpha_{i^*}^* < 1$. Then, by Lemma \ref{lem:vequal}, $\alpha_{i}^{f}( \alpha_{i^*}^* )$ exists, $i=1,\dots,k$.

Now we show the optimality of $\balpha^*$. If $\alpha_{i^*} = 0$, then $\Vcal_i( \alpha_i, \alpha_{i^*} ) = 0$ for any feasible value of $\alpha_i$, $i\ne i^*$; similarly, if $\alpha_{i^*} = 1$, then $\alpha_i = 0$ must hold such that $\Vcal_i( \alpha_i, \alpha_{i^*} ) = 0$, $i \ne i^*$. Meanwhile, if $0<\alpha_i<1$ for all $i=1,\dots,k$, then $\Vcal_i( \alpha_i, \alpha_{i^*} ) > 0$ by Lemma \ref{lem:vmono0}(iii). Thus, to obtain the optimal value of $\min_{i \ne i^*} \Vcal_i( \alpha_i, \alpha_{i^*} ) $, we should have $0 < \alpha_{i^*} < 1$.

By Lemma \ref{lem:umono}, the allocation $\balpha^{*} $ satisfies
\begin{align}\label{ineq:urelation1}
	&\sum_{i \ne i^*} \frac{\Ucal^{*,\min}_{i}( r_{i}^f(\alpha_{i^*}^*) ) }{ \Ucal^{\min}_{i}( r_{i}^f(\alpha_{i^*}^*) ) } = \sum_{i \ne i^*} \frac{\log( 1+(\mu_{i^*}-\phi_{i}^{\min}(r_{i}^{f}(\alpha_{i^*}^*)))^2/\sigma_{i^*}^2 )}{ \log( 1+(\mu_{i}-\phi_{i}^{\min}(r_{i}^{f}(\alpha_{i^*}^*)))^2/\sigma_{i}^2 ) } \ge 1.
\end{align}
Let $\bar{\alpha}_{i^*,1} = \alpha_{i^*}^* - \Delta$ with $0< \Delta \le \alpha_{i^*}^*$. We show $\Vcal_{i} ( \alpha_{i}^{f}(\bar{\alpha}_{i^*,1}), \bar{\alpha}_{i^*,1} ) < \Vcal_{i} ( \alpha_{i}^{f}(\alpha_{i^*}^*), \alpha_{i^*}^* )$ for some $i \ne i^*$ by contradiction. Suppose $\Vcal_{i} ( \alpha_{i}^{f}(\bar{\alpha}_{i^*,1}), \bar{\alpha}_{i^*,1} ) \ge \Vcal_{i} ( \alpha_{i}^{f}(\alpha_{i^*}^*), \alpha_{i^*}^* )$ for any $i \ne i^*$. Since $r_{i}^{f}(\bar{\alpha}_{i^*,1}) = \alpha_{i}^{f}(\bar{\alpha}_{i^*,1}) / \bar{\alpha}_{i^*,1} > \alpha_{i}^{f}(\alpha_{i^*}^*) / \alpha_{i^*}^* = r_{i}^{f}(\alpha_{i^*}^*)$ by Lemma \ref{lem:vequal}(ii), we have by Lemma \ref{lem:phimono0} that $\phi_{i}^{\max}(r_{i}^{f}(\bar{\alpha}_{i^*,1})) < \phi_{i}^{\min}(r_{i}^{f}(\alpha_{i^*}^*))$, which implies that $\phi_{i}^{\min}(r_{i}^{f}(\alpha_{i^*}^*))$ is not in $\arg\min_{\phi_i} g_i(\phi_i,r_{i}^{f}(\bar{\alpha}_{i^*,1}) )$. Then by Lemma \ref{lem:vmono0}(ii),
\begin{align*}
	2\Vcal_{i} ( \alpha_{i}^{f}(\bar{\alpha}_{i^*,1}), \bar{\alpha}_{i^*,1} )
	<& 2\Vcal_{i} ( \alpha_{i}^{f}(\alpha_{i^*}^*), \alpha_{i^*}^* ) - \Delta \log( 1+(\mu_{i^*}-\phi_{i}^{\min}(r_{i}^{f}(\alpha_{i^*}^*)))^2/\sigma_{i^*}^2 ) \nonumber\\
	& + (\alpha_{i}^{f}(\bar{\alpha}_{i^*,1}) - \alpha_{i}^{f}(\alpha_{i^*}^*)) \log( 1+(\mu_{i}-\phi_{i}^{\min}(r_{i}^{f}(\alpha_{i^*}^*) ))^2/\sigma_{i}^2 ).
\end{align*}
Based on the above inequality, if $\Vcal_{i} ( \alpha_{i}^{f}(\bar{\alpha}_{i^*,1}), \bar{\alpha}_{i^*,1} ) \ge \Vcal_{i} ( \alpha_{i}^{f}(\alpha_{i^*}^*), \alpha_{i^*}^* )$ is true, then
\begin{align}\label{ineq:alpha2inc}
	\alpha_{i}^{f}(\bar{\alpha}_{i^*,1}) - \alpha_{i}^{f}(\alpha_{i^*}^*) >  \Delta \frac{\log( 1+(\mu_{i^*}-\phi_{i}^{\min}(r_{i}^{f}(\alpha_{i^*}^*)))^2/\sigma_{i^*}^2 )}{ \log( 1+(\mu_{i}-\phi_{i}^{\min}(r_{i}^{f}(\alpha_{i^*}^*) ))^2/\sigma_{i}^2 ) } = \Delta \frac{\Ucal^{*,\min}_{i}( r_{i}^f(\alpha_{i^*}^*) ) }{ \Ucal^{\min}_{i}( r_{i}^f(\alpha_{i^*}^*) ) }.
\end{align}
However, if \eqref{ineq:alpha2inc} holds for any $i \ne i^*$, then
\begin{align*}
	\bar{\alpha}_{i^*,1} + \sum_{i \ne i^*} \alpha_{i}^{f}(\bar{\alpha}_{i^*,1}) >& \alpha_{i^*}^* - \Delta + \sum_{i \ne i^*} \alpha_{i}^{f}(\alpha_{i^*}^*)  + \Delta \sum_{i \ne i^*} \frac{\Ucal^{*,\min}_{i}( r_{i}^f(\alpha_{i^*}^*) ) }{ \Ucal^{\min}_{i}( r_{i}^f(\alpha_{i^*}^*) ) }
	\ge \alpha_{i^*}^* + \sum_{i \ne i^*} \alpha_{i}^{f}(\alpha_{i^*}^*) = 1,
\end{align*}
where the last inequality holds by \eqref{ineq:urelation1}.
This result is contradictory to the constraint $\bar{\alpha}_{i^*,1} + \sum_{i \ne i^*} \alpha_{i}^{f}(\bar{\alpha}_{i^*,1}) = 1$. Thus, when $\bar{\alpha}_{i^*,1} < \alpha_{i^*}^*$,
\begin{align*}
	\min_{i \ne i^*} \Vcal_{i} ( \alpha_{i}^{f}(\bar{\alpha}_{i^*,1}), \bar{\alpha}_{i^*,1} ) < \min_{i \ne i^*} \Vcal_{i} ( \alpha_{i}^{f}(\alpha_{i^*}^*), \alpha_{i^*}^* ).
\end{align*}

By Lemma \ref{lem:umono}, we have
\begin{align}\label{ineq:urelation3}
	&\sum_{i \ne i^*} \frac{\Ucal^{*,\max}_{i}( r_{i}^f(\alpha_{i^*}^*) ) }{ \Ucal^{\max}_{i}( r_{i}^f(\alpha_{i^*}^*) ) } = \sum_{i \ne i^*} \frac{\log( 1+(\mu_{i^*}-\phi_{i}^{\max}(r_{i}^{f}(\alpha_{i^*}^*)))^2/\sigma_{i^*}^2 )}{ \log( 1+(\mu_{i}-\phi_{i}^{\max}(r_{i}^{f}(\alpha_{i^*}^*) ))^2/\sigma_{i}^2 ) } \le 1.
\end{align}
Let $\bar{\alpha}_{i^*,2} = \alpha_{i^*}^* + \Delta$ with $0< \Delta < 1-\alpha_{i^*}^*$. We show $\Vcal_{i} ( \alpha_{i}^{f}(\bar{\alpha}_{i^*,2}), \bar{\alpha}_{i^*,2} ) < \Vcal_{i} ( \alpha_{i}^{f}(\alpha_{i^*}^*), \alpha_{i^*}^* )$ for some $i \ne i^*$ by contradiction.
Suppose $\Vcal_{i} ( \alpha_{i}^{f}(\bar{\alpha}_{i^*,2}), \bar{\alpha}_{i^*,2} ) \ge \Vcal_{i} ( \alpha_{i}^{f}(\alpha_{i^*}^*), \alpha_{i^*}^* )$ for any $i \ne i^*$. Making similar arguments to those used to obtain \eqref{ineq:alpha2inc}, if $\Vcal_{i} ( \alpha_{i}^{f}(\bar{\alpha}_{i^*,2}), \bar{\alpha}_{i^*,2} ) \ge \Vcal_{i} ( \alpha_{i}^{f}(\alpha_{i^*}^*), \alpha_{i^*}^* )$ is true, then
\begin{align}\label{ineq:ul1}
	\alpha_{i}^{f}(\bar{\alpha}_{i^*,2}) - \alpha_{i}^{f}(\alpha_{i^*}^*) >  -\Delta \frac{\log( 1+(\mu_{i^*}-\phi_{i}^{\max}(r_{i}^{f}(\alpha_{i^*}^*)))^2/\sigma_{i^*}^2 )}{ \log( 1+(\mu_{i}-\phi_{i}^{\max}(\alpha_{i}^{f}(r_{i^*}^*) ))^2/\sigma_{i}^2 ) } = -\Delta \frac{\Ucal^{*,\max}_{i}( r_{i}^f(\alpha_{i^*}^*) ) }{ \Ucal^{\max}_{i}( r_{i}^f(\alpha_{i^*}^*) ) }.
\end{align}
However, if \eqref{ineq:ul1} holds for any $i \ne i^*$, then
\begin{align*}
	\bar{\alpha}_{i^*,2} + \sum_{i \ne i^*} \alpha_{i}^{f}(\bar{\alpha}_{i^*,2}) >& \alpha_{i^*}^* + \Delta + \sum_{i \ne i^*} \alpha_{i}^{f}(\alpha_{i^*}^*)  - \Delta \sum_{i \ne i^*} \frac{\Ucal^{*,\max}_{i}( r_{i}^f(\alpha_{i^*}^*) ) }{ \Ucal^{\max}_{i}( r_{i}^f(\alpha_{i^*}^*) ) }
	\ge \alpha_{i^*}^* + \sum_{i \ne i^*} \alpha_{i}^{f}(\alpha_{i^*}^*)
	= 1,
\end{align*}
where the last inequality holds by \eqref{ineq:urelation3}.
This result is contradictory to the constraint $\bar{\alpha}_{i^*,2} + \sum_{i \ne i^*} \alpha_{i}^{f}(\bar{\alpha}_{i^*,2}) = 1$. Thus, when $\bar{\alpha}_{i^*,2} > \alpha_{i^*}^*$,
\begin{align*}
	\min_{i \ne i^*} \Vcal_{i} ( \alpha_{i}^{f}(\bar{\alpha}_{i^*,2}), \bar{\alpha}_{i^*,2} ) < \min_{i \ne i^*} \Vcal_{i} ( \alpha_{i}^{f}(\alpha_{i^*}^*), \alpha_{i^*}^* ),
\end{align*}
which completes the proof. $\square$

\section{Proofs for Section \ref{sec:algorithms}}\label{sec:proofalgs}

In the following, we prove Lemmas \ref{lem:const}-\ref{lem:one}, Proposition \ref{proposv}, and Theorem \ref{th:algconv}.

\subsection{Proof of Lemma \ref{lem:const}}\label{sec:alg_const}

By Lemmas \ref{lem:bayes_consist}-\ref{lem:posterior_non}, there exist $b_{eU}$, $b_{eL}$, $b_{vU}$ and $b_{vL}$ such that $b_{eL} < \hat{\mu}^m_{i} < b_{eU}$ and $0 < b_{vL} < \left(\hat{\sigma}^m_{i}\right)^2 < b_{vU}$ for all $i$ and all sufficiently large $n$.

We use the following technical result. It is stated without proof because the arguments are very similar to those used to show Lemma \ref{lem:phimono0}(iii).

\begin{lemma}\label{lem:phibd_new}
	If $\hat{\alpha}_{i}^m/\hat{\alpha}_{i^*}^m \le 1/b_0 $ for $b_0 \ge 1$, then $\hat{\mu}_{i^*}^m - \hat{\phi}_{i}^{m} \le \bar{\eta}(b_0) (\hat{\mu}_{i^*}^m-\hat{\mu}_{i}^m)$. On the other hand, if $\hat{\alpha}_{i}^m/\hat{\alpha}_{i^*}^m  \ge b_0$, then $ \hat{\phi}_{i}^{m} - \hat{\mu}_{i}^m \le \bar{\eta}(b_0) (\hat{\mu}_{i^*}^m-\hat{\mu}_{i}^m)$ where $\bar{\eta}(b_0) =  (b_{vU} + (b_{eU}-b_{eL})^2)/  (b_0  b_{vL}) $.
\end{lemma}

Let $\Ecal$ denote the set of designs such that $i \in \Ecal$ if and only if $N_{i}^m \to \infty$ as $m \to \infty$. Let $\bar{M}_1$ denote the random time such that for all $m \ge \bar{M}_1$, 1) $|\hat{\mu}_{i}^m - \mu_{i}| \le \varepsilon$ and $|(\hat{\sigma}_{i}^m)^2 - \sigma_{i}^2| \le \varepsilon$ where $\varepsilon \le \Lambda \triangleq \min \{\min_{j \ne j'} |\mu_{j} - \mu_{j'}| / 4 ,  \sigma_{\min}^2/4 \}$ for $i \in \Ecal$ and 2) any design in $\Ecal^c$ is not sampled. Then if $i,i' \in \Ecal$ and $\mu_{i} > \mu_{i'}$, we have $ \hat{\mu}_{i}^m - \hat{\mu}_{i'}^m \ge \mu_{i} - \mu_{i'} - 2 \varepsilon \ge 2\Lambda$ for $m \ge \bar{M}_1$. The remainder of the proof has two parts. In the first part, we show $\Ecal$ must contain at least two designs; in the second part, we show $\Ecal$ must contain all designs.

\textbf{Step 1: $\Ecal$ must contain at least two designs.} We proceed by contradiction. Suppose $\Ecal$ has only one design. Denote the design in $\Ecal$ by $i$. Let $b_{\alpha 1} \triangleq (4 (k-1) ( b_{vU}+(b_{eU} - b_{eL})^2 )^3 / b_{vL}^3 )^{1/2} + 1$. Let $\tilde{M}_1 \ge \bar{M}_1$ large enough such that $N_{i}^m/N_{j}^m \ge b_{\alpha 1}$ for all $j \ne i$ and all $m \ge \tilde{M}_1$. Notice that any design in $\Ecal^c$ is not sampled for $m \ge \tilde{M}_1 \ge \bar{M}_1$. Consider any $m \ge \tilde{M}_1$.
\begin{enumerate}
	\item If $\arg\max_{i'=1,\dots,k} \hat{\mu}_{i'}^m$ is not unique and $\hat{\mu}_{i}^m = \hat{\mu}_{i^{*,m}}^m$, then $i^{*,m} \ne i$ because $N_{i}^m \ge b_{\alpha 1} N_{j}^m > N_{j}^m$ for all $j \ne i$ and design $i^{*,m}$ has the smallest number of samples among all designs in $\arg\max_{i'=1,\dots,k} \hat{\mu}_{i'}^m$. By Step 1 of $\ocbau$, the design sampled at iteration $m+1$ is $i^{m+1} = i^{*,m}$, which is not $i$. This contradicts the definition of $\tilde{M}_1$ that only $i$ can be sampled for all $m+1 \ge \tilde{M}_1$.
	
	\item If $\arg\max_{i'=1,\dots,k} \hat{\mu}_{i'}^m$ is not unique and $\hat{\mu}_{i}^m \ne \hat{\mu}_{i^{*,m}}^m$, then it is obvious that $i^{*,m} \ne i$. By Step 1 of $\ocbau$, the design sampled at iteration $m+1$ is $i^{m+1} = i^{*,m}$, which is not $i$. This contradicts the definition of $\tilde{M}_1$ again.
	
	\item If $\arg\max_{i'=1,\dots,k} \hat{\mu}_{i'}^m$ is unique and $i^{*,m} = i$, then by Lemma \ref{lem:phibd_new}, we have for any $j \ne i$ that $\hat{\mu}_{i^{*,m}}^m - \hat{\phi}_{j}^m \le \bar{\eta}(b_{\alpha 1}) ( \hat{\mu}_{i^{*,m}}^m - \hat{\mu}_{j}^m )$. By the definition of $\bar{\eta}(b_{\alpha 1})$ and $b_{\alpha 1}$, we have
	\begin{align*}
		\bar{\eta}(b_{\alpha 1}) = \frac{b_{vU} + (b_{eU}-b_{eL})^2}{b_{\alpha 1}  b_{vL}} < \Big(\frac{b_{vL}^3}{4 (k-1) ( b_{vU}+(b_{eU} - b_{eL})^2 )^3}\Big)^{\frac{1}{2}}  \frac{b_{vU} + (b_{eU}-b_{eL})^2}{  b_{vL}} \le \frac{1}{2}
	\end{align*}
	such that $\hat{\phi}_{j}^m - \hat{\mu}_{j}^m = ( \hat{\mu}_{i^{*,m}}^m - \hat{\mu}_{j}^m ) - (\hat{\mu}_{i^{*,m}}^m - \hat{\phi}_{j}^m) \ge ( \hat{\mu}_{i^{*,m}}^m - \hat{\mu}_{j}^m )/2$.
	Then
	\begin{align*}
		&\hat{\Ucal}_{j}^m = \log \big(1 + (\hat{\mu}_{j}^m-\hat{\phi}_{j}^m)^2/(\hat{\sigma}_{j}^{m})^2 \big) \ge \frac{ (\hat{\mu}_{j}^m-\hat{\phi}_{j}^m)^2/(\hat{\sigma}_{j}^{m})^2 }{ 1 + (\hat{\mu}_{j}^m-\hat{\phi}_{j}^m)^2/(\hat{\sigma}_{j}^{m})^2  } \ge \frac{(\hat{\mu}_{j}^m-\hat{\mu}_{i^{*,m}}^m)^2}{4 ( b_{vU}+(b_{eU} - b_{eL})^2 ) }.
	\end{align*}
	Meanwhile,
	\begin{align*}
		&\hat{\Ucal}_{j}^{*,m} = \log \Big(1 + \frac{\big(\hat{\mu}_{i^{*,m}}^m-\hat{\phi}_{j}^m\big)^2}{(\hat{\sigma}_{i^{*,m}}^m)^2} \Big) \le \frac{ \bar{\eta}(b_{\alpha 1})^2 ( \hat{\mu}_{i^{*,m}}^m - \hat{\mu}_{j}^m )^2}{(\hat{\sigma}_{i^{*,m}}^m)^2} \le \frac{ \bar{\eta}(b_{\alpha 1})^2 ( \hat{\mu}_{i^{*,m}}^m - \hat{\mu}_{j}^m )^2}{b_{vL}},
	\end{align*}
	which, together with the lower bound of $\hat{\Ucal}_{j}^m$, leads to
	\begin{align}\label{ineq:const_ub1}
		\frac{\hat{\Ucal}_{j}^{*,m}}{\hat{\Ucal}_{j}^m} \le \bar{\eta}(b_{\alpha 1})^2  \frac{4 ( b_{vU}+(b_{eU} - b_{eL})^2 ) }{ b_{vL} } = \frac{4(b_{vU} + (b_{eU}-b_{eL})^2)^3}{b_{\alpha 1}^2  b_{vL}^3} < \frac{1}{k-1},
	\end{align}
	where the last inequality holds by $b_{\alpha 1}$'s definition.
	Then, $\sum_{j \ne i^{*,m}} \hat{\Ucal}_{j}^{*,m}/\hat{\Ucal}_{j}^m < 1$, which means $i^{m+1} \ne i^{*,m}$ and thus $i^{m+1} \ne i$. This contradicts the definition of $\tilde{M}_1$ again.
	
	\item If $\arg\max_{i'=1,\dots,k} \hat{\mu}_{i'}^m$ is unique and $i^{*,m} \ne i$, then by Lemma \ref{lem:phibd_new}, we have that $ \hat{\phi}_{i}^m - \hat{\mu}_{i}^m \le \bar{\eta}(b_{\alpha 1}) ( \hat{\mu}_{i^{*,m}}^m - \hat{\mu}_{i}^m )$. Since $\bar{\eta}(b_{\alpha 1}) \le 1/2$, we have  $\hat{\mu}_{i^{*,m}}^m-\hat{\phi}_{i}^m \ge ( \hat{\mu}_{i^{*,m}}^m - \hat{\mu}_{i}^m )/2$. Then
	\begin{align*}
		&\hat{\Ucal}_{i}^{*,m} = \log \big(1 + \big(\hat{\mu}_{i^{*,m}}^m-\hat{\phi}_{i}^m\big)^2/(\hat{\sigma}_{i^{*,m}}^m)^2 \big) \ge \frac{ \big(\hat{\mu}_{i^{*,m}}^m-\hat{\phi}_{i}^m\big)^2 }{ (\hat{\sigma}_{i^{*,m}}^m)^2 + \big(\hat{\mu}_{i^{*,m}}^m-\hat{\phi}_{i}^m\big)^2  } \ge \frac{(\hat{\mu}_{i}^m-\hat{\mu}_{i^{*,m}}^m)^2}{4 ( b_{vU}+(b_{eU} - b_{eL})^2 ) }, \\
		&\hat{\Ucal}_{i}^{m} = \log \big(1 + (\hat{\mu}_{i}^m-\hat{\phi}_{i}^m)^2/(\hat{\sigma}_{i}^m)^2 \big) \le \bar{\eta}(b_{\alpha 1})^2 ( \hat{\mu}_{i^{*,m}}^m - \hat{\mu}_{i}^m )^2/b_{vL},
	\end{align*}
	which leads to
	\begin{align}\label{ineq:const_ub2}
		\frac{\hat{\Ucal}_{i}^{*,m}}{\hat{\Ucal}_{i}^{m}} \ge \frac{b_{vL}}{4 \bar{\eta}(b_{\alpha 1})^2 ( b_{vU}+(b_{eU} - b_{eL})^2 ) } = \frac{b_{\alpha 1}^2  b_{vL}^3}{4(b_{vU} + (b_{eU}-b_{eL})^2)^3}  > 1.
	\end{align}
	Then $\sum_{i' \ne i^{*,m}} \hat{\Ucal}_{i'}^{*,m}/\hat{\Ucal}_{i'}^{m} \ge \hat{\Ucal}_{i}^{*,m}/\hat{\Ucal}_{i}^{m} > 1$ such that $i^{m+1} = i^{*,m}$ and thus $i^{m+1} \ne i$. This contradicts the definition of $\tilde{M}_1$.
\end{enumerate}
In summary, set $\Ecal$ must have at least two designs.

\textbf{Step 2: $\Ecal$ must contain all designs.} We show this part by contradiction. Suppose $\Ecal^c (\triangleq \{1,\dots,k\} \setminus \Ecal)$ is non-empty. Let $b_{\alpha 2} = \max\{ b_{\alpha 1}, \log (1 + (b_{eU}-b_{eL})^2/b_{v L}) / \log \big(1 + \Lambda^2/b_{vU} \big)+1 \}$. Let $\ddot{M}_1 \ge \bar{M}_1$ large enough such that $N_{j}^m$ remains unchanged for all $j \in \Ecal^c$ and $N_{i}^m/N_{j}^m \ge b_{\alpha 2}$ for all $i \in \Ecal$, all $j \in \Ecal^c$ and $m \ge \ddot{M}_1$.

Notice that $\arg\max_{i'=1,\dots,k} \hat{\mu}_{i'}^m$ must be unique for all $m \ge \ddot{M}_1$. On the one hand, the posterior means of two designs in $\Ecal$ are unequal because  $|\hat{\mu}_{i}^m - \mu_{i}| \le \varepsilon \le \min_{j \ne j'} |\mu_{j} - \mu_{j'}| / 4$ for any $i \in \Ecal$ and $m \ge \bar{M}_1$. On the other hand, if there exist design $j \in \Ecal^c$ and design $i \ne j$ such that $\hat{\mu}_{j}^m = \hat{\mu}_{i}^m = \max_{i'=1,\dots,k} \hat{\mu}_{i'}^m$, then $i^{*,m}$ must be design $j$ or another design in $\Ecal^c$ whose posterior mean happens to be equal $\hat{\mu}_{j}^m$. In this case, $i^{*,m} \in \Ecal^c$, which will be sampled at iteration $m+1$ according to Step 1 of $\ocbau$. This contradicts the definition of $\ddot{M}_1$.

Moreover, $i^{*,m}$ must be in $\Ecal$ for all $m \ge \ddot{M}_1$. Otherwise, suppose $i^{*,m} \in \Ecal^c$. For any $i \in \Ecal$, we have by Lemma \ref{lem:phibd_new} that $ \hat{\phi}_{i}^{m} - \hat{\mu}_{i}^m \le \bar{\eta}(b_{\alpha 2}) ( \hat{\mu}_{i^{*,m}}^m - \hat{\mu}_{i}^m )$. Since $\bar{\eta}(b_{\alpha 2}) \le \bar{\eta}(b_{\alpha 1}) \le 1/2$, we have  $\hat{\mu}_{i^{*,m}}^m-\hat{\phi}_{i}^{m} \ge ( \hat{\mu}_{i^{*,m}}^m - \hat{\mu}_{i}^m )/2$. Then, similar to the first part of proof, we have
\begin{align*}
	&\hat{\Ucal}_{i}^{*,m} = \log \big(1 + \big(\hat{\mu}_{i^{*,m}}^m-\hat{\phi}_{i}^{m}\big)^2/(\hat{\sigma}_{i^{*,m}}^m)^2 \big) \ge \frac{(\hat{\mu}_{i}^m-\hat{\mu}_{i^{*,m}}^m)^2}{4 ( b_{vU}+(b_{eU} - b_{eL})^2 ) }, \\
	&\hat{\Ucal}_{i}^{m} = \log \big(1 + (\hat{\mu}_{i}^m-\hat{\phi}_{i}^{m})^2/(\hat{\sigma}_{i}^m)^2 \big) \le \bar{\eta}(b_{\alpha 2})^2 ( \hat{\mu}_{i^{*,m}}^m - \hat{\mu}_{i}^m )^2/b_{vL},
\end{align*}
which leads to
$\hat{\Ucal}_{i}^{*,m}/\hat{\Ucal}_{i}^{m}  > 1 $.
Then $\sum_{i' \ne i^{*,m}} \hat{\Ucal}_{i'}^{*,m}/\hat{\Ucal}_{i'}^{m} > 1$ such that $i^{m+1} = i^{*,m} \in\Ecal^c$. This contradicts the definition of $\ddot{M}_1$ that $N_{j}^m$ remains unchanged for all $j \in \Ecal^c$ and $m \ge \ddot{M}_1$.

Notice that $\hat{\mu}_{i}^m < \hat{\mu}_{i'}^m$ if $\mu_i < \mu_{i'}$, $i,i' \in \Ecal$ and $m \ge \bar{M}_1$. Since $\arg\max_{i'=1,\dots,k} \hat{\mu}_{i'}^m$ is unique and $i^{*,m} \in \Ecal$ for all $m \ge \ddot{M}_1 \ge \bar{M}_1$, there exists $\bar{i} \in \Ecal$ such that $i^{*,m} = \bar{i}$ for all $m \ge \ddot{M}_1$. Consider the iteration $m \ge \ddot{M}_1$ where a design $i \ne \bar{i}$ with $ i \in \Ecal$ is sampled at $m+1$. Then
\begin{align}
	2N^{m} \hat{\Vcal}_{i}^m =&  N_{i}^m \log \left(1 + (\hat{\mu}_{i}^m-\phi_{i}^{m})^2/(\hat{\sigma}_{i}^m)^2 \right) + N_{\bar{i}}^m \log \left(1 + (\hat{\mu}_{\bar{i}}^m-\phi_{i}^{m})^2/(\hat{\sigma}_{\bar{i}}^m)^2 \right)  \nonumber \\
	\ge& \min\{ N_{i}^m, N_{\bar{i}}^m \} \log \big(1 + \Lambda^2/b_{vU} \big) \label{ineq:const_ub3} \\
	>& N_{j}^m \log (1 + (b_{eU}-b_{eL})^2/b_{v L} ),  \label{ineq:const_ub4}
\end{align}
where \eqref{ineq:const_ub3} holds because $\max\{ (\hat{\mu}_{i}^m-\phi_{i}^{m})^2, (\hat{\mu}_{\bar{i}}^m-\phi_{i}^{m})^2 \} \ge (\hat{\mu}_{\bar{i}}^m - \hat{\mu}_{i}^m)^2/4 \ge \Lambda^2$ and \eqref{ineq:const_ub4} holds because
\begin{align*}
	\min\{ N_{i}^m, N_{\bar{i}}^m \}/N_{j}^m \ge b_{\alpha 2} > \log (1 + (b_{eU}-b_{eL})^2/b_{v L}) / \log \big(1 + \Lambda^2/b_{vU} \big).
\end{align*}
Meanwhile, for design $j \in \Ecal^c$, we have $2N^{m} \hat{\Vcal}_{j}^m  \le  N_{j}^m \log (1 + (b_{eU}-b_{eL})^2/b_{v L} )$
by the definition of $b_{m U}$, $b_{m L}$ and $b_{v L}$ at the beginning of this subsection. Then $\hat{\Vcal}_{j}^m < \hat{\Vcal}_{i}^m$, which means $i^{m+1} \ne i$. This contradicts the definition of iteration $r$. Thus, set $\Ecal$ must contain all designs. $\square$	

\subsection{Proof of Lemma \ref{lem:one}}

Let $b_{con} \triangleq  (4 (k-1) ( b_{vU}+(b_{eU} - b_{eL})^2 )^3 / b_{vL}^3 )^{1/2} + 1$,
$b_{con2} = 2b_{con}  \log (1 + b_{\mu U}^2/b_{v L} ) / \log (1 + b_{\mu L}^2/(4b_{vU})) $. We prove the result for $b_{\alpha U} \triangleq 4 b_{con} b_{con2} $, $b_{\alpha L} \triangleq 1/b_{\alpha U}$,  $b_{\alpha U 2} \triangleq (1+(k-1)b_{\alpha L})^{-1}$ and $b_{\alpha L 2} \triangleq (1+(k-1)b_{\alpha U})^{-1}$.

First, we summarize the conclusions obtained from Lemma \ref{lem:const} together with Lemma \ref{lem:bayes_consist}. Let $\varepsilon_1 \triangleq \min\{\sigma_{\min}^2/8 , \min_{j \ne j'}(\mu_{j}-\mu_{j'})/8, 1/(2k), \bar{\epsilon}/4\}$. For each design $i$ and any $\varepsilon \le \varepsilon_1$, there exists $M_1(\varepsilon)$ large enough such that $|\hat{\mu}_{i}^m - \mu_i| \le \varepsilon$ and $|(\hat{\sigma}_{i}^m)^2 - \sigma_i^2| \le \varepsilon$ for all $m \ge M_1(\varepsilon)$. Consider any $m \ge M_1(\varepsilon_1)$, we have $ i^{*,m} = i^* $ and there exist $b_{\mu U}$ and $b_{\mu L}$ such that $0 < b_{\mu L} < \hat{\mu}^m_{i^*} - \hat{\mu}^m_{i} < b_{\mu U}$ for $i \ne i^*$.
Moreover, $\arg\max_{i'=1,\dots,k} \hat{\mu}_{i'}^m$ is unique. Then for $m \ge M_1(\varepsilon_1)$, the design sampled at iteration $m+1$ can be the best design $i^*$ or $j^{m}$ with the minimum value of $\hat{\Vcal}_{i}^m$ only.

\textbf{Step 1: Upper bound of $N_{i^*}^m/\max_{i \ne i^*} N_{i}^m$.} Suppose $ N_{i^*}^{m}/\max_{i \ne i^*} N_{i}^{m} >  b_{con}$ at $m \ge M_1(\varepsilon_1)$. By Lemma \ref{lem:phibd_new}, we have $\hat{\mu}_{i^*}^{m} - \hat{\phi}_{i}^{m}  \le \bar{\eta}(b_{con}) (\hat{\mu}_{i^*}^{m} - \hat{\mu}_{i}^{m})$ for any $i \ne i^*$. Similar to \eqref{ineq:const_ub1} in the proof of Lemma \ref{lem:const}, we have
\begin{align*}
	& \hat{\Ucal}_{i}^{*,m} \le  \bar{\eta}(b_{con})^2 ( \hat{\mu}_{i^*}^{m} - \hat{\mu}_{i}^{m} )^2/b_{vL}, \quad
	\hat{\Ucal}_{i}^{m} \ge (\hat{\mu}_{i^*}^{m}-\hat{\mu}_{i}^{m})^2 / (4 ( b_{vU}+(b_{eU} - b_{eL})^2 ))
\end{align*}
such that $\hat{\Ucal}_{i}^{*,m}/\hat{\Ucal}_{i}^{m} < 1/(k-1)$. Then $\sum_{i \ne i^*} \hat{\Ucal}_{i}^{*,m} /\hat{\Ucal}_{i}^{m} < 1$, which means $i^{m+1} \ne i^*$ by $\ocbau$. Then, when $N_{i^*}^{m}/\max_{i \ne i^*} N_{i}^{m} > b_{con}$ for $m \ge M_1(\varepsilon_1)$, $\ocbau$ will not sample the design $i^*$ at iteration $m+1$ and the ratio $N_{i^*}^m/\max_{i \ne i^*} N_{i}^m$ will not increase from iteration $m$ to $m+1$.

Let $\bar{M}_{2} = k(M_1(\varepsilon_1)+n_0)$. There must exist $M_1(\varepsilon_1)\le m \le \bar{M}_{2}$ such that $N_{i^*}^{m}/\max_{i \ne i^*} N_{i}^{m} \le b_{con}$.  Otherwise, if $N_{i^*}^m / \max_{i \ne i^*}N_{i}^m > b_{con}$ for all $M_1(\varepsilon_1) \le m \le \bar{M}_{2}$, then $i^{m+1} \ne i^*$ for any iteration $M_1(\varepsilon_1) \le m \le \bar{M}_{2}$, which implies $N_{i^*}^{\bar{M}_{2}} =  N_{i^*}^{M_1(\varepsilon_1)} \le M_1(\varepsilon_1)+n_0$. Meanwhile, there must exist $i \ne i^*$ satisfying $N_{i}^{\bar{M}_{2}} \ge  (N^{\bar{M}_2} - N^{M_1(\varepsilon_1)}) / (k-1) \ge  M_1(\varepsilon_1)+n_0$. This leads to that $N_{i^*}^{\bar{M}_{2}}/\max_{i \ne i^*} N_{i}^{\bar{M}_{2}} \le 1 \le b_{con}$, contradictory to the assumption $N_{i^*}^{\bar{M}_{2}}/\max_{i \ne i^*}N_{i}^{\bar{M}_{2}} > b_{con}$.

Note that $N_{i}^m \ge n_0$ such that $1/N_{i}^m \le 1 $ for any $i$ and $m \ge M_1(\varepsilon_1)$. Starting from the iteration $m$ where $M_1(\varepsilon_1) \le m \le \bar{M}_{2}$ and $N_{i^*}^m/\max_{i \ne i^*} N_{i}^m \le b_{con}$, even if there exists an iteration $m' \ge m $ such that $N_{i^*}^{m'} / \max_{i \ne i^*} N_{i}^{m'} $ increases from a value less than $b_{con}$ to a value greater than $b_{con}$, we have $N_{i^*}^{m'+1} / \max_{i \ne i^*} N_{i}^{m'+1} \le N_{i^*}^{m'} / \max_{i \ne i^*} N_{i}^{m'} + (\max_{i \ne i^*} N_{i}^{m'})^{-1} \le  b_{con} + 1 \le 2b_{con}$ and $N_{i^*}^m / \max_{i \ne i^*} N_{i}^m$ decreases with $m \ge m'+1$ until the ratio is smaller than $b_{con}$ again. Thus, we have $ N_{i^*}^m / \max_{i \ne i^*} N_{i}^m \le 2b_{con}$ for $m \ge \bar{M}_2$.

\textbf{Step 2: Upper bound of $\max_{i \ne i^*} N_{i}^m/N_{i^*}^m$.} Suppose $\max_{i \ne i^*} N_{i}^m/N_{i^*}^m > b_{con}$ when $m \ge M_1(\varepsilon_1)$.
Without the loss of generality, suppose $i_1 = \arg\max_{i \ne i^*} N_{i}^m$. Since $N_{i_1}^m/N_{i^*}^m > b_{con} $, we have $\hat{\phi}_{i_1}^m - \hat{\mu}_{i_1}^m < \bar{\eta}(b_{con}) (\hat{\mu}_{i^*}^m - \hat{\mu}_{i_1}^m)$ by Lemma \ref{lem:phibd_new}. Similar to  \eqref{ineq:const_ub2} in the proof of Lemma \ref{lem:const}, we have $\bar{\eta}(b_{con}) < \frac{1}{2}$
such that $\hat{\mu}_{i^*}^m-\hat{\phi}_{i_1}^m = ( \hat{\mu}_{i^*}^m - \hat{\mu}_{i_1}^m ) - (\hat{\phi}_{i_1}^m - \hat{\mu}_{i_1}^m) \ge ( \hat{\mu}_{i^*}^m - \hat{\mu}_{i_1}^m )/2$. Then
\begin{align*}
	&\hat{\Ucal}_{i_1}^{*,m} \ge (\hat{\mu}_{i^*}^m-\hat{\mu}_{i_1}^m)^2 / (4 ( b_{vU}+(b_{eU} - b_{eL})^2 )) , \quad
	\hat{\Ucal}_{i_1}^m \le \bar{\eta}(b_{con})^2 ( \hat{\mu}_{i^*}^m - \hat{\mu}_{i_1}^m )^2/b_{vL},
\end{align*}
which leads to $\hat{\Ucal}_{i_1}^{*,m}/\hat{\Ucal}_{i_1}^{m} > 1$.
By $\ocbau$, $i^{m+1} = i^*$ at iteration $m+1$.
By similar arguments to the first part of this proof, we can show that there exists $m$ with $M_1(\varepsilon_1)\le m \le \bar{M}_{2}$ such that $\max_{i \ne i^*} N_{i}^{m}/N_{i^*}^{m} \le b_{con}$ and thus, $ \max_{i \ne i^*} N_{i}^m/N_{i^*}^m \le 2b_{con}$ for $m \ge \bar{M}_2$.

\textbf{Step 3: Upper bound of $\max_{i \ne i^*} N_{i}^m/\min_{i \ne i^*} N_{i}^m$.}
Suppose $ \max_{i \ne i^* } N_{i}^m/ \min_{i \ne i^* } N_{i}^{m} > b_{con2} $ at $m \ge \bar{M}_2$ where $b_{con2} \triangleq 2b_{con}  \log (1 + b_{\mu U}^2/b_{v L} ) / \log (1 + b_{\mu L}^2/(4b_{vU}))$. Without loss of generality, suppose $i_1 = \arg\max_{i \ne i^*} N_{i}^{m}$ and $i_2 = \arg\min_{i \ne i^*} N_{i}^{m}$.  Then $ N_{i_1}^{m}/  N_{i_2}^{m} > b_{con2}  $ such that
\begin{align*}
	2N^{m} \hat{\Vcal}_{i_2}^{m}
	\le N_{i_2}^{m} \log \left(1 + (\hat{\mu}_{i_2}^{m}-\hat{\mu}_{i^*}^{m})^2/(\hat{\sigma}_{i_2}^{m})^2 \right)  < N_{i_1}^{m} \log (1 + b_{\mu L}^2/(4b_{vU})) / (2b_{con}),
\end{align*}
where the first inequality holds by the definition of $\hat{\Vcal}_{i_2}^{m}$ and the second one holds because $ N_{i_1}^{m}/  N_{i_2}^{m} > b_{con2}  $. Meanwhile,
\begin{align*}
	2N^{m} \hat{\Vcal}_{i_1}^{m}
	\ge& \min\{N_{i_1}^{m}, N_{i^*}^{m} \} \log (1 + b_{\mu L}^2/(4b_{vU}))  \ge  N_{i_1}^{m} \log (1 + b_{\mu L}^2/(4b_{vU})) / (2b_{con}),
\end{align*}
where the first inequality holds by a similar reason to \eqref{ineq:const_ub3} and the second inequality holds because $N_{i^*}^{m} \ge  N_{i_1}^{m}/(2b_{con})$ by Step 2 of this proof.
By $\ocbau$, $i^{m+1} = i^*$ or $i^{m+1} = i_2$ at iteration $m+1$, neither of which is $i_1$. It means that when $ \max_{i \ne i^* } N_{i}^{m}/ \min_{i \ne i^* } N_{i}^{m} >  b_{con2}  $ for $m \ge \bar{M}_2$, $\ocbau$ will not sample the design $\arg\max_{i \ne i^*} N_{i}^{m}$ at iteration $m+1$ and the ratio $\max_{i \ne i^*} N_{i}^m/\min_{i \ne i^*} N_{i}^m$ will not increase from iteration $m$ to $m+1$.

Let $\tilde{M}_{2} = \lceil (k-1+2b_{con})(\bar{M}_2+n_0)  + (k-2) \rceil > \bar{M}_2$. By similar arguments to the first part of this proof, we can show that $ \max_{i \ne i^*} N_{i}^m/\min_{i \ne i^*} N_{i}^m \le 2 b_{con2}$ for $m \ge \tilde{M}_{2} $.

\textbf{Step 4: Upper and lower bound of $N_{i}^m/N_{j}^m$, $i \ne j$.} Based on the conclusions from Steps 1-3,
if $N_{i^*}^m/\max_{i \ne i^*} N_{i}^m > 1$, then
\begin{align*}
	\frac{ \max_{i=1,\dots,k} N_{i}^m }{ \min_{i=1,\dots,k} N_{i}^m }  = \frac{ N_{i^*}^m }{ \max_{i \ne i^*} N_{i}^m } \frac{ \max_{i \ne i^*} N_{i}^m }{ \min_{i \ne i^*} N_{i}^m } \le 4 b_{con} b_{con2};
\end{align*}
if $N_{i^*}^m/\max_{i \ne i^*} N_{i}^m \le 1$, then
\begin{align*}
	\frac{ \max_{i=1,\dots,k} N_{i}^m }{ \min_{i=1,\dots,k} N_{i}^m }  = \frac{ \max_{i \ne i^*} N_{i}^m }{ \min \{ N_{i^*}^m, \min_{i \ne i^*} N_{i}^m \} } \le \max \{ 2b_{con}, 2 b_{con2} \} \le 4 b_{con} b_{con2}.
\end{align*}
Thus, letting $b_{\alpha U} \triangleq 4 b_{con} b_{con2}  $ and $b_{\alpha L} \triangleq 1/b_{\alpha U}$, we have for $m \ge M_2(\varepsilon_1) \triangleq \tilde{M}_2$ and any $i \ne j$ that $b_{\alpha L} \le N_{i}^m/N_{j}^m \le b_{\alpha U}$.
Moreover, we have
$(1+(k-1)b_{\alpha U})^{-1} \le \hat{\alpha}_{i}^m \le (1+(k-1)b_{\alpha L})^{-1}$, $i=1,\dots,k$,
because $b_{\alpha L} \hat{\alpha}_{i}^m \le \hat{\alpha}_{i'}^m \le b_{\alpha U} \hat{\alpha}_{i}^m$ such that $\hat{\alpha}_{i}^m+(k-1) b_{\alpha U} \hat{\alpha}_{i}^m \ge \sum_{i'=1}^k \hat{\alpha}_{i'}^m = 1$ and $\hat{\alpha}_{i}^m+(k-1) b_{\alpha L} \hat{\alpha}_{i}^m \le \sum_{i'=1}^k \hat{\alpha}_{i'}^m = 1$. $\square$

\subsection{Proof of Proposition \ref{proposv}}\label{sec:proofosv}

As in the proof of Lemma \ref{lem:const}, there exist $b_{eU}$, $b_{eL}$, $b_{vU}$ and $b_{vL}$ such that $b_{eL} < \hat{\mu}^m_{i} < b_{eU}$ and $0 < b_{vL} < \left(\hat{\sigma}^m_{i}\right)^2 < b_{vU}$ for all $i$ and $m \ge M_1(\varepsilon)$. Moreover, as in the proof of Lemma \ref{lem:one}, we have $b_{\alpha L}$, $b_{\alpha U}$,
$b_{\alpha L 2}$, and $b_{\alpha U 2}$ such that $b_{\alpha L} \le \hat{\alpha}^m_{i}/\hat{\alpha}^m_{j} \le b_{\alpha U}$, $b_{\alpha L 2} \le \hat{\alpha}^m_{i} \le b_{\alpha U 2}$, for all $i,j=1,2,\dots,k$ and $m \ge M_2(\varepsilon_1)$. Let $M_3(\varepsilon) \triangleq \max \{ M_1(\varepsilon), M_2(\varepsilon_1) \}$ for any $\varepsilon \le \varepsilon_1$.

By the uniform continuity of $\log (1 + (x-\phi)^2/v )$ for $x \in [b_{eL},b_{eU}]$, $v \in [b_{vL}, b_{vU}]$ and $\phi \in [b_{eL},b_{eU}]$, there exists $b_{lU} > 0$ such that for any $\varepsilon < \varepsilon_1$ and $m,m' \ge M_3(\varepsilon)$, we have
\begin{align}
	& \left| \log \left(1 + (\hat{\mu}_{i}^m-\phi_i)^2/(\hat{\sigma}_{i}^m)^2 \right) - \log \left(1 + (\hat{\mu}_{i}^{m'}-\phi_i)^2/(\hat{\sigma}_{i}^{m'})^2 \right) \right| \le b_{lU} \varepsilon. \label{ineq:log_ub}
\end{align}

To describe the behavior of $\hat{\phi}_{i}^m$, we need several technical lemmas. The first is stated without proof. It can be easily proved by repeating similar arguments from the proof of Lemma \ref{lem:phimono0}(iii) and applying the result of Lemma \ref{lem:one}.

\begin{lemma}\label{lem:phibd}
	For any $m \ge M_2(\varepsilon_1)$, we have $ \hat{\mu}_{i}^m + b_{\phi} \le \hat{\phi}_i^m \le \hat{\mu}_{i^*}^m - b_{\phi}$ where $b_{\phi} \triangleq  \frac{ b_{vL} b_{\mu L}}{ 2b_{\alpha U}(b_{v U} + b_{\mu U}^2)} $.
\end{lemma}

By Lemma \ref{lem:phibd}, when $m \ge M_3(\varepsilon)$,
\begin{align}\label{ineq:kappa_r2}
	\hat{\mu}_{i}^m + b_{\phi} \le \hat{\phi}_{i}^{m} \le \hat{\mu}_{i^*}^m - b_{\phi}.
\end{align}
Let $ \varepsilon_2 \triangleq \min \{ \varepsilon_1, b_{\phi}/4 \}$. Combining \eqref{ineq:kappa_r2} with the fact that $| \hat{\mu}_{i}^m - \mu_{i} | \le \varepsilon \le \varepsilon_2 \le b_{\phi}/4$ for all $m \ge M_3(\varepsilon)$, we know for any $m,m' \ge M_3(\varepsilon)$ that
\begin{align}
	&\mu_{i} + b_{\phi}/2 \le \hat{\phi}_{i}^{m} \le \mu_{i^*} - b_{\phi}/2, \quad \hat{\mu}_{i}^{m'} + b_{\phi}/2 \le \hat{\phi}_{i}^{m} \le \hat{\mu}_{i^*}^{m'} - b_{\phi}/2, \label{ineq:kappa_r31}
\end{align}
which, together with constants $b_{\mu U}$, $b_{vL}$ and $b_{vU}$, yields
\begin{align}
	&\max\Big\{ \log \Big(1 + \frac{(\hat{\mu}_{i^*}^{m'}-\hat{\phi}_{i}^{m})^2}{(\hat{\sigma}_{i^*}^{m'})^2} \Big), \log \Big(1 + \frac{(\hat{\mu}_{i}^{m'}-\hat{\phi}_{i}^{m})^2}{(\hat{\sigma}_{i}^{m'})^2}\Big) \Big\} \le \log ( 1+b_{\mu U}^2/b_{v L} ) \triangleq b_{log1}, \label{ineq:log_ub1} \\
	&\max\Big\{ \log \Big(1 + \frac{(\hat{\mu}_{i^*}^{m'}-\hat{\phi}_{i}^{m})^2}{(\hat{\sigma}_{i^*}^{m'})^2} \Big), \log \Big(1 + \frac{(\hat{\mu}_{i}^{m'}-\hat{\phi}_{i}^{m})^2}{(\hat{\sigma}_{i}^{m'})^2}\Big) \Big\}\ge \log ( 1+b_{\phi}^2/(4b_{vU}) ) \triangleq  b_{log2} . \label{ineq:log_ub2}
\end{align}
As in the proof of Lemma \ref{lem:phimono0}, we define the first and second order derivative of $\phi$ as $\Ical( \phi, \mu, \sigma )$ and $\Hcal( \phi, \mu, \sigma )$.
Let $b_{\Ical} \triangleq  b_{\phi} / (b_{vU} + b_{\mu U}^2) $. For $m \ge  M_3(\varepsilon)$ and $\phi \in [\hat{\mu}_{i}^m + b_{\phi}/2 , \hat{\mu}_{i^*}^m - b_{\phi}/2]$,
\begin{align}
	&\Ical( \phi, \hat{\mu}_{i}^m, (\hat{\sigma}_{i}^m)^2 ) =   \frac{ 2(\phi-\hat{\mu}_{i}^m) }{ (\hat{\sigma}_{i}^m)^2 + (\phi-\hat{\mu}_{i}^m)^2 }   \ge b_{\Ical}, \quad \Ical ( \phi, \hat{\mu}_{i^*}^m, (\hat{\sigma}_{i^*}^m)^2 ) \le   -b_{\Ical}. \label{ineq:derv_lb}
\end{align}
By the uniform continuity of $\Ical( \phi, x, v )$ for $\phi \in [b_{eL},b_{eU}]$, $x \in [b_{eL},b_{eU}]$, $v \in [b_{vL}, b_{vU}]$, there exists $b_{dU} > 0$ such that
\begin{align}
	\left| \Ical ( \phi, \hat{\mu}_{i^*}^m, (\hat{\sigma}_{i^*}^m)^2 )/\Ical( \phi, \hat{\mu}_{i}^m, (\hat{\sigma}_{i}^m)^2 ) - \Ical ( \phi, \mu_{i^*}, \sigma_{i^*}^2 )/\Ical( \phi, \mu_{i}, \sigma_{i}^2 ) \right| \le b_{dU} \varepsilon. \label{ineq:diff_ub}
\end{align}
For all $\phi \in [\mu_i + b_{\phi}/2, \mu_{i^*} - b_{\phi}/2]$, the derivative of $\frac{\Ical ( \phi, \mu_{i^*}, \sigma_{i^*}^2 )}{\Ical( \phi, \mu_{i}, \sigma_{i}^2 )}$ with respect to $\phi$ satisfies
\begin{align}\label{ineq:Ider2_up}
	\left| \frac{\Hcal ( \phi, \mu_{i^*}, \sigma_{i^*}^2 ) \Ical( \phi, \mu_{i}, \sigma_{i}^2 ) - \Ical ( \phi, \mu_{i^*}, \sigma_{i^*}^2 ) \Hcal( \phi, \mu_{i}, \sigma_{i}^2 )}{\Ical( \phi, \mu_{i}, \sigma_{i}^2 )^2} \right| \le b_{\Ical \Hcal},
\end{align}
where $b_{\Ical \Hcal}$ is a constant.

We now state five intermediate technical lemmas, whose proofs are relegated to Section \ref{sec:propW_proof}. Lemmas \ref{lem:phihatmon2}-\ref{lem:upart} describe the behavior of $\hat{\phi}^m_{i}$. Lemma \ref{lem:vdiff} concerns the differences $\hat{\Wcal}^m_i - \hat{\Wcal}^m_j$. Lemma \ref{lem:interv_opt} addresses the relative rates of the empirical proportions $\hat{\alpha}^m_i$ and uses the result of Lemma \ref{lem:upart}.  Finally, Lemma \ref{lem:interv} examines the relative sampling rates and uses the results of Lemmas \ref{lem:vdiff}-\ref{lem:interv_opt}. After that, we use the results of Lemmas \ref{lem:vdiff}-\ref{lem:interv} to complete the proof of Proposition \ref{proposv}.

\begin{lemma}\label{lem:phihatmon2}
	Let $b_{\alpha \phi 2}$ denote a constant large enough. When $\varepsilon \le \varepsilon_2$ and $m,m' \ge  M_3(\varepsilon)$, solutions $\hat{\phi}_{i}^{m}$ and $\hat{\phi}_{i}^{m'}$
	satisfy
	\begin{itemize}
		\item $\hat{\phi}_{i}^{m'} - \hat{\phi}_{i}^{m} \ge  2\varepsilon^{1/2}$ if $\hat{\alpha}_{i}^{m'}/\hat{\alpha}_{i^*}^{m'} - \hat{\alpha}_{i}^m/\hat{\alpha}_{i^*}^m \le - b_{\alpha \phi 2} \varepsilon^{1/2}$
		
		\item $\hat{\phi}_{i}^{m'} - \hat{\phi}_{i}^{m} \le  -2\varepsilon^{1/2}$ if $\hat{\alpha}_{i}^{m'}/\hat{\alpha}_{i^*}^{m'} - \hat{\alpha}_{i}^m/\hat{\alpha}_{i^*}^m \ge b_{\alpha \phi 2} \varepsilon^{1/2}$.
	\end{itemize}
\end{lemma}

\begin{lemma}\label{lem:upart}
	Let $ \varepsilon_3 \le \varepsilon_2$ denote a constant small enough. When $\varepsilon \le \varepsilon_3$ and $m,m' \ge  M_3(\varepsilon)$, there exists a constant $b_{up} > 0$ such that
	\begin{itemize}
		\item if $\hat{\alpha}_{i}^{m'}/ \hat{\alpha}_{i^*}^{m'}  \le \hat{\alpha}_{i}^m/ \hat{\alpha}_{i^*}^m  -  b_{\alpha \phi 2} \varepsilon^{\frac{1}{2}}
		$, then $\hat{\Ucal}^{*,m'}_{i}/\hat{\Ucal}^{m'}_{i}
		\le \left(1 - b_{up} \varepsilon^{ \frac{1}{2} }\right) \hat{\Ucal}^{*,m}_{i}/\hat{\Ucal}^{m}_{i}$
		
		\item if $\hat{\alpha}_{i}^{m'}/ \hat{\alpha}_{i^*}^{m'}  \ge \hat{\alpha}_{i}^m/ \hat{\alpha}_{i^*}^m  +  b_{\alpha \phi 2} \varepsilon^{\frac{1}{2}}
		$, then $\hat{\Ucal}^{*,m'}_{i}/\hat{\Ucal}^{m'}_{i} \ge \left(1 + b_{up} \varepsilon^{ \frac{1}{2} }\right) \hat{\Ucal}^{*,m}_{i}/\hat{\Ucal}^{m}_{i}$.
	\end{itemize}
\end{lemma}	

\begin{lemma}\label{lem:vdiff}
	For $\varepsilon \le \varepsilon_3$, $m,m' \ge M_3(\varepsilon)$ and $m \ne m'$, we have for any non-best designs $i \ne j$ that
	\begin{align*}
		\hat{\Wcal}_{i}^{m'} -& \hat{\Wcal}_{j}^{m'}
		\le 2 (b_{\alpha U} + 1) b_{lU} \varepsilon  + \Big(\frac{\hat{\alpha}_{i}^{m'}}{\hat{\alpha}_{i^*}^{m'}} - \frac{\hat{\alpha}_{i}^m}{\hat{\alpha}_{i^*}^m}\Big) \hat{\Ucal}^{m}_{i}
		+ \hat{\Wcal}_{i}^m - \hat{\Wcal}_{j}^m +  \Big( \frac{\hat{\alpha}_{j}^m}{\hat{\alpha}_{i^*}^m} - \frac{\hat{\alpha}_{j}^{m'}}{\hat{\alpha}_{i^*}^{m'}} \Big)  \hat{\Ucal}^{m'}_{j}.
	\end{align*}
\end{lemma}

\begin{lemma}\label{lem:interv_opt}
	Let $ \varepsilon_4 \le \varepsilon_3$ denote a constant small enough and $M_4(\varepsilon) \ge M_3(\varepsilon)$ denote a large enough random time for any $\varepsilon \le \varepsilon_4$. When $\varepsilon \le \varepsilon_4$ and $m \ge M_4(\varepsilon)$, suppose $i \ne i^*$ is sampled at iteration $m+1$ and define $m' = \inf\{ s > m: i^{s+1} = i \}$. If there exists an iteration between iteration $m+1$ and $m'+1$ such that $i^*$ is simulated, there must exist a design $i^\dag$ such that $\hat{\alpha}_{i^\dag}^{m'}/ \hat{\alpha}_{i^*}^{m'}  > (\hat{\alpha}_{i^\dag}^m/ \hat{\alpha}_{i^*}^m)  ( 1- (2 b_{\alpha \phi 2}/ b_{\alpha L})  \varepsilon^{\frac{1}{2}} )$.
\end{lemma}

\begin{lemma}\label{lem:interv}
	Let $ \varepsilon_5 \le \varepsilon_4$ denote a constant small enough and $M_5(\varepsilon)  \ge M_4(\varepsilon)$ denote a large enough random time for any $\varepsilon \le \varepsilon_5$. When $\varepsilon \le \varepsilon_5$ and $m \ge M_5(\varepsilon)$, suppose $i \ne i^*$ is sampled at iteration $m+1$ and define $m' = \inf\{ s > m: i^{s+1} = i \}$. Then $N_{i_1}^{m'} / N_{i_1}^m  \le 1 + b_{\Wcal U 1} \varepsilon^{\frac{1}{2}}$ for all $i_1 \ne i$.
\end{lemma}

Now, we are ready to complete the proof of Proposition \ref{proposv}. Let $\tilde{M}_5(\varepsilon) \ge M_5(\varepsilon)$ denote the first iteration after iteration $M_5(\varepsilon)$ where $N_i^{\tilde{M}_5(\varepsilon)} \ge N_i^{M_5(\varepsilon)} +1$ for all $i=1,\dots,k$. We will show that, when $m > \tilde{M}_5(\varepsilon)$, there exists a constant $b_{\Wcal}$ such that $\max_{i,i' \ne i^*} | \hat{\Wcal}_{i}^m - \hat{\Wcal}_{i'}^{m} | \le b_{\Wcal} \varepsilon^{\frac{1}{2}}$.

Consider a design $i \ne i^*$. Let $\{m_l(i),l=1,2,\dots\}$ denote a sequence of all iterations with $m_{l+1}(i) > m_{l}(i)$, $l=1,2,\dots$, such that $i^{m_l(i)+1} = i$. Let $m_{l_0}(i) \in \{m_l(i),l=1,2,\dots\}$ be an iteration satisfying $M_5(\varepsilon) \le m_{l_0}(i) \le \tilde{M}_5(\varepsilon)$. For any $i' \ne i ,i^*$ and $m_{l}(i) < m < m_{l+1}(i)$, $l=l_0,l_0+1,\dots$, we have
\begin{align*}
	&\frac{N_{i}^m}{N_{i^*}^m} - \frac{N_{i}^{m_l(i)}}{N_{i^*}^{m_l(i)}} = \frac{N_{i}^{m_l(i)}+1}{N_{i^*}^m} - \frac{N_{i}^{m_l(i)}}{N_{i^*}^{m_l(i)}}
	\le \frac{N_{i}^{m_l(i)}+1}{N_{i^*}^{m_l(i)}} - \frac{N_{i}^{m_l(i)}}{N_{i^*}^{m_l(i)}} \le \varepsilon,
\end{align*}
where the last inequality holds because $(N_{i^*}^{m_l(i)})^{-1} \le \varepsilon$ for $m_l(i) \ge M_5(\varepsilon)$, and
\begin{align}
	\frac{N_{i'}^{m_l(i)}}{N_{i^*}^{m_l(i)}} - \frac{N_{i'}^{m}}{N_{i^*}^m} \le& \frac{N_{i'}^{m_l(i)}}{N_{i^*}^{m_l(i)}} - \frac{N_{i'}^{m_l(i)}}{N_{i^*}^{m_{l+1}(i)}} = \frac{N_{i'}^{m_l(i)}}{N_{i^*}^{m_l(i)}} - \frac{N_{i'}^{m_l(i)}}{N_{i^*}^{m_l(i)}} \frac{N_{i^*}^{m_l(i)}}{N_{i^*}^{m_{l+1}(i)}}   \nonumber  \\
	\le&  \frac{N_{i'}^{m_l(i)}}{N_{i^*}^{m_l(i)}} \Big( 1 - \frac{1}{1+b_{\Wcal U 1} \varepsilon^{\frac{1}{2}}}  \Big) \label{ineq:propW_ub1} \\
	\le& b_{\alpha U} b_{\Wcal U 1} \varepsilon^{\frac{1}{2}} \nonumber,
\end{align}
where \eqref{ineq:propW_ub1} holds by Lemma \ref{lem:interv}. Notice that $\hat{\Wcal}_{i}^{m_l(i)} \le \hat{\Wcal}_{i'}^{m_l(i)}$. By Lemma \ref{lem:vdiff}, we have
\begin{align}
	\hat{\Wcal}_{i}^m - \hat{\Wcal}_{i'}^{m}
	\le& 2 (b_{\alpha U} + 1) b_{lU} \varepsilon + (\hat{\alpha}_{i}^m/\hat{\alpha}_{i^*}^m - \hat{\alpha}_{i}^{m_l(i)}/\hat{\alpha}_{i^*}^{m_l(i)}) \hat{\Ucal}^{m_l(i)}_{i}
	+  ( \hat{\alpha}_{i'}^{m_l(i)}/\hat{\alpha}_{i^*}^{m_l(i)} - \hat{\alpha}_{i'}^{m}/\hat{\alpha}_{i^*}^m ) \hat{\Ucal}^{m}_{i'}   \nonumber \\
	\le& 2 (b_{\alpha U} + 1) b_{lU} \varepsilon + b_{log1} \varepsilon + b_{log1} b_{\alpha U} b_{\Wcal U 1} \varepsilon^{\frac{1}{2}}  \label{ineq:propW_ub3} \\
	\le&  ( 2 (b_{\alpha U} + 1) b_{lU} + b_{log1} (1 +  b_{\alpha U} b_{\Wcal U 1}) ) \varepsilon^{\frac{1}{2}} \nonumber,
\end{align}
where \eqref{ineq:propW_ub3} holds by \eqref{ineq:log_ub1}.
Then, for any $m \ge \tilde{M}_5(\varepsilon) \ge m_{l_0}(i)$, we have $\hat{\Wcal}_{i}^m - \hat{\Wcal}_{i'}^{m} \le b_{\Wcal} \varepsilon^{\frac{1}{2}}$ where $b_{\Wcal} \triangleq  2 (b_{\alpha U} + 1) b_{lU} + b_{log1} (1 +  b_{\alpha U} b_{\Wcal U 1})$. By symmetry, $\hat{\Wcal}_{i'}^{m} - \hat{\Wcal}_{i}^m \le b_{\Wcal} \varepsilon^{\frac{1}{2}}$ for $m \ge \tilde{M}_5(\varepsilon)$. Thus, for any $m \ge \tilde{M}_5(\varepsilon^2/b_{\Wcal}^2)$ with $\varepsilon^2/b_{\Wcal}^2 \le \varepsilon_5$ and $i,i' \ne i^*$, we have $| \hat{\Wcal}_{i'}^{m} - \hat{\Wcal}_{i}^m | \le  \varepsilon$. $\square$

\subsection{Proof of Theorem \ref{th:algconv}}

Finally, we complete the proof of the main convergence result. We continue to use the various constants defined in Sections \ref{sec:alg_const}-\ref{sec:proofosv}; some intermediate results from these sections will also be used. We begin with two technical lemmas concerning the behavior of the empirical approximation of \eqref{eq:totalbalance} of the main text. The first lemma is proved in Section \ref{sec:lem:optnonub_proof}. The second uses very similar arguments and is stated without proof.

\begin{lemma}\label{lem:optnonub}
	Let $ \varepsilon_6 \le \varepsilon_5$ denote a constant small enough and $M_6(\varepsilon)  \ge M_5(\varepsilon)$ denote a large enough random time for any $\varepsilon \le \varepsilon_6$. When $\varepsilon \le \varepsilon_6$, if $i^*$ is sampled at iteration $m^*+1$ for $ m^* \ge M_6(\varepsilon)$, then for all $m' \ge m^*$, $\sum_{i \ne i^*} \hat{\alpha}_{i}^{m'} / \sum_{i \ne i^*} \hat{\alpha}_{i}^{m^*} < 1 +  b_{\Ucal} \varepsilon^{1/2}$.
\end{lemma}

\begin{lemma}\label{lem:optnonlb}
	When $\varepsilon \le \varepsilon_6$, if design $i \ne i^*$ is sampled at iteration $m+1$ for $ m \ge M_6(\varepsilon)$, then for all $m' \ge m$, $\sum_{i \ne i^*} \hat{\alpha}_{i}^{m'}/\sum_{i \ne i^*} \hat{\alpha}_{i}^m > 1 -  b_{\Ucal} \varepsilon^{1/2}$.
\end{lemma}

Next, we present two propositions concerning the behavior of the empirical proportions $\hat{\alpha}^m_i$. The proofs are given in Sections \ref{sec:prop:optnonconv_proof} and \ref{sec:prop:nonconv_proof}.

\begin{proposition}\label{prop:optnonconv}
	Let $M_7(\varepsilon) \ge  M_6(\varepsilon)$ denote a large enough random time for any $\varepsilon \le \varepsilon_6$. When $\varepsilon \le \varepsilon_6$ and $m',m'' \ge  M_7(\varepsilon)$,
	\begin{align*}
		&| \sum_{i \ne i^*} \hat{\alpha}_{i}^{m'} - \sum_{i \ne i^*} \hat{\alpha}_{i}^{m''} | <   4 b_{\Ucal} \varepsilon^{1/2}, \ \  | \hat{\alpha}_{i^*}^{m'} - \hat{\alpha}_{i^*}^{m''} | <   4 b_{\Ucal} \varepsilon^{1/2}.
	\end{align*}
\end{proposition}

\begin{proposition}\label{prop:nonconv}
	Let $ \varepsilon_7 \le \varepsilon_6$ denote a constant small enough.
	When $\varepsilon \le \varepsilon_7 $ and $m',m'' \ge  M_7(\varepsilon)$, there exists a constant $b_{\alpha U 3}$ such that $|\hat{\alpha}_{i}^{m'} - \hat{\alpha}_{i}^{m''}| <   2 b_{\alpha U 3} \varepsilon^{1/2} $ for any $i \ne i^*$.
\end{proposition}

Now, we can show Theorem \ref{th:algconv} based on the previous lemmas and propositions. By Propositions \ref{prop:optnonconv} and \ref{prop:nonconv}, $(\hat{\alpha}_{1}^{m},\dots,\hat{\alpha}_{k}^{m})$ of $\ocbau$ converges as $m \to \infty$ such that there exists an allocation $\balpha^\circ = (\alpha_1^\circ,\dots,\alpha_k^\circ)$ with $\lim_{m \to \infty} \hat{\alpha}_{i}^m = \alpha_i^\circ$, $i=1,\dots,k$. Moreover, let $b_{\alpha} \triangleq \max\{4b_{\Ucal}, 2b_{\alpha U 3}\}$. For any given $\varepsilon \le \varepsilon_7$ and any $m \ge M_7(\varepsilon)$, $|\hat{\alpha}_{i}^m - \alpha_i^{\circ}| \le b_{\alpha} \varepsilon^{1/2}$. Thus for $m \ge M_7(\varepsilon^2/b_{\alpha}^2)$ where $\varepsilon$ satisfies $\varepsilon^2/b_{\alpha}^2 \le \varepsilon_7$, we have $|\hat{\alpha}_{i}^m - \alpha_i^{\circ}| \le \varepsilon$.

In the following, we show that $\balpha^\circ$ is the optimal allocation $\balpha^*$ in Theorem \ref{th:ocba}. First, we show that $\Vcal_{i}(\alpha_i^\circ,\alpha_{i^*}^\circ) = \Vcal_{j}(\alpha_j^\circ,\alpha_{i^*}^\circ)$, $i,j \ne i^*$. Suppose $m$ is large enough. For any $i \ne i^*$, we have by \eqref{ineq:log_ub} that $|\Vcal_i(\hat{\alpha}_{i}^m,\hat{\alpha}_{i^*}^m) - \hat{\Vcal}_{i}^m| \le b_{lU} \varepsilon$.
By the continuity of $\Vcal_i$ shown in Lemma \ref{lem:vmono0}, there exists $b_{\alpha \Vcal}$ such that $|\Vcal_i(\hat{\alpha}_{i}^m,\hat{\alpha}_{i^*}^m) - \Vcal_i(\alpha_{i}^\circ,\alpha_{i^*}^\circ)| \le b_{\alpha \Vcal} \varepsilon$.
Then, we have
\begin{align*}
	|\hat{\Vcal}_{i}^m - \Vcal_i(\alpha_{i}^\circ,\alpha_{i^*}^\circ)|  \le (b_{lU}+b_{\alpha \Vcal}) \varepsilon.
\end{align*}
Meanwhile, for any $i,j \ne i^*$, by Proposition \ref{proposv},
\begin{align*}
	|\hat{\Vcal}_{i}^m - \hat{\Vcal}_{j}^m| = \frac{\hat{\alpha}_{i^*}^m}{2} |\hat{\Wcal}_{i}^m - \hat{\Wcal}_{j}^m| \le \varepsilon.
\end{align*}
The above two inequalities lead to
$|\Vcal_i(\alpha_{i}^\circ,\alpha_{i^*}^\circ) - \Vcal_j(\alpha_{i}^\circ,\alpha_{i^*}^\circ)|  \le (2b_{lU}+2b_{\alpha \Vcal}+1) \varepsilon$.
Since $\varepsilon$ can be arbitrarily small,  we have $\Vcal_{i}(\alpha_i^\circ,\alpha_{i^*}^\circ) = \Vcal_{j}(\alpha_j^\circ,\alpha_{i^*}^\circ)$, $i,j \ne i^*$.

By Lemma \ref{lem:vequal}, the allocation $\balpha^\circ$ satisfies $\alpha_i^\circ = \alpha_i^f(\alpha_{i^*}^\circ)$, $i \ne i^*$. In the following, we show that $\alpha_{i^*}^\circ = \alpha_{i^*}^*$.
Notice that $r_i^f(\alpha_{i^*}) \triangleq \alpha_i^f(\alpha_{i^*})/\alpha_{i^*}$ for any $0 < \alpha_{i^*} < 1$. By similar reasons to Lemma \ref{lem:phihatmon2}, we can show that for any $m$ large enough and $\varepsilon$ small enough,
\begin{itemize}
	\item $\phi_{i}^{\min}(r_i^f(\alpha_{i^*}^*)) - \hat{\phi}_{i}^m \ge  2\varepsilon^{1/2}$ if $r_i^f(\alpha_{i^*}^*) - \hat{\alpha}_{i}^m/\hat{\alpha}_{i^*}^m \le - b_{\alpha \phi 2} \varepsilon^{1/2}$;
	
	\item $\phi_{i}^{\min}(r_i^f(\alpha_{i^*}^*)) - \hat{\phi}_{i}^m \le  -2\varepsilon^{1/2}$ if $r_i^f(\alpha_{i^*}^*) - \hat{\alpha}_{i}^m/\hat{\alpha}_{i^*}^m \ge b_{\alpha \phi 2} \varepsilon^{1/2}$.
\end{itemize}
Suppose $\alpha_{i^*}^\circ < \alpha_{i^*}^*$, which implies $r_i^f(\alpha_{i^*}^\circ) > r_i^f(\alpha_{i^*}^*)$, $i \ne i^*$ by Lemma \ref{lem:vequal}(ii). Since $\hat{\alpha}_{i}^m/\hat{\alpha}_{i^*}^m \to \alpha_i^\circ/\alpha_{i^*}^\circ = r_i^f(\alpha_{i^*}^\circ)$, for any $\varepsilon$ small enough and any $m$ large enough, we have $\hat{\alpha}_{i}^m/\hat{\alpha}_{i^*}^m > r_i^f(\alpha_{i^*}^*) + b_{\alpha \phi 2} \varepsilon^{1/2}$. Then $\phi_{i}^{\min}(r_i^f(\alpha_{i^*}^*)) - \hat{\phi}_{i}^m \ge  2\varepsilon^{1/2}$, which, by noting that $\mu_i < \hat{\phi}_{i}^m < \mu_{i^*}$ by \eqref{ineq:kappa_r31} and $\mu_i < \phi_{i}^{\min}(r_i^f(\alpha_{i^*}^*)) < \mu_{i^*}$ by Lemma \ref{lem:phimono0}, yields
\begin{align*}
	\sum_{i \ne i^*} \frac{\log \left(1 + (\mu_{i^*}-\hat{\phi}_{i}^m)^2/\sigma_{i^*}^2 \right)}{ \log \left(1 + (\mu_{i}-\hat{\phi}_{i}^m)^2/\sigma_{i}^2 \right) } > \sum_{i \ne i^*} \frac{\log (1 + (\mu_{i^*}-\phi_{i}^{\min}(r_i^f(\alpha_{i^*}^*)))^2/\sigma_{i^*}^2 )}{ \log (1 + (\mu_{i}-\phi_{i}^{\min}(r_i^f(\alpha_{i^*}^*)))^2/\sigma_{i}^2 ) } = \sum_{i\neq i^*} \frac{\Ucal^{*,\min}_i( r^f_i(\alpha_{i^*}^*) )}{\Ucal^{\min}_i( r^f_i(\alpha_{i^*}^*) )} \ge 1.
\end{align*}
Meanwhile, by the continuity in $\hat{\mu}_{j}^m$ and $(\hat{\sigma}_{j}^m)^2$, $j=i,i^*$,
\begin{align*}
	&\left| \sum_{i \ne i^*} \hat{\Ucal}_{i}^{*,m}/\hat{\Ucal}_{i}^{m} - \sum_{i \ne i^*} \frac{\log \left(1 + (\mu_{i^*}-\hat{\phi}_{i}^{m})^2/\sigma_{i^*}^2 \right)}{ \log \left(1 + (\mu_{i}-\hat{\phi}_{i}^{m})^2/\sigma_{i}^2 \right) }  \right|
\end{align*}
converges to zero as $m \to \infty$. Thus, $\sum_{i \ne i^*} \hat{\Ucal}_{i}^{*,m}/\hat{\Ucal}_{i}^{m} > 1$ for all $m$ large enough, which implies any non-best design $i \ne i^*$ will not be sampled for all $m$ large enough. This contradicts the conclusion in Lemma \ref{lem:const}. If $\alpha_{i^*}^\circ > \alpha_{i^*}^*$, we will obtain a similar contradiction.

In summary, for any $\varepsilon$ with $\varepsilon^2/b_{\alpha}^2 \le \varepsilon_7$ and $m \ge M_7(\varepsilon^2/b_{\alpha}^2)$, we have $|\hat{\alpha}_{i}^m - \alpha_i^*| \le \varepsilon$. $\square$

\section{Proofs for Section \ref{sec:appendixsec1}}\label{sec:proof_appendixsec1}

In the following, we prove Lemmas \ref{lem:mle_eoc_z}-\ref{lem:xi_ai}.

\subsection{Proof of Lemma \ref{lem:mle_eoc_z}}

Since $\bar{\phi}_i \in [\mu_i-\varepsilon,\mu_i+\varepsilon]$, $\bar{\phi}_{i^*} \in [\mu_{i^*}-\varepsilon,\mu_{i^*}+\varepsilon]$ and $\varepsilon \le \Delta \le \min_{j \ne i^*} (\mu_{i^*} - \mu_{j})/8$, we have $\bar{\phi}_i < \bar{\phi}_{i^*}$. Notice that $\log (1 + (\bar{\phi}_{i^*}-\phi_{i^*})^2/\bar{\psi}_{i^*})$ decreases with $\phi_{i^*}$ for $\phi_{i^*} \le \bar{\phi}_{i^*}$ and then increases with $\phi_{i^*}$ for $\phi_{i^*} > \bar{\phi}_{i^*}$; $\log (1 + (\bar{\phi}_{i}-\phi_{i})^2/\bar{\psi}_{i})$ decreases with $\phi_{i}$ for $\phi_i \le \bar{\phi}_{i}$ and then increases with $\phi_{i}$ for $\phi_i > \bar{\phi}_{i}$. The optimal solution $(\phi_{i}^{\min,i}(\bar{\psi}_{i},\bar{\psi}_{i^*},\bar{\phi}_{i},\bar{\phi}_{i^*}), \phi_{i^*}^{\min,i}(\bar{\psi}_{i},\bar{\psi}_{i^*},\bar{\phi}_{i},\bar{\phi}_{i^*}))$ should satisfy
\begin{align*}
	\phi_{i}^{\min,i}(\bar{\psi}_{i},\bar{\psi}_{i^*},\bar{\phi}_{i},\bar{\phi}_{i^*}) = \phi_{i^*}^{\min,i}(\bar{\psi}_{i},\bar{\psi}_{i^*},\bar{\phi}_{i},\bar{\phi}_{i^*})
\end{align*}
because one can always find a better solution if $\phi_{i} > \phi_{i^*}$. To see this: when $\phi_{i^*} \ge \bar{\phi}_{i^*}$, which together with $\phi_{i} > \phi_{i^*}$ and $ \bar{\phi}_{i^*} > \bar{\phi}_i$ implies $\phi_{i} > \bar{\phi}_{i}$, we can decrease $\phi_{i}$ to $\phi_{i^*}$ to obtain a better solution; when $\phi_{i^*} < \bar{\phi}_{i^*}$, we can increase $\phi_{i^*}$ to $\min\{\phi_{i},\bar{\phi}_{i^*}\}$ to obtain a better solution.
Thus, \eqref{ineq:phiiistar} is proved. Moreover,
\begin{align*}
	\phi_{i^*}^{\min,i}(\bar{\psi}_{i},\bar{\psi}_{i^*},\bar{\phi}_{i},\bar{\phi}_{i^*}) \in [\bar{\phi}_{i}, \bar{\phi}_{i^*}]
\end{align*}
holds, because if $\phi_{i} = \phi_{i^*} < \bar{\phi}_{i}$, both $\log (1 + (\bar{\phi}_{i^*}-\phi_{i^*})^2/\bar{\psi}_{i^*})$ and $\log (1 + (\bar{\phi}_{i}-\phi_{i})^2/\bar{\psi}_{i})$ can be reduced by a larger $\phi_{i^*}$; similarly, if $\phi_{i} = \phi_{i^*} > \bar{\phi}_{i^*}$, both $\log (1 + (\bar{\phi}_{i^*}-\phi_{i^*})^2/\bar{\psi}_{i^*})$ and $\log (1 + (\bar{\phi}_{i}-\phi_{i})^2/\bar{\psi}_{i})$ can be reduced by a smaller $\phi_{i^*}$.
Since $\bar{\phi}_i \in [\mu_i-\varepsilon,\mu_i+\varepsilon]$ and $\bar{\phi}_{i^*} \in [\mu_{i^*}-\varepsilon,\mu_{i^*}+\varepsilon]$, we have
$\phi_{i^*}^{\min,i}(\bar{\psi}_{i},\bar{\psi}_{i^*},\bar{\phi}_{i},\bar{\phi}_{i^*}) \in [ \mu_i - \varepsilon , \mu_{i^*}+\varepsilon ].
$
Thus, \eqref{ineq:eoc_phi_ep} is proved.

Notice that function $w(\bar{\psi}_{i},\bar{\psi}_{i^*},\bar{\phi}_{i},\bar{\phi}_{i^*},\phi_{i},\phi_{i^*}) $ defined in the bounded domain $\mathcal{D}_v \triangleq [\sigma_{i}^2- \Delta, \sigma_{i}^2+ \Delta] \times [\sigma_{i^*}^2- \Delta, \sigma_{i^*}^2+ \Delta] \times [\mu_{i}-\Delta, \mu_{i}+\Delta] \times [\mu_{i^*}-\Delta, \mu_{i^*}+\Delta] \times [\mu_{i}-\Delta, \mu_{i^*}+\Delta] \times [\mu_{i}-\Delta, \mu_{i^*}+\Delta]$ is a continuous function and has bounded gradients. This is because $\bar{\psi}_i \ge \sigma_i^2 - \Delta \ge \sigma_i^2 - \sigma_{\min}^2/4 \ge \sigma_i^2/2$, $\bar{\psi}_{i^*} \ge \sigma_{i^*}^2/2$, $|\bar{\phi}_i - \phi_i| \le \mu_{i^*}+\Delta - (\mu_i-\Delta) \le \mu_{\max} - \mu_{\min}+2\Delta$ and $|\bar{\phi}_{i^*} - \phi_{i^*}| \le \mu_{\max} - \mu_{\min}+2\Delta$ in domain $\mathcal{D}_v$ and $\alpha_i,\alpha_{i^*} \le 1$ such that
\begin{align*}
	&|\partial w / \partial \bar{\psi}_{i}| = \left| \frac{\alpha_{i}}{2} \frac{ -(\bar{\phi}_{i}-\phi_{i})^2/\bar{\psi}_{i}^2 }{1 + (\bar{\phi}_{i}-\phi_{i})^2/\bar{\psi}_{i}}\right| \le  \left| (\bar{\phi}_{i}-\phi_{i})^2 /\bar{\psi}_{i}^2 \right| \le (\mu_{\max}-\mu_{\min}+2\Delta)^2/(\sigma_{\min}^2/2)^2,  \\
	&|\partial w / \partial \bar{\phi}_{i}| = \left|\frac{\alpha_{i}}{2} \frac{ 2(\bar{\phi}_{i}-\phi_{i})/\bar{\psi}_{i} }{1 + (\bar{\phi}_{i}-\phi_{i})^2/\bar{\psi}_{i}}\right| \le |(\bar{\phi}_{i}-\phi_{i})/\bar{\psi}_{i}| \le (\mu_{\max}-\mu_{\min}+2\Delta)/(\sigma_{\min}^2/2),  \\
	&|\partial w / \partial \phi_{i}| = \left|\frac{\alpha_{i}}{2} \frac{ 2(\phi_{i}-\bar{\phi}_{i})/\bar{\psi}_{i} }{1 + (\bar{\phi}_{i}-\phi_{i})^2/\bar{\psi}_{i}}\right| \le |(\phi_{i}-\bar{\phi}_{i})/\bar{\psi}_{i}| \le (\mu_{\max}-\mu_{\min}+2\Delta)/(\sigma_{\min}^2/2),
\end{align*}
and similarly, $|\partial w / \partial \bar{\psi}_{i^*}|$, $|\partial w / \partial \bar{\phi}_{i^*}|$ and $|\partial w / \partial \phi_{i^*}|$ are also bounded. Then
we can find a constant $b_{a} > 0$ such that
\begin{align*}
	|w(\bar{\psi}_{i},\bar{\psi}_{i^*},\bar{\phi}_{i},\bar{\phi}_{i^*},\phi_{i},\phi_{i^*}) - w(\bar{\psi}_{i}',\bar{\psi}_{i^*}',\bar{\phi}_{i}',\bar{\phi}_{i^*}',\phi_{i}',\phi_{i^*}')| \le b_{a} \varepsilon
\end{align*}
if $\| (\bar{\psi}_{i},\bar{\psi}_{i^*},\bar{\phi}_{i},\bar{\phi}_{i^*},\phi_{i},\phi_{i^*}) - (\bar{\psi}_{i}',\bar{\psi}_{i^*}',\bar{\phi}_{i}',\bar{\phi}_{i^*}',\phi_{i}',\phi_{i^*}') \|_{\infty} \le \varepsilon$. Combining this property with the optimality of $(\phi_{i}^{\min,i}(\bar{\psi}_{i},\bar{\psi}_{i^*},\bar{\phi}_{i},\bar{\phi}_{i^*}), \phi_{i^*}^{\min,i}(\bar{\psi}_{i},\bar{\psi}_{i^*},\bar{\phi}_{i},\bar{\phi}_{i^*})) $, we have
\begin{align*}
	&a(\bar{\psi}_{i},\bar{\psi}_{i^*},\bar{\phi}_{i},\bar{\phi}_{i^*}) = -  w(\bar{\psi}_{i},\bar{\psi}_{i^*},\bar{\phi}_{i},\bar{\phi}_{i^*},\phi_{i}^{\min,i}(\bar{\psi}_{i},\bar{\psi}_{i^*},\bar{\phi}_{i},\bar{\phi}_{i^*}), \phi_{i^*}^{\min,i}(\bar{\psi}_{i},\bar{\psi}_{i^*},\bar{\phi}_{i},\bar{\phi}_{i^*})) \\
	\ge& -w(\bar{\psi}_{i},\bar{\psi}_{i^*},\bar{\phi}_{i},\bar{\phi}_{i^*},\phi_{i}^{\min,i}(\sigma_{i}^2,\sigma_{i^*}^2,\mu_{i},\mu_{i^*}), \phi_{i^*}^{\min,i}(\sigma_{i}^2,\sigma_{i^*}^2,\mu_{i},\mu_{i^*}))  \\
	\ge& -w(\sigma_{i}^2,\sigma_{i^*}^2,\mu_{i},\mu_{i^*},\phi_{i}^{\min,i}(\sigma_{i}^2,\sigma_{i^*}^2,\mu_{i},\mu_{i^*}), \phi_{i^*}^{\min,i}(\sigma_{i}^2,\sigma_{i^*}^2,\mu_{i},\mu_{i^*}))  - b_a \varepsilon \\
	=& a(\sigma_{i}^2,\sigma_{i^*}^2,\mu_{i},\mu_{i^*}) - b_a \varepsilon
\end{align*}
and
\begin{align*}
	a(\bar{\psi}_{i},\bar{\psi}_{i^*},\bar{\phi}_{i},\bar{\phi}_{i^*})
	\le&  -w(\sigma_{i}^2,\sigma_{i^*}^2,\mu_{i},\mu_{i^*},\phi_{i}^{\min,i}(\bar{\psi}_{i},\bar{\psi}_{i^*},\bar{\phi}_{i},\bar{\phi}_{i^*}), \phi_{i^*}^{\min,i}(\bar{\psi}_{i},\bar{\psi}_{i^*},\bar{\phi}_{i},\bar{\phi}_{i^*})) + b_{a} \varepsilon  \\
	\le&  a(\sigma_{i}^2,\sigma_{i^*}^2,\mu_{i},\mu_{i^*}) + b_{a} \varepsilon,
\end{align*}
which yields
$
\left| a(\bar{\psi}_{i},\bar{\psi}_{i^*},\bar{\phi}_{i},\bar{\phi}_{i^*}) - a(\sigma_{i}^2,\sigma_{i^*}^2,\mu_{i},\mu_{i^*}) \right| \le b_{a} \varepsilon.
$
Thus, \eqref{ineq:z_est} is proved. $\square$

\subsection{Proof of Lemma \ref{lem:mle}}

Define the log likelihood function for design $i$ as
\begin{align*}
	\iota(\phi_i,\psi_i) \triangleq& \log(L^{n} (\phi_{i},\psi_{i}))
	= - \frac{N_i}{2} \log (2 \pi) - \frac{N_i}{2} \log \psi_i - \frac{1}{ \psi_i } \sum_{l=1}^{N_i} \frac{(X_{il}-\phi_i)^2}{2},
\end{align*}
which can be simplified as
\begin{align}\label{eq:single_likelih}
	\iota(\phi_i,\psi_i) = - \frac{N_i}{2} \log (2 \pi) - \frac{N_i}{2} \log \psi_i - \frac{ N_i ( (\tilde{S}_{i}^n)^2 + (\bar{X}_{i}^n - \phi_{i})^2 ) }{ 2 \psi_i }.
\end{align}
It is well known that the MLE of $(\boldsymbol{\mu},\boldsymbol{\sigma}^2)$ is $(\bar{\boldsymbol{X}}^n,(\tilde{\boldsymbol{S}}^n)^2)$. Then $\boldsymbol{\phi}^{*,i^*} = \bar{\boldsymbol{X}}^n$ and $\boldsymbol{\psi}^{*,i^*} = (\tilde{\boldsymbol{S}}^n)^2$. We can simplify the following maximum likelihood estimation problem:
\begin{align*}
	&\max_{\boldsymbol{\phi}, \boldsymbol{\psi}} \sum_{j=1}^k \sum_{l=1}^{N_j} \log f (X_{jl} | \phi_{j},\psi_{j})
	= -\frac{n}{2} \left( \log (2\pi) + 1 \right) - \sum_{j=1}^k \frac{N_j}{2} \log (\tilde{S}_j^n)^2.
\end{align*}

Consider $\max_{(\boldsymbol{\phi}, \boldsymbol{\psi}) \in \Xi_{i}} \sum_{j=1}^k \sum_{l=1}^{N_j} \log f (X_{jl} | \phi_{j},\psi_{j})$. Taking the derivative of $\iota(\phi_j,\psi_j)$, we have
\begin{align*}
	&\partial \iota(\phi_j,\psi_j) / \partial \psi_j  = -  N_j /(2\psi_j) + N_j ( (\tilde{S}_{j}^n)^2 + (\bar{X}^n_{j} - \phi_{j})^2)/ (2 \psi_j^2)  ,  \\
	&\partial^2 \iota(\phi_j,\psi_j) / \partial \psi_j^2  =  N_j /(2\psi_j^2) - N_j ( (\tilde{S}_{j}^n)^2 + (\bar{X}^n_{j} - \phi_{j})^2)/ \psi_j^3  ,  \ \partial \iota(\phi_j,\psi_j) / \partial \phi_j  = - N_j(\phi_{j}-\bar{X}^n_{j})/ \psi_j  .
\end{align*}
Letting $\partial \iota(\phi_j,\psi_j) / \partial \psi_j = 0$ yields $\psi_j = \psi_j^*(\phi_j) \triangleq (\tilde{S}_{j}^n)^2 + (\bar{X}^n_{j} - \phi_{j})^2$.
Plugging $\psi_j^*(\phi_j)$ into $ \partial^2 \iota(\phi_j,\psi_j) / \partial \psi_j^2 $, we have
\begin{align*}
	\frac{\partial^2 \iota(\phi_j,\psi_j) }{ \partial \psi_j^2 } \Big|_{\psi_j = \psi_j^*(\phi_j)} (\psi_j^*(\phi_j))^3 = -\frac{N_j}{2} ((\tilde{S}_{j}^n)^2 + (\bar{X}^n_{j} - \phi_{j})^2)  < 0.
\end{align*}
Then for any $\phi_j$, the optimal solution for $\psi_j$ to optimize $\max_{\psi_j>0} \iota(\phi_j,\psi_j)$ is $\psi_j^*(\phi_j)$. Based on $\psi_{j}^*(\phi_{j})$, we analyze the maximum likelihood estimation problem under the constraint $\phi_{i} \ge \phi_{i^*}$ where $i \ne i^*$. The constrained problem can be simplified by plugging $\psi_{j}^*(\phi_{j})$ into the equation as
\begin{align*}
	&\max_{(\boldsymbol{\phi}, \boldsymbol{\psi}) \in \Xi_{i}} \sum_{j=1}^k \sum_{l=1}^{N_j} \log f (X_{jl} | \phi_{j},\psi_{j}) + \frac{n}{2} \left( \log (2\pi) + 1 \right)
	\\
	=&  \max_{ \phi_{i} \ge \phi_{i^*} } -\frac{N_{i}}{2}\log ((\tilde{S}_{i}^n)^2 + (\bar{X}^n_{i} - \phi_{i})^2) - \frac{N_{i^*}}{2}\log ((\tilde{S}_{i^*}^n)^2 + (\bar{X}^n_{i^*} - \phi_{i^*})^2)   - \sum_{j\ne i^*, i}  \left( \frac{N_j}{2}\log  (\tilde{S}_j^n)^2   \right) \nonumber,
\end{align*}
where the equality holds because the constraint is independent of design $j \ne i, i^*$ such that the optimal value for $\phi_j$ and $\psi_j$ are still $\bar{X}^n_j$ and $(\tilde{S}_j^n)^2$, as obtained in the unconstrained problem.

Since the number of samples for each design $j$ increases to infinity as total budget $n \to \infty$, by the strong law of large numbers, we have  that when $n$ is large enough, $| \bar{X}^n_j - \mu_j | \le \varepsilon$ and $| (\tilde{S}_j^n)^2 - \sigma_j^2 | \le \varepsilon$ for $\varepsilon < \Delta$. The optimal solution $(\phi_{i}^{*,i}, \phi_{i^*}^{*,i})$ to
\begin{align*}
	\max_{ \phi_{i} \ge \phi_{i^*} } -\frac{N_{i}}{2}\log ((\tilde{S}_{i}^n)^2 + (\bar{X}^n_{i} - \phi_{i})^2) - \frac{N_{i^*}}{2}\log ((\tilde{S}_{i^*}^n)^2 + (\bar{X}^n_{i^*} - \phi_{i^*})^2)
\end{align*}
is  $(\phi_{i}^{\min,i}((\tilde{S}_{i}^n)^2,(\tilde{S}_{i^*}^n)^2,\bar{X}^n_{i},\bar{X}^n_{i^*}), \phi_{i^*}^{\min,i}((\tilde{S}_{i}^n)^2,(\tilde{S}_{i^*}^n)^2,\bar{X}^n_{i},\bar{X}^n_{i^*}))$ for \eqref{ineq:eoc_num_z}. By \eqref{ineq:phiiistar} of Lemma \ref{lem:mle_eoc_z}, we have $\phi_{i}^{*,i} = \phi_{i^*}^{*,i}$. Thus,
\begin{align*}
	&\max_{(\boldsymbol{\phi}, \boldsymbol{\psi}) \in \Xi_{i}} \sum_{j=1}^k \sum_{l=1}^{N_j} \log f (X_{jl} | \phi_{j},\psi_{j}) + \frac{n}{2} \left( \log (2\pi) + 1 \right)
	\nonumber \\
	=&   \max_{ \phi_{i}  } -\frac{N_{i}}{2}\log ((\tilde{S}_{i}^n)^2 + (\bar{X}^n_{i} - \phi_{i})^2) - \frac{N_{i^*}}{2}\log ((\tilde{S}_{i^*}^n)^2 + (\bar{X}^n_{i^*} - \phi_{i})^2)   - \sum_{j\ne i^*, i}  \left( \frac{N_j}{2}\log  (\tilde{S}_j^n)^2   \right),
\end{align*}
which completes the proof. $\square$

\subsection{Proof of Lemma \ref{lem:xi_ai}}

Let $\Delta_1 \triangleq \min\{\Delta, \bar{\epsilon}\}$. Notice that by \eqref{eq:single_likelih},
\begin{align*}
	&\frac{1}{n}\sum_{j=1}^k \sum_{l=1}^{N_j} \log f (X_{jl} | \phi_{j},\psi_{j})
	= -\sum_{j=1}^k \frac{\alpha_j}{2}  \left( \log (2\pi) + \log \psi_j +  ((\tilde{S}_j^n)^2 + (\bar{X}^n_j-\phi_j)^2)/\psi_j \right).
\end{align*}
For notation simplicity, let $\boldsymbol{\iota}(\bar{\boldsymbol{\phi}}, \bar{\boldsymbol{\psi}}, \boldsymbol{\phi}, \boldsymbol{\psi})\triangleq -\sum_{j=1}^k \frac{\alpha_j}{2} (\log (2\pi\psi_j) + (\bar{\psi}_j + (\bar{\phi}_j-\phi_j)^2)/\psi_j) $. Notice that function $\boldsymbol{\iota}(\bar{\boldsymbol{\phi}}, \bar{\boldsymbol{\psi}}, \boldsymbol{\phi}, \boldsymbol{\psi})$ is a continuous function and has bounded gradient when $\bar{\boldsymbol{\phi}} \in [\mu_{\min}-\Delta_1,\mu_{\max}+\Delta_1]^k$, $\bar{\boldsymbol{\psi}} \in [\sigma_{\min}^2-\Delta_1,\sigma_{\max}^2+\Delta_1]^k$ and $
(\boldsymbol{\phi},\boldsymbol{\psi}) \in H_w$. This is because $\mu_{\min} - \bar{\epsilon} \le \phi_j \le \mu_{\max} + \bar{\epsilon} $ and $\psi_j \ge \sigma_{\min}^2 - \bar{\epsilon}$ for $
(\boldsymbol{\phi},\boldsymbol{\psi}) \in H_w$ and each design $j$ such that
\begin{align*}
	&|\partial \boldsymbol{\iota} / \partial \bar{\psi}_{j}| = \left| \alpha_j /(2\psi_j) \right| \le  1/(\sigma_{\min}^2 - \bar{\epsilon}) ,  \\
	&|\partial \boldsymbol{\iota} / \partial \bar{\phi}_{j}| = \left|  \alpha_j (\bar{\phi}_j-\phi_j)/\psi_j \right| \le (\mu_{\max}-\mu_{\min}+2\bar{\epsilon})/(\sigma_{\min}^2 - \bar{\epsilon}),  \\
	&|\partial \boldsymbol{\iota} / \partial \psi_{j}| = \left| \frac{\alpha_j}{2} ( \frac{1}{\psi_j} - \frac{\bar{\psi}_j + (\bar{\phi}_j-\phi_j)^2}{\psi_j^2} ) \right| \le \frac{1}{\sigma_{\min}^2 - \bar{\epsilon}} + \frac{( \sigma_{\max}^2 + \bar{\epsilon} + (\mu_{\max}-\mu_{\min}+2\bar{\epsilon})^2 )}{ (\sigma_{\min}^2 - \bar{\epsilon})^2},  \\
	&|\partial \boldsymbol{\iota} / \partial \phi_{j}| = \left| \alpha_j   (\phi_j-\bar{\phi}_j)/\psi_j \right| \le (\mu_{\max}-\mu_{\min}+2\bar{\epsilon}) / (\sigma_{\min}^2 - \bar{\epsilon}).
\end{align*}
Then, there exists $b_{\iota} > 0$ independent of $n$ and $\balpha$ such that $| \boldsymbol{\iota}(\bar{\boldsymbol{\phi}}, \bar{\boldsymbol{\psi}}, \boldsymbol{\phi}, \boldsymbol{\psi}) - \boldsymbol{\iota}(\bar{\boldsymbol{\phi}}', \bar{\boldsymbol{\psi}}', \boldsymbol{\phi}', \boldsymbol{\psi}')| \le \varepsilon$ if $\| (\bar{\boldsymbol{\phi}}, \bar{\boldsymbol{\psi}}, \boldsymbol{\phi}, \boldsymbol{\psi}) -  (\bar{\boldsymbol{\phi}}', \bar{\boldsymbol{\psi}}', \boldsymbol{\phi}', \boldsymbol{\psi}')\|_{\infty} \le b_{\iota} \varepsilon$ for $\bar{\boldsymbol{\phi}},\bar{\boldsymbol{\phi}}' \in [\mu_{\min}-\Delta_1,\mu_{\max}+\Delta_1]^k$, $\bar{\boldsymbol{\psi}},\bar{\boldsymbol{\psi}}' \in [\sigma_{\min}^2-\Delta_1,\sigma_{\max}^2+\Delta_1]^k$, $
(\boldsymbol{\phi},\boldsymbol{\psi}),(\boldsymbol{\phi}',\boldsymbol{\psi}') \in H_w	$.

For any sample path and any $\varepsilon < \Delta_1$, when $n$ is large enough,  with probability one, we have $| \bar{X}^n_j - \mu_j | \le \varepsilon$ and $|(\tilde{S}_j^n)^2 - \sigma_j^2| \le \varepsilon$ for each design $j$.
For a non-best design $i \ne i^*$, solution $( \boldsymbol{\phi}^{*,i}, \boldsymbol{\psi}^{*,i})$ is also optimal in $\max_{(\boldsymbol{\phi}, \boldsymbol{\psi}) \in \Xi_{i}} \boldsymbol{\iota}(\bar{\boldsymbol{X}}^n, (\tilde{\boldsymbol{S}}^n)^2, \boldsymbol{\phi}, \boldsymbol{\psi})$ because $\max_{(\boldsymbol{\phi}, \boldsymbol{\psi}) \in \Xi_{i}} \sum_{j=1}^k \sum_{l=1}^{N_j} \log f (X_{jl} | \phi_{j},\psi_{j})  = n \max_{(\boldsymbol{\phi}, \boldsymbol{\psi}) \in \Xi_{i}} \boldsymbol{\iota}(\bar{\boldsymbol{X}}^n, (\tilde{\boldsymbol{S}}^n)^2, \boldsymbol{\phi}, \boldsymbol{\psi})$.
By Lemma \ref{lem:mle},
\begin{itemize}
	\item $\mu_{j} - \varepsilon \le \phi_{j}^{*,i} = \bar{X}^n_j \le \mu_{j} + \varepsilon$ if $j \ne i,i^*$;
	
	\item $\phi_{i}^{*,i} = \phi_{i^{*}}^{*,i} \in [ \mu_{i} - \varepsilon, \mu_{i^*} + \varepsilon]$ for designs $i,i^*$;
	
	\item $\sigma_{j}^2 - \varepsilon \le \psi_{j}^{*,i} = (\tilde{S}_j^n)^2 \le \sigma_{j}^2 + \varepsilon$ if $j \ne i,i^*$;
	
	\item $\sigma_{i}^2 - \varepsilon \le \psi_{i}^{*,i} = (\tilde{S}_{i}^n)^2+(\bar{X}^n_{i}-\phi_{i}^{*,i})^2 \le \sigma_{i}^2 + \varepsilon + (\mu_{i^*} - \mu_{i} + 2\varepsilon)^2$ for design $i$;
	
	\item $\sigma_{i^*}^2 - \varepsilon \le \psi_{i^*}^{*,i} = (\tilde{S}_{i^*}^n)^2+(\bar{X}^n_{i^*}-\phi_{i}^{*,i})^2 \le \sigma_{i^*}^2 + \varepsilon + (\mu_{i^*} - \mu_{i} + 2\varepsilon)^2$ for design $i^*$.
\end{itemize}
If $2\varepsilon \le \bar{\epsilon}$ which implies $\varepsilon \le \bar{\epsilon}$, then
\begin{align*}
	(\boldsymbol{\phi}^{*,i},\boldsymbol{\psi}^{*,i}) \in [ \mu_{\min} - \varepsilon,\mu_{\max} + \varepsilon]^k \times [  \sigma_{\min}^2 - \varepsilon , \sigma_{\max}^2 + \varepsilon + ( \mu_{\max} - \mu_{\min} + 2 \varepsilon )^2 ]^k \subset H_w.
\end{align*}
Moreover, if $2(b_{\iota}+1)\varepsilon \le \bar{\epsilon}$, then any $( \boldsymbol{\phi}, \boldsymbol{\psi}) \in \Xi_{i}$ with
$\| ( \boldsymbol{\phi}, \boldsymbol{\psi}) -  ( \boldsymbol{\phi}^{*,i}, \boldsymbol{\psi}^{*,i})\|_{\infty} \le b_{\iota} \varepsilon$
satisfies
\begin{align*}
	(\boldsymbol{\phi},\boldsymbol{\psi}) \in& [\mu_{\min} - (b_{\iota}+1)\varepsilon,\mu_{\max} + (b_{\iota}+1) \varepsilon ]^k  \\
	&\times [ \sigma_{\min}^2 - (b_{\iota}+1) \varepsilon , \sigma_{\max}^2 + (b_{\iota}+1) \varepsilon + ( \mu_{\max} - \mu_{\min} + 2 \varepsilon )^2 ]^k \subset H_w.
\end{align*}
Then, when $\bar{\boldsymbol{\phi}} = \bar{\boldsymbol{X}}^n$ and $\bar{\boldsymbol{\psi}} = (\tilde{\boldsymbol{S}}^n)^2$ and if $( \boldsymbol{\phi}, \boldsymbol{\psi}) \in \Xi_{i}$ and
$\| ( \boldsymbol{\phi}, \boldsymbol{\psi}) -  ( \boldsymbol{\phi}^{*,i}, \boldsymbol{\psi}^{*,i})\|_{\infty} \le b_{\iota} \varepsilon$,
we have $ | \boldsymbol{\iota}(\bar{\boldsymbol{X}}^n, (\tilde{\boldsymbol{S}}^n)^2, \boldsymbol{\phi}^{*,i}, \boldsymbol{\psi}^{*,i}) - \boldsymbol{\iota}(\bar{\boldsymbol{X}}^n, (\tilde{\boldsymbol{S}}^n)^2, \boldsymbol{\phi}, \boldsymbol{\psi}) | \le \varepsilon $.
Moreover, the volume of the subset $H_{i}$ is independent of the budget $n$. Specifically,
\begin{align*}
	\int \int_{ ( \boldsymbol{\phi}, \boldsymbol{\psi}) \in H_{i} } d \boldsymbol{\phi} d \boldsymbol{\psi} =& \int \int_{ ( \boldsymbol{\phi}, \boldsymbol{\psi}): \| ( \boldsymbol{\phi}, \boldsymbol{\psi}) -  ( \boldsymbol{\phi}^{*,i}, \boldsymbol{\psi}^{*,i})\|_{\infty} \le b_{\iota} \varepsilon } \1 ( \phi_i \ge \phi_{i^*} ) d \boldsymbol{\phi} d \boldsymbol{\psi}  \\
	=&(2b_{\iota}\varepsilon)^{2k-2}\int_{ \phi_{i^*}^{*,i}-b_{\iota}\varepsilon }^{ \phi_{i^*}^{*,i}+b_{\iota}\varepsilon } \int_{ \phi_{i}^{*,i}-b_{\iota}\varepsilon }^{ \phi_{i}^{*,i}+b_{\iota}\varepsilon } \1 ( \phi_i \ge \phi_{i^*} ) d \phi_i d \phi_{i^*} \\
	=& (2b_{\iota}\varepsilon)^{2k-2} \int_{ \phi_{i^*}^{*,i}-b_{\iota}\varepsilon }^{ \phi_{i^*}^{*,i}+b_{\iota}\varepsilon } \int_{ \phi_{i^*}  }^{ \phi_{i^*}^{*,i}+b_{\iota}\varepsilon }  d \phi_i d \phi_{i^*}  = \frac{1}{2} (2b_{\iota}\varepsilon )^{2k},
\end{align*}
where the third equality holds because $\phi_{i}^{*,i} = \phi_{i^*}^{*,i}$ by Lemma \ref{lem:mle}.

Similarly, for the best design $i^*$, the optimal solution of $\max_{\boldsymbol{\phi}, \boldsymbol{\psi}} \boldsymbol{\iota}(\bar{\boldsymbol{X}}^n, (\tilde{\boldsymbol{S}}^n)^2, \boldsymbol{\phi}, \boldsymbol{\psi})$
is the MLE $( \boldsymbol{\phi}^{*,i^*}, \boldsymbol{\psi}^{*,i^*}) = (\bar{\boldsymbol{X}}^n, (\tilde{\boldsymbol{S}}^n)^2)$. Then, if $(b_{\iota}+1)\varepsilon \le \bar{\epsilon}$ and $( \boldsymbol{\phi}, \boldsymbol{\psi}) \in H_{i^*}$ where
$
\| ( \boldsymbol{\phi}, \boldsymbol{\psi}) -  (\bar{\boldsymbol{X}}^n, (\tilde{\boldsymbol{S}}^n)^2)\|_{\infty} \le b_{\iota} \varepsilon,
$
we have $( \boldsymbol{\phi}, \boldsymbol{\psi}) \in H_w$ and
$| \boldsymbol{\iota}(\bar{\boldsymbol{X}}^n, (\tilde{\boldsymbol{S}}^n)^2, \bar{\boldsymbol{X}}^n, (\tilde{\boldsymbol{S}}^n)^2) - \boldsymbol{\iota}(\bar{\boldsymbol{X}}^n, (\tilde{\boldsymbol{S}}^n)^2, \boldsymbol{\phi}, \boldsymbol{\psi}) | \le \varepsilon$. The volume of $H_{i^*}$ is $(2b_{\iota}\varepsilon)^{2k}$, which is independent of $n$. $\square$

\section{Proofs for Section \ref{sec:proofalgs}}\label{sec:propW_proof}

In the following, we prove Lemmas \ref{lem:phihatmon2}-\ref{lem:optnonub} and Propositions \ref{prop:optnonconv}-\ref{prop:nonconv}.

\subsection{Proof of Lemma \ref{lem:phihatmon2}}

Let $b_{\alpha \phi 2} \triangleq \max \{ 4 b_{dU} \varepsilon_1^{1/2}, 4 b_{\Ical \Hcal}, b_{\alpha \phi 1} \}+1$ where $b_{\alpha \phi 1} \triangleq  (2(b_{\alpha U} + 1)  b_{lU})  /  b_{\Ical} $. Suppose $\hat{\alpha}_{i}^{m'}/\hat{\alpha}_{i^*}^{m'} - \hat{\alpha}_{i}^m/\hat{\alpha}_{i^*}^m \le - b_{\alpha \phi 2} \varepsilon^{1/2}$. For $m \ge M_3(\varepsilon)$ and $\phi_i \in [ \hat{\mu}_{i}^m + b_{\phi}, \hat{\phi}_{i}^{m} - \varepsilon^{1/2} ]$,
\begin{align}
	&(\hat{\alpha}_{i}^{m'}/\hat{\alpha}_{i^*}^{m'}) \log \left(1 + (\hat{\mu}_{i}^{m'}-\phi_i)^2/(\hat{\sigma}_{i}^{m'})^2 \right) +  \log \left(1 + (\hat{\mu}_{i^*}^{m'}-\phi_{i})^2/(\hat{\sigma}_{i^*}^{m'})^2 \right)  \nonumber  \\
	& - (\hat{\alpha}_{i}^{m'}/\hat{\alpha}_{i^*}^{m'}) \log \left(1 + (\hat{\mu}_{i}^{m'}-\hat{\phi}_{i}^{m})^2/(\hat{\sigma}_{i}^{m'})^2 \right) -  \log \left(1 + (\hat{\mu}_{i^*}^{m'}-\hat{\phi}_{i}^{m})^2/(\hat{\sigma}_{i^*}^{m'})^2 \right)  \nonumber \\
	\ge& - (2 b_{\alpha U} + 2)  b_{lU} \varepsilon +  (\hat{\alpha}_{i}^m/\hat{\alpha}_{i^*}^m) \log \left(1 + (\hat{\mu}_{i}^m-\phi_i)^2/(\hat{\sigma}_{i}^m)^2 \right) +  \log \left(1 + (\hat{\mu}_{i^*}^m-\phi_{i})^2/(\hat{\sigma}_{i^*}^m)^2 \right)  \nonumber \\
	& - (\hat{\alpha}_{i}^m/\hat{\alpha}_{i^*}^m) \log \left(1 + (\hat{\mu}_{i}^m-\hat{\phi}_{i}^{m})^2/(\hat{\sigma}_{i}^m)^2 \right) -  \log \left(1 + (\hat{\mu}_{i^*}^m-\hat{\phi}_{i}^{m})^2/(\hat{\sigma}_{i^*}^m)^2 \right) \nonumber \\
	& + (\hat{\alpha}_{i}^{m'}/\hat{\alpha}_{i^*}^{m'} - \hat{\alpha}_{i}^m/\hat{\alpha}_{i^*}^m) \left( \log \left(1 + (\hat{\mu}_{i}^m-\phi_i)^2/(\hat{\sigma}_{i}^m)^2 \right) - \log \left(1 + (\hat{\mu}_{i}^m-\hat{\phi}_{i}^{m})^2/(\hat{\sigma}_{i}^m)^2 \right) \right) \label{ineq:phi_lb1}
	\\
	\ge& (\hat{\alpha}_{i}^{m'}/\hat{\alpha}_{i^*}^{m'} - \hat{\alpha}_{i}^m/\hat{\alpha}_{i^*}^m) \big( \log \big(1 + (\hat{\mu}_{i}^m-\phi_i)^2/(\hat{\sigma}_{i}^m)^2 \big) - \log \big(1 + (\hat{\mu}_{i}^m-\hat{\phi}_{i}^{m})^2/(\hat{\sigma}_{i}^m)^2 \big) \big) \nonumber \\
	& - (2 b_{\alpha U} + 2)  b_{lU} \varepsilon  \label{ineq:phi_lb2} \\
	\ge& b_{\alpha \phi 2} \varepsilon^{1/2} b_{\Ical} \varepsilon^{1/2} - (2 b_{\alpha U} + 2)  b_{lU} \varepsilon  \label{ineq:phi_lb3}\\
	>& 0, \label{ineq:phi_lb4}
\end{align}
where \eqref{ineq:phi_lb1} holds by \eqref{ineq:log_ub}, \eqref{ineq:phi_lb2} holds because $\hat{\phi}_{i}^{m}$ is the  solution that minimizes the function about $\phi_i$,
\eqref{ineq:phi_lb3} holds by $\hat{\alpha}_{i}^{m'}/\hat{\alpha}_{i^*}^{m'} - \hat{\alpha}_{i}^m/\hat{\alpha}_{i^*}^m < - b_{\alpha \phi 2} \varepsilon^{1/2}$ and the lower bound of the derivative in \eqref{ineq:derv_lb}, and \eqref{ineq:phi_lb4} holds because $b_{\alpha \phi 2} > b_{\alpha \phi 1} = (2(b_{\alpha U} + 1)  b_{lU})  /  b_{\Ical}$. Then, any solution $\phi_i \in [ \hat{\mu}_{i}^m + b_{\phi}, \hat{\phi}_{i}^{m} - \varepsilon^{1/2} ]$ is not the optimal solution $\hat{\phi}_{i}^{m'}$. If $\hat{\alpha}_{i}^{m'}/\hat{\alpha}_{i^*}^{m'} - \hat{\alpha}_{i}^m/\hat{\alpha}_{i^*}^m \le - b_{\alpha \phi 2} \varepsilon^{1/2}$,
\begin{align}\label{ineq:phihatmon}
	\hat{\phi}_{i}^{m'} - \hat{\phi}_{i}^{m} >  -\varepsilon^{1/2}.
\end{align}
Furthermore,
\begin{align}
	- \frac{\Ical ( \hat{\phi}_{i}^{m'}, \mu_{i^*}, \sigma_{i^*}^2 )}{\Ical( \hat{\phi}_{i}^{m'}, \mu_{i}, \sigma_{i}^2 )} + \frac{\Ical ( \hat{\phi}_{i}^{m}, \mu_{i^*}, \sigma_{i^*}^2 )}{\Ical( \hat{\phi}_{i}^{m}, \mu_{i}, \sigma_{i}^2 )}
	&\ge - \frac{\Ical ( \hat{\phi}_{i}^{m'}, \hat{\mu}_{i^*}^{m'}, (\hat{\sigma}_{i^*}^{m'})^2 )}{\Ical( \hat{\phi}_{i}^{m'}, \hat{\mu}_{i}^{m'}, (\hat{\sigma}_{i}^{m'})^2 )} + \frac{\Ical ( \hat{\phi}_{i}^{m}, \hat{\mu}_{i^*}^m, (\hat{\sigma}_{i^*}^m)^2 )}{\Ical( \hat{\phi}_{i}^{m}, \hat{\mu}_{i}^m, (\hat{\sigma}_{i}^m)^2 )} - 2b_{dU} \varepsilon  \label{ineq:hess_lb0} \\
	&= - \hat{\alpha}_{i}^{m'} / \hat{\alpha}_{i^*}^{m'}  +  \hat{\alpha}_{i}^m / \hat{\alpha}_{i^*}^m  - 2b_{dU} \varepsilon
	\ge  b_{\alpha \phi 2} \varepsilon^{1/2} - 2b_{dU} \varepsilon  \nonumber \\
	&\ge 4 b_{\Ical \Hcal} \varepsilon^{1/2},  \label{ineq:hess_lb}
\end{align}
where \eqref{ineq:hess_lb0} holds by \eqref{ineq:diff_ub}, \eqref{ineq:hess_lb} holds because $\frac{b_{\alpha \phi 2}}{2} \ge 2b_{dU} \varepsilon^{1/2}$ and $b_{\alpha \phi 2} \ge 8 b_{\Ical \Hcal}$.
By \eqref{ineq:Ider2_up}, we have
\begin{align*}
	\left| \Ical ( \phi_1, \mu_{i^*}, \sigma_{i^*}^2 )/\Ical( \phi_1, \mu_{i}, \sigma_{i}^2 ) - \Ical ( \phi_2, \mu_{i^*}, \sigma_{i^*}^2 )/\Ical( \phi_2, \mu_{i}, \sigma_{i}^2 ) \right| \le b_{\Ical \Hcal} \varepsilon
\end{align*}
for $ |\phi_1 - \phi_2| \le \varepsilon $. 	This together with \eqref{ineq:hess_lb} implies that
$|\hat{\phi}_{i}^{m'} - \hat{\phi}_{i}^{m}| \ge 2\varepsilon^{1/2}$.
Meanwhile, since $\hat{\phi}_{i}^{m'} > \hat{\phi}_{i}^{m} - \varepsilon^{1/2}$ by \eqref{ineq:phihatmon}, we have $\hat{\phi}_{i}^{m'} - \hat{\phi}_{i}^{m} \ge  2\varepsilon^{1/2}$.

Suppose $\hat{\alpha}_{i}^{m'}/\hat{\alpha}_{i^*}^{m'} - \hat{\alpha}_{i}^m/\hat{\alpha}_{i^*}^m \ge b_{\alpha \phi 2} \varepsilon^{1/2}$. The proof is similar and is thus omitted. $\square$

\subsection{Proof of Lemma \ref{lem:upart}}

Let $ \varepsilon_3 \triangleq \min \{ \varepsilon_2,  b_{log2} /(8b_{lU}), ( b_{\Ical} b_{log2} / (4 b_{lU} b_{log1}))^2, (b_{log1} / 4 b_{\Ical})^2    \}$ and $b_{up} \triangleq  b_{\Ical} / b_{log1}  $. If $\hat{\alpha}_{i}^{m'}/ \hat{\alpha}_{i^*}^{m'}$ $  \le \hat{\alpha}_{i}^m/ \hat{\alpha}_{i^*}^m  -  b_{\alpha \phi 2} \varepsilon^{\frac{1}{2}}
$, we have by Lemma \ref{lem:phihatmon2} that
$\hat{\phi}_{i}^{m'} - \hat{\phi}_{i}^{m} \ge  2\varepsilon^{\frac{1}{2}}$, which leads to
\begin{align}
	&\log \left(1 + (\hat{\mu}_{i^*}^m-\hat{\phi}_{i}^{m'})^2/(\hat{\sigma}_{i^*}^m)^2 \right)
	\le \log \left(1 + (\hat{\mu}_{i^*}^m-\hat{\phi}_{i}^{m})^2/(\hat{\sigma}_{i^*}^m)^2 \right) - 2 b_{\Ical} \varepsilon^{\frac{1}{2}} \label{ineq:upart_0} \\
	\le& \log \left(1 + (\hat{\mu}_{i^*}^m-\hat{\phi}_{i}^{m})^2/(\hat{\sigma}_{i^*}^m)^2 \right) \left( 1 - 2 b_{\Ical}  \varepsilon^{\frac{1}{2}} / b_{log1} \right), \label{ineq:upart_1}
\end{align}
where \eqref{ineq:upart_0} holds by the mean value theorem and \eqref{ineq:derv_lb}, and \eqref{ineq:upart_1} holds by the definition of $b_{log1}$ in \eqref{ineq:log_ub1}. Meanwhile, combining the definition of $b_{log2}$ in \eqref{ineq:log_ub2} with \eqref{ineq:log_ub}, we have
\begin{align}\label{ineq:upart_1s}
	&\left| \frac{\log \left(1 + (\hat{\mu}_{i^*}^{m'}-\hat{\phi}_{i}^{m'})^2/(\hat{\sigma}_{i^*}^{m'})^2 \right)}{ \log \left(1 + (\hat{\mu}_{i^*}^m-\hat{\phi}_{i}^{m'})^2/(\hat{\sigma}_{i^*}^m)^2 \right) } -1 \right| \le \left| \frac{ b_{lU}  \varepsilon }{ \log \left(1 + (\hat{\mu}_{i^*}^m-\hat{\phi}_{i}^{m'})^2/(\hat{\sigma}_{i^*}^m)^2 \right) } \right|
	\le  b_{lU}  \varepsilon / b_{log2}.
\end{align}
Moreover, since $\hat{\phi}_{i}^{m'} - \hat{\phi}_{i}^{m} \ge  2\varepsilon^{\frac{1}{2}}$ and $\hat{\phi}_{i}^{m} > \hat{\mu}_{i}^m$, we have
\begin{align}\label{ineq:upart_2}
	\log \left(1 + (\hat{\mu}_{i}^m-\hat{\phi}_{i}^{m'})^2/(\hat{\sigma}_{i}^m)^2 \right) \ge \log \left(1 + (\hat{\mu}_{i}^m-\hat{\phi}_{i}^{m})^2/(\hat{\sigma}_{i}^m)^2 \right).
\end{align}
Finally, we have
\begin{align}
	&\log \left(1 + (\hat{\mu}_{i^*}^{m'}-\hat{\phi}_{i}^{m'})^2/(\hat{\sigma}_{i^*}^{m'})^2 \right) / \log \left(1 + (\hat{\mu}_{i}^{m'}-\hat{\phi}_{i}^{m'})^2/(\hat{\sigma}_{i}^{m'})^2 \right)  \nonumber \\
	\le& \frac{1+b_{lU} \varepsilon/ b_{log2} }{ 1 - b_{lU} \varepsilon/ b_{log2}   } \log \left(1 + (\hat{\mu}_{i^*}^m-\hat{\phi}_{i}^{m'})^2/(\hat{\sigma}_{i^*}^m)^2 \right) / \log \left(1 + (\hat{\mu}_{i}^m-\hat{\phi}_{i}^{m'})^2/(\hat{\sigma}_{i}^m)^2 \right)  \label{ineq:upart_3} \\
	\le& (1 + 4 b_{lU} \varepsilon/ b_{log2} ) \log \left(1 + (\hat{\mu}_{i^*}^m-\hat{\phi}_{i}^{m'})^2/(\hat{\sigma}_{i^*}^m)^2 \right) / \log \left(1 + (\hat{\mu}_{i}^m-\hat{\phi}_{i}^{m'})^2/(\hat{\sigma}_{i}^m)^2 \right)  \label{ineq:upart_4} \\
	\le& \left(1 + \frac{4 b_{lU} \varepsilon}{ b_{log2} }\right) \left(1 - \frac{ 2 b_{\Ical} }{ b_{log1} } \varepsilon^{ \frac{1}{2} }\right) \log \left(1 + \frac{(\hat{\mu}_{i^*}^m-\hat{\phi}_{i}^{m})^2}{(\hat{\sigma}_{i^*}^m)^2} \right)/ \log \left(1 + \frac{(\hat{\mu}_{i}^m-\hat{\phi}_{i}^{m})^2}{(\hat{\sigma}_{i}^m)^2} \right)  \label{ineq:upart_5} \\
	\le& \left(1 -   b_{\Ical}    \varepsilon^{ \frac{1}{2} } / b_{log1}\right) \log \left(1 + (\hat{\mu}_{i^*}^m-\hat{\phi}_{i}^{m})^2/(\hat{\sigma}_{i^*}^m)^2 \right) / \log \left(1 + (\hat{\mu}_{i}^m-\hat{\phi}_{i}^{m})^2/(\hat{\sigma}_{i}^m)^2 \right), \label{ineq:upart_6}
\end{align}
where \eqref{ineq:upart_3} holds by \eqref{ineq:upart_1s}, \eqref{ineq:upart_4} holds because   $\varepsilon \le \varepsilon_3 < b_{log2} /(2b_{lU})$ such that
\begin{align*}
	\big(1 + \frac{4 b_{lU} \varepsilon}{ b_{log2} }\big) \big(1-\frac{b_{lU}\varepsilon}{b_{log2}}\big) = 1+\frac{3 b_{lU} \varepsilon}{ b_{log2} }-\frac{4b_{lU}^2\varepsilon^2}{b_{log2}^2} \ge 1+\frac{3 b_{lU} \varepsilon}{ b_{log2} }-\frac{2 b_{lU} \varepsilon}{ b_{log2} } = 1+\frac{b_{lU} \varepsilon}{ b_{log2} },
\end{align*}
\eqref{ineq:upart_5} holds by \eqref{ineq:upart_1} and \eqref{ineq:upart_2}, and \eqref{ineq:upart_6} holds because $\varepsilon \le ( b_{\Ical} b_{log2} / (4 b_{lU} b_{log1}) )^2$ such that $4 b_{lU} \varepsilon/ b_{log2} \le  b_{\Ical} / (b_{log1}) \varepsilon^{\frac{1}{2}} $. If $\hat{\alpha}_{i}^{m'}/ \hat{\alpha}_{i^*}^{m'}  \ge \hat{\alpha}_{i}^m/ \hat{\alpha}_{i^*}^m + b_{\alpha \phi 2} \varepsilon^{\frac{1}{2}}$, the proof is similar and is thus omitted. $\square$

\subsection{Proof of Lemma \ref{lem:vdiff}}

For notational simplicity, let
$\Dcal_1 \triangleq  \hat{\Wcal}_{i}^{m'} - \hat{\Wcal}_{i}^m$, $\Dcal_2 \triangleq \hat{\Wcal}_{i}^m - \hat{\Wcal}_{j}^m$, $\Dcal_3 \triangleq \hat{\Wcal}_{j}^m - \hat{\Wcal}_{j}^{m'}$,
\begin{align*}
	&\hat{g}_i^{m'} (\hat{\alpha}_{i}^{m'}/\hat{\alpha}_{i^*}^{m'},\hat{\phi}_{i}^{m}) \triangleq \log \left(1 + (\hat{\mu}_{i}^{m'}-\hat{\phi}_{i}^{m})^2/(\hat{\sigma}_{i}^{m'})^2 \right) \hat{\alpha}_{i}^{m'}/\hat{\alpha}_{i^*}^{m'} + \log \left(1 + (\hat{\mu}_{i^*}^{m'}-\hat{\phi}_{i}^{m})^2/(\hat{\sigma}_{i^*}^{m'})^2 \right), \\
	&\hat{g}_i^{m} (\hat{\alpha}_{i}^{m'}/\hat{\alpha}_{i^*}^{m'},\hat{\phi}_{i}^{m}) \triangleq \log \left(1 + (\hat{\mu}_{i}^m-\hat{\phi}_{i}^{m})^2/(\hat{\sigma}_{i}^m)^2 \right) \hat{\alpha}_{i}^{m'}/\hat{\alpha}_{i^*}^{m'} + \log \left(1 + (\hat{\mu}_{i^*}^m-\hat{\phi}_{i}^{m})^2/(\hat{\sigma}_{i^*}^m)^2 \right).
\end{align*}
Since
\begin{align}
	& \hat{\Wcal}_{i}^{m'}  -  \hat{g}_i^{m'} (\hat{\alpha}_{i}^{m'}/\hat{\alpha}_{i^*}^{m'},\hat{\phi}_{i}^{m}) \le 0, \label{ineq:w_ub1}\\
	& \hat{g}_i^{m'} (\hat{\alpha}_{i}^{m'}/\hat{\alpha}_{i^*}^{m'},\hat{\phi}_{i}^{m}) -  \hat{g}_i^{m} (\hat{\alpha}_{i}^{m'}/\hat{\alpha}_{i^*}^{m'},\hat{\phi}_{i}^{m}) \le  (b_{\alpha U} + 1) b_{lU} \varepsilon,  \label{ineq:w_ub2} \\
	&\hat{g}_i^{m} (\hat{\alpha}_{i}^{m'}/\hat{\alpha}_{i^*}^{m'},\hat{\phi}_{i}^{m})
	- \hat{\Wcal}_{i}^{m}  = (\hat{\alpha}_{i}^{m'}/\hat{\alpha}_{i^*}^{m'} - \hat{\alpha}_{i}^m/\hat{\alpha}_{i^*}^m) \log (1 + (\hat{\mu}_{i}^m-\hat{\phi}_{i}^{m})^2/(\hat{\sigma}_{i}^m)^2 ), \nonumber
\end{align}
where \eqref{ineq:w_ub1} holds because $\hat{\Wcal}_{i}^{m'}$ is the minimal value, and \eqref{ineq:w_ub2} holds by \eqref{ineq:log_ub}. We have
\begin{align*}
	\Dcal_1
	\le& (b_{\alpha U} + 1) b_{lU} \varepsilon + (\hat{\alpha}_{i}^{m'}/\hat{\alpha}_{i^*}^{m'} - \hat{\alpha}_{i}^m/\hat{\alpha}_{i^*}^m) \log (1 + (\hat{\mu}_{i}^m-\hat{\phi}_{i}^{m})^2/(\hat{\sigma}_{i}^m)^2 ).
\end{align*}
Similarly, $\Dcal_3 \le (b_{\alpha U} + 1) b_{lU} \varepsilon + ( \hat{\alpha}_{j}^m/\hat{\alpha}_{i^*}^m - \hat{\alpha}_{j}^{m'}/\hat{\alpha}_{i^*}^{m'} ) \log \big(1 + (\hat{\mu}_{j}^{m'}-\hat{\phi}_{j}^{m'})^2/(\hat{\sigma}_{j}^{m'})^2 \big)$.
Then
\begin{align*}
	&\hat{\Wcal}_{i}^{m'} - \hat{\Wcal}_{j}^{m'}
	= \Dcal_1 + \Dcal_2 + \Dcal_3 \nonumber \\
	\le& 2 (b_{\alpha U} + 1) b_{lU} \varepsilon + \hat{\Wcal}_{i}^m - \hat{\Wcal}_{j}^m + (\hat{\alpha}_{i}^{m'}/\hat{\alpha}_{i^*}^{m'} - \hat{\alpha}_{i}^m/\hat{\alpha}_{i^*}^m) \log (1 + (\hat{\mu}_{i}^m-\hat{\phi}_{i}^{m})^2/(\hat{\sigma}_{i}^m)^2 )  \nonumber \\
	& +  ( \hat{\alpha}_{j}^m/\hat{\alpha}_{i^*}^m - \hat{\alpha}_{j}^{m'}/\hat{\alpha}_{i^*}^{m'} ) \log (1 + (\hat{\mu}_{j}^{m'}-\hat{\phi}_{j}^{m'})^2/(\hat{\sigma}_{j}^{m'})^2 ) \\
	=& 2 (b_{\alpha U} + 1) b_{lU} \varepsilon + \hat{\Wcal}_{i}^m - \hat{\Wcal}_{j}^m + (\hat{\alpha}_{i}^{m'}/\hat{\alpha}_{i^*}^{m'} - \hat{\alpha}_{i}^m/\hat{\alpha}_{i^*}^m) \hat{\Ucal}^{m}_{i}
	+ ( \hat{\alpha}_{j}^m/\hat{\alpha}_{i^*}^m - \hat{\alpha}_{j}^{m'}/\hat{\alpha}_{i^*}^{m'} ) \hat{\Ucal}^{m'}_{j}.
\end{align*}
This completes the proof. $\square$

\subsection{Proof of Lemma \ref{lem:interv_opt}}

Let $ \varepsilon_4 \triangleq \min\{\varepsilon_3, b_{\alpha \phi 2}^2 \}$ and $M_4(\varepsilon) \triangleq \max\{ M_3(\varepsilon), \inf_{M} \{\forall m \ge M, \forall i=1,\dots,k:  (N_i^m - 1)^{-1}  \le  \varepsilon/ b_{\alpha U}  \} \}$. We show this lemma by contradiction. Suppose $\hat{\alpha}_{i_1}^{m'}/ \hat{\alpha}_{i^*}^{m'}  \le (\hat{\alpha}_{i_1}^m/ \hat{\alpha}_{i^*}^m)  ( 1- (2 b_{\alpha \phi 2}/ b_{\alpha L})  \varepsilon^{\frac{1}{2}} )$, $\forall i_1 \ne i^*$.
Since $\hat{\alpha}_{i_1}^m/ \hat{\alpha}_{i^*}^m \ge b_{\alpha L}$, we have $\hat{\alpha}_{i_1}^{m'}/ \hat{\alpha}_{i^*}^{m'} \le \hat{\alpha}_{i_1}^m / \hat{\alpha}_{i^*}^m  - 2 b_{\alpha \phi 2} \varepsilon^{\frac{1}{2}}$.

Let $m+1 < m^*+1 < m'+1$ denote the last iteration at or before iteration $m'$ where $i^*$ is sampled. Notice that $N_{i^*}^{m^*} = N_{i^*}^{m'} - 1$ and $(N_{i^*}^{m'} - 1)^{-1} \le  \varepsilon /b_{\alpha U}$ for $m' \ge M_4(\varepsilon)$.  Then
\begin{align*}
	\hat{\alpha}_{i_1}^{m^*} / \hat{\alpha}_{i^*}^{m^*} = N_{i_1}^{m^*}/ N_{i^*}^{m^*} \le N_{i_1}^{m'}/ (N_{i^*}^{m'} - 1)  \le N_{i_1}^{m'}/ N_{i^*}^{m'}  + \varepsilon = \hat{\alpha}_{i_1}^{m'} / \hat{\alpha}_{i^*}^{m'} + \varepsilon.
\end{align*}
Since $\varepsilon \le b_{\alpha \phi 2}^2$ and $\hat{\alpha}_{i_1}^{m'}/ \hat{\alpha}_{i^*}^{m'} \le \hat{\alpha}_{i_1}^m / \hat{\alpha}_{i^*}^m  - 2 b_{\alpha \phi 2} \varepsilon^{\frac{1}{2}}$, the above inequality leads to $\hat{\alpha}_{i_1}^{m^*} / \hat{\alpha}_{i^*}^{m^*}  \le \hat{\alpha}_{i_1}^m / \hat{\alpha}_{i^*}^m  -  b_{\alpha \phi 2} \varepsilon^{\frac{1}{2}}$.
By Lemma \ref{lem:upart}, we have $\hat{\Ucal}^{*,m^*}_{i_1}/\hat{\Ucal}^{m^*}_{i_1} \le (1 - b_{up} \varepsilon^{ \frac{1}{2} }) \hat{\Ucal}^{*,m}_{i_1}/\hat{\Ucal}^{m}_{i_1}$.
Then
\begin{align*}
	\sum_{i_1 \ne i^*} \hat{\Ucal}^{*,m^*}_{i_1}/\hat{\Ucal}^{m^*}_{i_1}
	\le (1 - b_{up} \varepsilon^{\frac{1}{2}}) \sum_{i_1 \ne i^*} \hat{\Ucal}^{*,m}_{i_1}/\hat{\Ucal}^{m}_{i_1}
	\le 1.
\end{align*}
This implies that $i^*$ is not sampled in iteration $m^*+1$, which contradicts the definition of $m^*+1$. $\square$

\subsection{Proof of Lemma \ref{lem:interv}}

Let $\varepsilon_5 \triangleq \min \{ \varepsilon_4, b_{\Wcal U 1}^{-2}, (b_{\Wcal U 1} b_{\alpha L}^2/ (16 b_{\Wcal U 2} b_{\alpha U}) )^2 \}$ where
\begin{align*}
	&b_{\Wcal U 1} \triangleq \max\{ 8 b_{\alpha \phi 2}/ b_{\alpha L} , (2 (b_{\alpha U} + 1) b_{lU} + (2 b_{\alpha \phi 2} b_{\alpha U}/ b_{\alpha L}  + 1) b_{log1} + 1)/ (b_{\alpha L}^2 b_{log2} /16)  \}, \\
	&b_{\Wcal U 2} \triangleq (2 (b_{\alpha U} + 1) b_{lU} + b_{log1} + 1) / b_{log2} , \\
	&M_5(\varepsilon) = \max\{ M_4(\varepsilon), \inf\nolimits_{M} \{\forall m \ge M, \forall i=1,\dots,k:     b_{\Wcal U 1} N_i^m \varepsilon^{\frac{1}{2}} \ge 4 \} \}.
\end{align*}
Suppose there exists a design $i_2=1,\dots,k$ and $i_2\ne i$ such that $ N_{i_2}^{m'} / N_{i_2}^{m} > 1 + b_{\Wcal U 1} \varepsilon^{\frac{1}{2}}$. Let $i^\dag \triangleq \arg \max_{ i_1 \ne i^* } N_{i_1}^{m'} / N_{i'}^{m} $. Consider the following two collectively exhaustive cases.
\begin{itemize}
	\item Suppose $ N_{i^\dag}^{m'} / N_{i^\dag}^{m} > 1 + b_{\Wcal U 1} \varepsilon^{\frac{1}{2}}$. Since
	\begin{align*}
		( 1- b_{\Wcal U 1} \varepsilon^{\frac{1}{2}}/2 ) (1+b_{\Wcal U 1} \varepsilon^{\frac{1}{2}} ) = 1 +  b_{\Wcal U 1} \varepsilon^{\frac{1}{2}} / 2 -  b_{\Wcal U 1}^2 \varepsilon / 2 \ge 1
	\end{align*}
	for $\varepsilon \le b_{\Wcal U 1}^{-2}$, we have $ N_{i^\dag}^{m} / N_{i^\dag}^{m'}  < (1 + b_{\Wcal U 1} \varepsilon^{1/2})^{-1} \le 1-\frac{1}{2} b_{\Wcal U 1} \varepsilon^{\frac{1}{2}}$.
	
	\item Suppose $ N_{i^\dag}^{m'} / N_{i^\dag}^{m} \le 1 + b_{\Wcal U 1} \varepsilon^{\frac{1}{2}}$. By the assumption, we must have
	\begin{align}\label{ineq:case2_optn}
		N_{i^*}^{m'} / N_{i^*}^m  > 1 + b_{\Wcal U 1} \varepsilon^{\frac{1}{2}}.
	\end{align}
	Notice that $b_{\Wcal U 1} \varepsilon^{\frac{1}{2}} N_{i^*}^m \ge 1$ when $m \ge M_5(\varepsilon)$, which, together with \eqref{ineq:case2_optn}, leads to $N_{i^*}^{m'}  \ge N_{i^*}^m + 1$. Then, there exist iterations between $m+1$ and $m'+1$ in which $i^*$ is sampled. By Lemma \ref{lem:interv_opt},  $i^\dag$ should satisfy $N_{i^\dag}^{m'}/ N_{i^*}^{m'}  > (N_{i^\dag}^m/ N_{i^*}^m)  ( 1- (2 b_{\alpha \phi 2}/ b_{\alpha L})  \varepsilon^{\frac{1}{2}} )$.
	Multiplying both sides by $N_{i^*}^{m'}/N_{i^\dag}^m$ and based on \eqref{ineq:case2_optn}, we obtain
	\begin{align}\label{ineq:exist_lb1}
		&N_{i^\dag}^{m'}/ N_{i^\dag}^m  > (N_{i^*}^{m'} / N_{i^*}^m)  ( 1- 2 b_{\alpha \phi 2}   \varepsilon^{\frac{1}{2}} / b_{\alpha L} ) \ge (1+b_{\Wcal U 1} \varepsilon^{\frac{1}{2}}) ( 1- 2 b_{\alpha \phi 2}   \varepsilon^{\frac{1}{2}} / b_{\alpha L} ).
	\end{align}
	Since $b_{\Wcal U 1} \ge  8 b_{\alpha \phi 2}/ b_{\alpha L} $ and $\varepsilon \le b_{\Wcal U 1}^{-2}$ (which implies $ b_{\Wcal U 1}^2 \varepsilon / 4 \le b_{\Wcal U 1} \varepsilon^{\frac{1}{2}} / 4$), we have
	\begin{align*}
		&(1+b_{\Wcal U 1} \varepsilon^{\frac{1}{2}}) ( 1- 2 b_{\alpha \phi 2}   \varepsilon^{\frac{1}{2}} / b_{\alpha L} )
		\ge (1+b_{\Wcal U 1} \varepsilon^{\frac{1}{2}}) ( 1- b_{\Wcal U 1}   \varepsilon^{\frac{1}{2}} / 4 ) \ge 1 + b_{\Wcal U 1} \varepsilon^{\frac{1}{2}} / 2, \\
		&( 1-b_{\Wcal U 1} \varepsilon^{\frac{1}{2}}/4 ) (1+b_{\Wcal U 1} \varepsilon^{\frac{1}{2}}/2 ) = 1 + b_{\Wcal U 1} \varepsilon^{\frac{1}{2}}/4 - b_{\Wcal U 1}^2 \varepsilon / 8 \ge 1,
	\end{align*}
	which, together with \eqref{ineq:exist_lb1}, yields $N_{i^\dag}^m / N_{i^\dag}^{m'}  < (1 + b_{\Wcal U 1} \varepsilon^{1/2}/2)^{-1} \le 1-b_{\Wcal U 1}  \varepsilon^{\frac{1}{2}} / 4$.
\end{itemize}
Thus, no matter which case happens, we have
\begin{align*}
	N_{i^\dag}^{m'} / N_{i^\dag}^m  \ge 1+b_{\Wcal U 1}  \varepsilon^{\frac{1}{2}} / 2, \quad N_{i^\dag}^m / N_{i^\dag}^{m'}  \le 1- b_{\Wcal U 1}  \varepsilon^{\frac{1}{2}} / 4.
\end{align*}
Since $\frac{b_{\Wcal U 1}}{2} \varepsilon^{\frac{1}{2}} N_{i^\dag}^m  \ge 2$ for $m \ge M_5(\varepsilon)$, we have $N_{i^\dag}^{m'} \ge N_{i^\dag}^{m} + 2$ such that $i^\dag \ne i$ and there exist iterations between $m+1$ and $m'$ where $i^{\dag}$ is sampled. Let $m^{\dag}+1$ denote the last iteration at or before iteration $m'$ where $i^{\dag}$ is sampled. We will need to discuss it case by case.
\begin{itemize}
	\item Suppose $ N_{i^\dag}^{m^\dag} / N_{i^*}^{m^\dag}  -  N_{i^\dag}^{m}  / N_{i^*}^m  \le b_{\Wcal U 2} \varepsilon $. Notice that $N_{i}^{m^\dag} = N_{i}^{m+1} = N_{i}^m+1$ such that
	\begin{align*}
		\frac{ N_{i^\dag}^m }{ N_{i^\dag}^{m^\dag} } \frac{ N_{i}^{m^\dag} }{ N_{i^\dag}^m } = \frac{ N_{i^\dag}^m }{ N_{i^\dag}^{m^\dag} } \frac{ N_{i}^m + 1 }{ N_{i^\dag}^m } \le (1-b_{\Wcal U 1}  \varepsilon^{\frac{1}{2}} /4)  \frac{ N_{i}^m }{ N_{i^\dag}^m } (1+  \varepsilon) \le (1-b_{\Wcal U 1}  \varepsilon^{\frac{1}{2}} /8) \frac{ N_{i}^m }{ N_{i^\dag}^m },
	\end{align*}
	because $(N_{i}^{m})^{-1} \le \varepsilon$ for $m \ge M_4(\varepsilon)$ and $\varepsilon \le \varepsilon_5 \le (b_{\Wcal U 1}/8)^2$. Then
	\begin{align*}
		\frac{ N_{i}^{m^\dag} }{ N_{i^*}^{m^\dag} } - \frac{ N_{i}^m  }{ N_{i^*}^m } =& \frac{ N_{i^\dag}^{m^\dag} }{ N_{i^*}^{m^\dag} } \frac{ N_{i^\dag}^m }{ N_{i^\dag}^{m^\dag} } \frac{ N_{i}^{m^\dag} }{ N_{i^\dag}^m }  - \frac{ N_{i^\dag}^m  }{ N_{i^*}^m } \frac{ N_{i}^m  }{ N_{i^\dag}^m }
		\le \frac{ N_{i^\dag}^{m^\dag} }{ N_{i^*}^{m^\dag} } (1-b_{\Wcal U 1}  \varepsilon^{\frac{1}{2}} / 8)  \frac{ N_{i}^m }{ N_{i^\dag}^{m} }  - \frac{ N_{i^\dag}^{m}  }{ N_{i^*}^m } \frac{ N_{i}^m  }{ N_{i^\dag}^{m} } \\
		=&  (\frac{ N_{i^\dag}^{m^\dag} }{ N_{i^*}^{m^\dag} } -  \frac{ N_{i^\dag}^{m}  }{ N_{i^*}^m }) \frac{ N_{i}^m }{ N_{i^\dag}^{m} } -\frac{b_{\Wcal U 1}}{8}  \varepsilon^{\frac{1}{2}} \frac{ N_{i^\dag}^{m^\dag} }{ N_{i^*}^{m^\dag} } \frac{ N_{i}^m }{ N_{i^\dag}^{m} }   \\
		\le& b_{\alpha U} b_{\Wcal U 2} \varepsilon - \frac{b_{\Wcal U 1} b_{\alpha L}^2}{8} \varepsilon^{\frac{1}{2}}  \\
		\le& - \frac{b_{\Wcal U 1} b_{\alpha L}^2}{16} \varepsilon^{\frac{1}{2}},
	\end{align*}
	where the last inequality holds  $\varepsilon \le (\frac{b_{\Wcal U 1} b_{\alpha L}^2}{ 16 b_{\Wcal U 2} b_{\alpha U} })^2$ such that $b_{\alpha U} b_{\Wcal U 2} \varepsilon \le \frac{b_{\Wcal U 1} b_{\alpha L}^2}{16} \varepsilon^{\frac{1}{2}}$.
	
	Moreover, notice that if $N_{i^*}^{m^\dag} = N_{i^*}^m$, then $N_{i^{\dag}}^{m}/N_{i^*}^m - N_{i^\dag}^{m^\dag}/N_{i^*}^{m^\dag} \le 0$;
	if $N_{i^*}^{m^\dag} > N_{i^*}^m$, then by Lemma \ref{lem:interv_opt}, $N_{i^\dag}^m/ N_{i^*}^m  - N_{i^\dag}^{m'}/ N_{i^*}^{m'}  \le (2 b_{\alpha \phi 2} b_{\alpha U} / b_{\alpha L})  \varepsilon^{\frac{1}{2}}$,
	which leads to that
	\begin{align*}
		\frac{N_{i^\dag}^{m}}{ N_{i^*}^m } - \frac{N_{i^\dag}^{m^\dag}}{ N_{i^*}^{m^\dag} } =& \frac{N_{i^\dag}^{m}}{ N_{i^*}^m } - \frac{N_{i^\dag}^{m'}-1}{ N_{i^*}^{m^\dag} } \le \frac{N_{i^\dag}^{m}}{ N_{i^*}^m } - \frac{N_{i^\dag}^{m'}-1}{ N_{i^*}^{m'} }
		\le \frac{2 b_{\alpha \phi 2} b_{\alpha U}}{ b_{\alpha L} } \varepsilon^{\frac{1}{2}} + \varepsilon.
	\end{align*}
	Then, $N_{i^{\dag}}^{m}/N_{i^*}^m - N_{i^\dag}^{m^\dag}/N_{i^*}^{m^\dag} \le \frac{2 b_{\alpha \phi 2} b_{\alpha U}}{ b_{\alpha L} } \varepsilon^{\frac{1}{2}} + \varepsilon$ no matter $N_{i^*}^{m^\dag} > N_{i^*}^m$ or $N_{i^*}^{m^\dag} = N_{i^*}^m$.
	
	Since $\hat{\Wcal}_{i}^m - \hat{\Wcal}_{i^\dag}^{m} \le 0$ by the definition of iteration $r$, we have by Lemma \ref{lem:vdiff} that
	\begin{align}
		\hat{\Wcal}_{i}^{m^\dag} - \hat{\Wcal}_{i^\dag}^{m^\dag}
		\le& 2 (b_{\alpha U} + 1) b_{lU} \varepsilon + (\hat{\alpha}_{i}^{m^\dag}/\hat{\alpha}_{i^*}^{m^\dag} - \hat{\alpha}_{i}^m/\hat{\alpha}_{i^*}^m) \hat{\Ucal}^{m}_{i}
		+  ( \hat{\alpha}_{i^\dag}^{m}/\hat{\alpha}_{i^*}^m - \hat{\alpha}_{i^\dag}^{m^\dag}/\hat{\alpha}_{i^*}^{m^\dag} ) \hat{\Ucal}^{m}_{i^\dag}  \nonumber \\
		\le&2 (b_{\alpha U} + 1) b_{lU} \varepsilon - \frac{b_{\Wcal U 1} b_{\alpha L}^2}{16}  \varepsilon^{\frac{1}{2}} b_{log2} + (\frac{2 b_{\alpha \phi 2} b_{\alpha U}}{ b_{\alpha L} } \varepsilon^{\frac{1}{2}} + \varepsilon) b_{log1}  \nonumber \\
		\le&  2 (b_{\alpha U} + 1) b_{lU} (\varepsilon - \varepsilon^{\frac{1}{2}})  + (\varepsilon - \varepsilon^{\frac{1}{2}}) b_{log1}  - \varepsilon^{\frac{1}{2}} \label{ineq:exist_ub1} \\
		<& 0,  \nonumber
	\end{align}
	where \eqref{ineq:exist_ub1} holds because $b_{\Wcal U 1} \ge  \frac{2 (b_{\alpha U} + 1) b_{lU} + (2 b_{\alpha \phi 2} b_{\alpha U}/ b_{\alpha L}  + 1) b_{log1} + 1}{ b_{\alpha L}^2 b_{log2} /16 } $.
	
	\item Suppose $ N_{i^\dag}^{m^\dag} / N_{i^*}^{m^\dag}  -  N_{i^\dag}^{m}  / N_{i^*}^m  > b_{\Wcal U 2} \varepsilon $. We have
	\begin{align*}
		\frac{N_{i}^{m^\dag}}{N_{i^*}^{m^\dag}} - \frac{N_{i}^m}{N_{i^*}^m} = \frac{N_{i}^m+1}{N_{i^*}^{m^\dag}} - \frac{N_{i}^m}{N_{i^*}^m} \le \frac{N_{i}^m+1}{N_{i^*}^m} - \frac{N_{i}^m}{N_{i^*}^m} \le \frac{N_{i}^m}{N_{i^*}^m} - \frac{N_{i}^m}{N_{i^*}^m} + \varepsilon = \varepsilon
	\end{align*}
	because $(N_{i^*}^m)^{-1}  \le \varepsilon$ for $ m \ge M_4(\varepsilon)$. We have by Lemma \ref{lem:vdiff} and $b_{\Wcal U 2}$'s definition that
	\begin{align*}
		\hat{\Wcal}_{i}^{m^\dag} - \hat{\Wcal}_{i^\dag}^{m^\dag} \le& 2 (b_{\alpha U} + 1) b_{lU} \varepsilon + (\hat{\alpha}_{i}^{m^\dag}/\hat{\alpha}_{i^*}^{m^\dag} - \hat{\alpha}_{i}^m/\hat{\alpha}_{i^*}^m) \hat{\Ucal}^{m}_{i}   +  ( \hat{\alpha}_{i^\dag}^{m}/\hat{\alpha}_{i^*}^m - \hat{\alpha}_{i^\dag}^{m^\dag}/\hat{\alpha}_{i^*}^{m^\dag} ) \hat{\Ucal}^{m}_{i^\dag} \\
		\le& 2 (b_{\alpha U} + 1) b_{lU} \varepsilon + b_{log1} \varepsilon - b_{\Wcal U 2} \varepsilon b_{log2}   < 0.
	\end{align*}
\end{itemize}

Thus, no matter $ N_{i^\dag}^{m^\dag} / N_{i^*}^{m^\dag}  -  N_{i^\dag}^{m}  / N_{i^*}^m  \le b_{\Wcal U 2} \varepsilon$ holds or not, $\hat{\Wcal}_{i}^{m^\dag} - \hat{\Wcal}_{i^\dag}^{m^\dag} < 0$, which means $i^{m^\dag+1} \ne i^\dag$. This contradicts the definition of $m^\dag+1$ where $i^\dag$ should be sampled. $\square$

\subsection{Proof of Lemma \ref{lem:optnonub}} \label{sec:lem:optnonub_proof}

Let $\varepsilon_6 \triangleq \min\{ \varepsilon_5, (b_{\Ucal}/6)^2, 1/6, (b_{\Wcal}/(2(b_{\alpha U}+1)b_{lU}))^2\}$ where
\begin{align*}
	&b_{\Ucal} \triangleq \max\{ 2 b_{\alpha \phi 2}/b_{\alpha L}, 8 b_{\Wcal} / (b_{\alpha L} b_{log2} ), 8 b_{\alpha \phi 2} b_{log1} / (b_{\alpha L} b_{log2})   \}.
\end{align*}
Let $\tilde{M}_6(\varepsilon) = \inf_{R} \{\forall m \ge R:    m \ge  \varepsilon^{-\frac{1}{2}} / (b_{\Ucal} b_{\alpha L 2}(1-b_{\alpha L 2})) \}$ and $M_6(\varepsilon) = \max \{ \tilde{M}_5(\varepsilon), \tilde{M}_6(\varepsilon) \} $ where $\tilde{M}_5(\varepsilon)$ was defined in the proof of Proposition \ref{proposv}. Since $i^*$ is sampled at iteration $m^*+1$,
\begin{align}\label{ineq:rs_def}
	\sum_{i \ne i^*} \hat{\Ucal}_{i}^{*,m^*}/\hat{\Ucal}_{i}^{m^*} > 1.
\end{align}
Suppose $(\sum_{i \ne i^*} \hat{\alpha}_{i}^{m'})/(\sum_{i \ne i^*} \hat{\alpha}_{i}^{m^*}) \ge 1 +  b_{\Ucal} \varepsilon^{1/2}$ for $m' > m^*$. Then $ \hat{\alpha}_{i^*}^{m'} \le \hat{\alpha}_{i^*}^{m^*} $ and there must exist $i^\dag \ne i^*$ with $\hat{\alpha}_{i^\dag}^{m'}/\hat{\alpha}_{i^\dag}^{m^*} \ge 1 +  b_{\Ucal} \varepsilon^{1/2}$, which yields $(\hat{\alpha}_{i^\dag}^{m'}/\hat{\alpha}_{i^\dag}^{m^*}) (\hat{\alpha}_{i^*}^{m^*}/\hat{\alpha}_{i^*}^{m'}) \ge 1 +  b_{\Ucal} \varepsilon^{1/2}$. That is,
\begin{align}\label{ineq:istar_ddag}
	\hat{\alpha}_{i^\dag}^{m'} / \hat{\alpha}_{i^*}^{m'} \ge (1 +  b_{\Ucal} \varepsilon^{1/2}) \hat{\alpha}_{i^\dag}^{m^*} / \hat{\alpha}_{i^*}^{m^*}.
\end{align}
Since $m' > m^*$, \eqref{ineq:istar_ddag} implies that $N_{i^\dag}^{m'} \ge N_{i^\dag}^{m^*} +  1$, and there exist iterations between $m^*+1$ and $m'$ in which $i^\dag$ is sampled. Let $m^\dag+1$ denote the last iteration before iteration $m'+1$ such that $i^{m^\dag+1} = i^\dag$. Notice that $(N_{i^\dag}^{m'}-1)/N_{i^\dag}^{m'} \ge 1-\varepsilon$ for $m' \ge M_6(\varepsilon)$ such that
\begin{align*}
	N_{i^\dag}^{m^\dag}/N_{i^*}^{m^\dag}  \ge (N_{i^\dag}^{m'}-1)/N_{i^*}^{m'} \ge (1-\varepsilon) N_{i^\dag}^{m'}/N_{i^*}^{m'} \ge (1-\varepsilon) (1 +  b_{\Ucal} \varepsilon^{1/2})  \hat{\alpha}_{i^\dag}^{m^*} / \hat{\alpha}_{i^*}^{m^*} .
\end{align*}
Since $\varepsilon \le \min\{(b_{\Ucal}/4)^2 , 1/4\}$, we have $\frac{b_{\Ucal}}{4}\varepsilon^{\frac{1}{2}} \ge \varepsilon$ and $\frac{b_{\Ucal}}{4} \varepsilon^{\frac{1}{2}} \ge b_{\Ucal} \varepsilon^{\frac{3}{2}}$ such that $(1-\varepsilon) (1 +  b_{\Ucal} \varepsilon^{1/2}) \ge 1 + b_{\Ucal} \varepsilon^{1/2}/2$.
Combining the above equation with $ \hat{\alpha}_{i^\dag}^{m^*}/ \hat{\alpha}_{i^*}^{m^*} \ge b_{\alpha L}$, we have
\begin{align}
	&\frac{\hat{\alpha}_{i^\dag}^{m^\dag}}{\hat{\alpha}_{i^*}^{m^\dag}} - \frac{\hat{\alpha}_{i^\dag}^{m^*}}{\hat{\alpha}_{i^*}^{m^*}} = \frac{N_{i^\dag}^{m^\dag}}{N_{i^*}^{^{m^\dag}}} - \frac{\hat{\alpha}_{i^\dag}^{m^*}}{\hat{\alpha}_{i^*}^{m^*}} \ge (1 +  b_{\Ucal} \varepsilon^{\frac{1}{2}}/2 ) \frac{ \hat{\alpha}_{i^\dag}^{m^*} }{ \hat{\alpha}_{i^*}^{m^*} } - \frac{\hat{\alpha}_{i^\dag}^{m^*}}{\hat{\alpha}_{i^*}^{m^*}} \ge  \frac{b_{\Ucal} b_{\alpha L}}{2} \varepsilon^{\frac{1}{2}}.  \label{ineq:rddag_dif}
\end{align}
Since $b_{\Ucal} b_{\alpha L}/2 \ge b_{\alpha \phi 2}$, we have by \eqref{ineq:rddag_dif} that $\hat{\alpha}_{i^\dag}^{m^\dag}/\hat{\alpha}_{i^*}^{m^\dag} - \hat{\alpha}_{i^\dag}^{m^*}/\hat{\alpha}_{i^*}^{m^*} \ge b_{\alpha \phi 2} \varepsilon^{\frac{1}{2}}$. By Lemma \ref{lem:upart},
\begin{align}
	&\hat{\Ucal}^{*,m^\dag}_{i^\dag}/\hat{\Ucal}^{m^\dag}_{i^\dag}
	\ge ( 1 + b_{up} \varepsilon^{\frac{1}{2}} )  \hat{\Ucal}^{*,m^*}_{i^\dag}/\hat{\Ucal}^{m^*}_{i^\dag}.   \label{ineq:U_lb1}
\end{align}

On the other hand, for any $i' \ne i^\dag$ and $i^*$, we have by Proposition \ref{proposv} that $ \hat{\Wcal}_{i'}^{m^*} - \hat{\Wcal}_{i^\dag}^{m^*} \le b_{\Wcal} \varepsilon^{\frac{1}{2}}$. Then, by Lemma \ref{lem:vdiff}, we have that
\begin{align*}
	&\hat{\Wcal}_{i'}^{m^\dag} - \hat{\Wcal}_{i^\dag}^{m^\dag}  \\
	\le& 2 (b_{\alpha U} + 1) b_{lU} \varepsilon + (\hat{\alpha}_{i'}^{m^\dag}/\hat{\alpha}_{i^*}^{m^\dag} - \hat{\alpha}_{i'}^{m^*}/\hat{\alpha}_{i^*}^{m^*}) \hat{\Ucal}^{m^*}_{i'} + \hat{\Wcal}_{i'}^{m^*} - \hat{\Wcal}_{i^\dag}^{m^*} +  ( \hat{\alpha}_{i^\dag}^{m^*}/\hat{\alpha}_{i^*}^{m^*} - \hat{\alpha}_{i^\dag}^{m^\dag}/\hat{\alpha}_{i^*}^{m^\dag} ) \hat{\Ucal}^{m^\dag}_{i^\dag}   \\
	\le& 2(b_{\alpha U} + 1) b_{lU} \varepsilon + (\hat{\alpha}_{i'}^{m^\dag}/\hat{\alpha}_{i^*}^{m^\dag} - \hat{\alpha}_{i'}^{m^*}/\hat{\alpha}_{i^*}^{m^*}) \hat{\Ucal}^{m^*}_{i'} + b_{\Wcal} \varepsilon^{\frac{1}{2}}  - \frac{b_{\Ucal} b_{\alpha L}}{2} \varepsilon^{\frac{1}{2}} b_{log2}  \\
	\le& (\hat{\alpha}_{i'}^{m^\dag}/\hat{\alpha}_{i^*}^{m^\dag} - \hat{\alpha}_{i'}^{m^*}/\hat{\alpha}_{i^*}^{m^*}) \hat{\Ucal}^{m^*}_{i'} - \frac{b_{\Ucal} b_{\alpha L}}{4} \varepsilon^{\frac{1}{2}} b_{log2},
\end{align*}
where the last inequality holds because  $\varepsilon \le (\frac{b_{\Wcal}}{2(b_{\alpha U}+1)b_{lU}})^2$ and $b_{\Ucal} \ge 8 b_{\Wcal} / (b_{\alpha L} b_{log2})$.
Since $\hat{\Wcal}_{i'}^{m^\dag} - \hat{\Wcal}_{i^\dag}^{m^\dag} \ge 0$ by the definition of $m^\dag$, the above inequality yields
\begin{align*}
	\hat{\alpha}_{i'}^{m^\dag}/\hat{\alpha}_{i^*}^{m^\dag} - \hat{\alpha}_{i'}^{m^*}/\hat{\alpha}_{i^*}^{m^*}  \ge \frac{b_{\Ucal} b_{\alpha L}}{4 \hat{\Ucal}^{m^*}_{i'}}  b_{log2} \varepsilon^{\frac{1}{2}} \ge \frac{b_{\Ucal} b_{\alpha L} b_{log2}}{4 b_{log1}}   \varepsilon^{\frac{1}{2}}  \ge b_{\alpha \phi 2} \varepsilon^{\frac{1}{2}},
\end{align*}
where the last inequality holds by $b_{\Ucal} \ge 4 b_{log1} b_{\alpha \phi 2} / ( b_{\alpha L}  b_{log2} )$.
Again by Lemma \ref{lem:upart}, we have
\begin{align}
	&\hat{\Ucal}^{*,m^\dag}_{i'}/\hat{\Ucal}^{m^\dag}_{i'}
	\ge ( 1 + b_{up} \varepsilon^{\frac{1}{2}} )  \hat{\Ucal}^{*,m^*}_{i'}/\hat{\Ucal}^{m^*}_{i'}.  \label{ineq:U_lb2}
\end{align}
Combining \eqref{ineq:rs_def}, \eqref{ineq:U_lb1} and \eqref{ineq:U_lb2}, we have $\sum_{i \ne i^*} \hat{\Ucal}^{*,m^\dag}_{i}/\hat{\Ucal}^{m^\dag}_{i} > 1$.
This means $i^{m^\dag+1} = i^*$, which contradicts the definition of iteration $m^\dag+1$. $\square$

\subsection{Proof of Proposition \ref{prop:optnonconv}}\label{sec:prop:optnonconv_proof}

Let $\tilde{M}_7(\varepsilon) \ge  M_6(\varepsilon)$ denote the first random time after $M_6(\varepsilon)$ where either $i^{\tilde{M}_7(\varepsilon)+1} = i^*$ and $i^{\tilde{M}_7(\varepsilon)+2} \ne i^*$ or $i^{\tilde{M}_7(\varepsilon)+1} \ne i^*$ and $i^{\tilde{M}_7(\varepsilon)+2} = i^*$ holds. That is, one of iterations $\tilde{M}_7(\varepsilon)+1$ and $\tilde{M}_7(\varepsilon)+2$ samples the best design $i^*$ and the other iteration samples a non-best design. Let $M_7(\varepsilon) \triangleq \tilde{M}_7(\varepsilon) +3$. For notational simplicity, let $m^*$ be the iteration among $\tilde{M}_7(\varepsilon)$ and $\tilde{M}_7(\varepsilon)+1$ such that $i^{m^*+1} = i^*$ and $m^\vartriangle$ be the other iteration among $\tilde{M}_7(\varepsilon)$ and $\tilde{M}_7(\varepsilon)+1$. For $\varepsilon \le 1$,
\begin{align*}
	\sum_{i \ne i^*} \hat{\alpha}_{i}^{m^*} = \sum_{i \ne i^*} \frac{ N_{i}^{m^*} }{N^{m^*}} \le \sum_{i \ne i^*} \frac{ N_{i}^{m^{\vartriangle}} + 1}{N^{m^{\vartriangle}}-1} \le (1 + \varepsilon)^2 \sum_{i \ne i^*} \hat{\alpha}_{i}^{m^{\vartriangle}} \le (1 + 3 \varepsilon) \sum_{i \ne i^*} \hat{\alpha}_{i}^{m^{\vartriangle}}
\end{align*}
because $N_{i}^{m^{\vartriangle}}+1 \le (1+\varepsilon) N_{i}^{m^{\vartriangle}}
$ and $1 / (N^{m^{\vartriangle}}-1) \le (1+\varepsilon) /N^{m^{\vartriangle}}$ for $m \ge M_7(\varepsilon) $.
Meanwhile, we have by Lemma \ref{lem:optnonub} that
$\sum_{i \ne i^*} \hat{\alpha}_{i}^{m'} < (1 +  b_{\Ucal} \varepsilon^{1/2}) \sum_{i \ne i^*} \hat{\alpha}_{i}^{m^*}$ for all $m' \ge m^*$. Then
\begin{align*}
	\sum_{i \ne i^*} \hat{\alpha}_{i}^{m'} \le (1 +  b_{\Ucal} \varepsilon^{1/2}) \sum_{i \ne i^*} \hat{\alpha}_{i}^{m^*} \le (1 +  b_{\Ucal} \varepsilon^{1/2}) (1+3\varepsilon) \sum_{i \ne i^*} \hat{\alpha}_{i}^{m^{\vartriangle}} \le (1 +  2b_{\Ucal} \varepsilon^{1/2}) \sum_{i \ne i^*} \hat{\alpha}_{i}^{m^{\vartriangle}}
\end{align*}
where the last inequality holds because $\varepsilon \le (b_{\Ucal}/6)^2$ such that $3\varepsilon \le \frac{b_{\Ucal}}{2}\varepsilon^{\frac{1}{2}}$ and $\varepsilon \le \frac{1}{6}$ such that $ 3b_{\Ucal} \varepsilon^{\frac{3}{2}} \le \frac{b_{\Ucal}}{2} \varepsilon^{\frac{1}{2}}$.
By noticing that $\sum_{i \ne i^*} \hat{\alpha}_{i}^{m^{\vartriangle}} \le 1$,
\begin{align}\label{ineq:nonbest_ub}
	\sum_{i \ne i^*} \hat{\alpha}_{i}^{m'} \le (1 +  2b_{\Ucal} \varepsilon^{1/2}) \sum_{i \ne i^*} \hat{\alpha}_{i}^{m^{\vartriangle}}
	\le \sum_{i \ne i^*} \hat{\alpha}_{i}^{m^{\vartriangle}} + 2b_{\Ucal}  \varepsilon^{1/2}.
\end{align}
Meanwhile, by Lemma \ref{lem:optnonlb} and for $m' \ge m^\vartriangle$,
$(\sum_{i \ne i^*} \hat{\alpha}_{i}^{m'}) / (\sum_{i \ne i^*} \hat{\alpha}_{i}^{m^{\vartriangle}}) > 1 -  b_{\Ucal} \varepsilon^{1/2}
$,
implying
\begin{align}\label{ineq:nonbest_lb}
	\sum_{i \ne i^*} \hat{\alpha}_{i}^{m'} > \sum_{i \ne i^*} \hat{\alpha}_{i}^{m^{\vartriangle}} -  b_{\Ucal} \varepsilon^{1/2} \sum_{i \ne i^*} \hat{\alpha}_{i}^{m^{\vartriangle}} \ge \sum_{i \ne i^*} \hat{\alpha}_{i}^{m^{\vartriangle}} -  b_{\Ucal} \varepsilon^{1/2}.
\end{align}
For $m',m'' \ge  M_7(\varepsilon) > m^{\vartriangle}$, we have by \eqref{ineq:nonbest_ub} and \eqref{ineq:nonbest_lb} that
\begin{align*}
	| \sum_{i \ne i^*} \hat{\alpha}_{i}^{m'} - \sum_{i \ne i^*} \hat{\alpha}_{i}^{m''} | \le | \sum_{i \ne i^*} \hat{\alpha}_{i}^{m'} - \sum_{i \ne i^*} \hat{\alpha}_{i}^{m^{\vartriangle}} | + | \sum_{i \ne i^*} \hat{\alpha}_{i}^{m^{\vartriangle}} - \sum_{i \ne i^*} \hat{\alpha}_{i}^{m''} | \le 4b_{\Ucal} \varepsilon^{1/2}. \ \square
\end{align*}

\subsection{Proof of Proposition \ref{prop:nonconv}}\label{sec:prop:nonconv_proof}

Let $\varepsilon_7 \triangleq \min \{\varepsilon_6, (b_{\alpha L 2} / (16 b_{\Ucal}))^{2}\} $.
If $k=2$, the conclusion follows immediately by Proposition \ref{prop:optnonconv} for $b_{\alpha U 3} \triangleq 2b_{\Ucal}$. Suppose $k > 2$ and let $b_{\alpha U 3} \triangleq \max\{ 16 b_{\Ucal} b_{\alpha U 2} /b_{\alpha L 2},  4 b_{\Ucal} ((k-2)  b_{\alpha U} + 1), 8 b_{\Wcal} b_{\alpha U 2} / (b_{\alpha L } b_{log2}) \}$. Since $b_{\alpha U 3} \ge \frac{ 16 b_{\Ucal} b_{\alpha U 2} }{b_{\alpha L 2}}$ such that $\frac{ b_{\alpha U 3} }{ 4 b_{\alpha U 2} } - \frac{4 b_{\Ucal}}{b_{\alpha L 2}} \ge 0$ and $\varepsilon \le \varepsilon_7 \le (\frac{b_{\alpha L 2} }{16 b_{\Ucal}})^{2}$ such that $\frac{ b_{\alpha U 3} }{ 4 b_{\alpha U 2} } \varepsilon^{1/2} - \frac{4 b_{\Ucal}}{b_{\alpha L 2}} \frac{ b_{\alpha U 3} }{ b_{\alpha U 2} } \varepsilon \ge 0$, we have
\begin{align}\label{ineq:prop_i1_lb00}
	(1+\frac{ b_{\alpha U 3} }{ b_{\alpha U 2} }\varepsilon^{\frac{1}{2}}) (1-\frac{4 b_{\Ucal}}{b_{\alpha L 2}}\varepsilon^{\frac{1}{2}})
	=  1+\frac{ b_{\alpha U 3} }{ b_{\alpha U 2} }\varepsilon^{\frac{1}{2}}-\frac{4 b_{\Ucal}}{b_{\alpha L 2}}\varepsilon^{\frac{1}{2}}-\frac{4 b_{\Ucal}}{b_{\alpha L 2}} \frac{ b_{\alpha U 3} }{ b_{\alpha U 2} } \varepsilon \ge 1+\frac{ b_{\alpha U 3} }{2 b_{\alpha U 2} } \varepsilon^{\frac{1}{2}}.
\end{align}

For notational simplicity, let $\bar{m} \triangleq M_7(\varepsilon)$. Suppose there exist an iteration $m > \bar{m}$ and design $i \ne i^*$ such that $\hat{\alpha}_{i}^{m} > \hat{\alpha}_{i}^{\bar{m}} + b_{\alpha U 3} \varepsilon^{1/2}$. Notice that $\hat{\alpha}_{i^*}^{m} \le \hat{\alpha}_{i^*}^{\bar{m}} + 4 b_{\Ucal} \varepsilon^{1/2}$ by Proposition \ref{prop:optnonconv}. Then
\begin{align}
	\frac{\hat{\alpha}_{i}^{m}}{\hat{\alpha}_{i^*}^{m}} \ge& \frac{ \hat{\alpha}_{i}^{\bar{m}} + b_{\alpha U 3} \varepsilon^{1/2} }{ \hat{\alpha}_{i^*}^{\bar{m}} + 4 b_{\Ucal} \varepsilon^{1/2} } = \frac{ \hat{\alpha}_{i}^{\bar{m}} + b_{\alpha U 3} \varepsilon^{1/2} }{ \hat{\alpha}_{i}^{\bar{m}} } \frac{ \hat{\alpha}_{i}^{\bar{m}} }{ \hat{\alpha}_{i^*}^{\bar{m}}  } \frac{ \hat{\alpha}_{i^*}^{\bar{m}} }{ \hat{\alpha}_{i^*}^{\bar{m}} + 4 b_{\Ucal} \varepsilon^{1/2} } \nonumber \\
	\ge& (1+\frac{ b_{\alpha U 3} }{ b_{\alpha U 2} }\varepsilon^{1/2}) \frac{ \hat{\alpha}_{i}^{\bar{m}} }{ \hat{\alpha}_{i^*}^{\bar{m}}  } (1-\frac{4 b_{\Ucal}\varepsilon^{1/2}}{b_{\alpha L 2}} )    \label{ineq:prop_i1_lb0} \\
	\ge& \frac{ \hat{\alpha}_{i}^{\bar{m}} }{ \hat{\alpha}_{i^*}^{\bar{m}}  } ( 1+\frac{ b_{\alpha U 3} }{2 b_{\alpha U 2} }\varepsilon^{1/2} ), \label{ineq:prop_i1_lb}
\end{align}
where \eqref{ineq:prop_i1_lb0} holds because
$(1-4 b_{\Ucal}\varepsilon^{1/2} /b_{\alpha L 2}) (1+4 b_{\Ucal}\varepsilon^{1/2} /b_{\alpha L 2}) \le 1$ and \eqref{ineq:prop_i1_lb} holds by \eqref{ineq:prop_i1_lb00}.
Meanwhile, since $\sum_{i' \ne i^*} \hat{\alpha}_{i'}^{m} - \sum_{i \ne i^*} \hat{\alpha}_{i'}^{\bar{m}} <   4 b_{\Ucal} \varepsilon^{1/2}$ by Proposition \ref{prop:optnonconv}, we have
\begin{align*}
	\sum_{i' \ne i^*,i} \hat{\alpha}_{i'}^{m} - \sum_{i' \ne i^*,i} \hat{\alpha}_{i'}^{\bar{m}} <  4 b_{\Ucal} \varepsilon^{1/2} - ( \hat{\alpha}_{i}^{m} - \hat{\alpha}_{i}^{\bar{m}} ) \le (4 b_{\Ucal}  - b_{\alpha U 3}) \varepsilon^{1/2}.
\end{align*}
There must exist a non-best design $j \ne i$ such that $\hat{\alpha}_{j}^{m} \le \hat{\alpha}_{j}^{\bar{m}} - (b_{\alpha U 3} - 4 b_{\Ucal}) \varepsilon^{1/2} / (k-2) $, which can be further upper bounded as $\hat{\alpha}_{j}^{m} \le \hat{\alpha}_{j}^{\bar{m}} - 4 b_{\Ucal} b_{\alpha U}  \varepsilon^{1/2} $ because $b_{\alpha U 3} \ge 4 (k-2) b_{\Ucal} b_{\alpha U} + 4 b_{\Ucal}$. Moreover, by Proposition \ref{prop:optnonconv}, $\hat{\alpha}_{i^*}^{m} \ge \hat{\alpha}_{i^*}^{\bar{m}} - 4 b_{\Ucal} \varepsilon^{1/2}$. Then
\begin{align*}
	\frac{\hat{\alpha}_{j}^{m}}{\hat{\alpha}_{i^*}^{m}} \le \frac{\hat{\alpha}_{j}^{\bar{m}} - 4 b_{\Ucal} b_{\alpha U}  \varepsilon^{1/2}}{ \hat{\alpha}_{i^*}^{\bar{m}} - 4 b_{\Ucal} \varepsilon^{1/2} } \le \frac{ \hat{\alpha}_{j}^{\bar{m}} - 4 b_{\Ucal} \varepsilon^{1/2} \hat{\alpha}_{j}^{\bar{m}}/\hat{\alpha}_{i^*}^{\bar{m}}  }{ \hat{\alpha}_{i^*}^{\bar{m}} - 4 b_{\Ucal} \varepsilon^{1/2} } = \frac{ \hat{\alpha}_{j}^{\bar{m}} }{ \hat{\alpha}_{i^*}^{\bar{m}}  }.
\end{align*}
By Lemma \ref{lem:vdiff} and \eqref{ineq:prop_i1_lb},
\begin{align*}
	\hat{\Wcal}_{j}^{m} - \hat{\Wcal}_{i}^{m}  \nonumber
	\le& 2 (b_{\alpha U} + 1) b_{lU} \varepsilon + \hat{\Wcal}_{j}^{\bar{m}} - \hat{\Wcal}_{i}^{\bar{m}} + (\hat{\alpha}_{j}^{m}/\hat{\alpha}_{i^*}^{m} - \hat{\alpha}_{j}^{\bar{m}}/\hat{\alpha}_{i^*}^{\bar{m}}) \hat{\Ucal}^{\bar{m}}_{j}
	+  ( \hat{\alpha}_{i}^{\bar{m}}/\hat{\alpha}_{i^*}^{\bar{m}} - \hat{\alpha}_{i}^{m}/\hat{\alpha}_{i^*}^{m} ) \hat{\Ucal}^{m}_{i}  \nonumber \\
	\le& 2 (b_{\alpha U} + 1) b_{lU} \varepsilon + b_{\Wcal} \varepsilon^{\frac{1}{2}} -  \frac{ \hat{\alpha}_{i}^{\bar{m}} }{ \hat{\alpha}_{i^*}^{\bar{m}}  } \frac{ b_{\alpha U 3} }{2 b_{\alpha U 2} }\varepsilon^{1/2}  b_{log2}  	\\
	\le& -2b_{\Wcal} \varepsilon^{\frac{1}{2}},
\end{align*}
where the last inequality holds because $\varepsilon \le \varepsilon_6 \le (\frac{b_{\Wcal}}{2(b_{\alpha U}+1)b_{lU}})^2$ and $b_{\alpha U 3} \ge 8 b_{\Wcal} b_{\alpha U 2} / (b_{\alpha L } b_{log2}) $. However, this result contradicts Proposition \ref{proposv}. Thus, $\hat{\alpha}_{i}^{m} \le \hat{\alpha}_{i}^{\bar{m}} + b_{\alpha U 3} \varepsilon^{1/2}$ for any $m \ge \bar{m}$ and $i \ne i^*$.
Similarly, we can show $\hat{\alpha}_{i}^{m} \ge \hat{\alpha}_{i}^{\bar{m}} - b_{\alpha U 3} \varepsilon^{1/2}$ for any $m \ge \bar{m}$ and $i \ne i^*$. The proposition can be proved by further noticing $|\hat{\alpha}_{i}^{m'} - \hat{\alpha}_{i}^{m''}| \le |\hat{\alpha}_{i}^{m'} - \hat{\alpha}_{i}^{\bar{m}}| + |\hat{\alpha}_{i}^{m''} - \hat{\alpha}_{i}^{\bar{m}}| $. $\square$

% Acknowledgments here
%\ACKNOWLEDGMENT{The authors gratefully acknowledge the existence of
%the Journal of Irreproducible Results and the support of the Society
%for the Preservation of Inane Research.}

% References here (outcomment the appropriate case)

% CASE 1: BiBTeX used to constantly update the references
%   (while the paper is being written).
%\bibliographystyle{informs2014} % outcomment this and next line in Case 1
%\bibliography{<your bib file(s)>} % if more than one, comma separated

% CASE 2: BiBTeX used to generate mypaper.bbl (to be further fine tuned)
%\input{mypaper.bbl} % outcomment this line in Case 2

%If you don't use BiBTex, you can manually itemize references as shown below.

%\bibliographystyle{nonumber}

%\newpage
%\appendix
%	
%\setcounter{page}{1}

%%%%%%%%%%%%%%%%%
\end{document}